\documentclass{amsart}
\usepackage{amsfonts}
\usepackage{amsmath}
\usepackage{amssymb}
\usepackage{amsthm}
\usepackage{eucal}
\usepackage{graphicx}
\usepackage{graphicx,color}
\usepackage[usenames,dvipsnames,svgnames,table]{xcolor}

\usepackage{hyperref}
\hypersetup{colorlinks,
            filecolor=black,
            linkcolor=MidnightBlue,
            citecolor=NavyBlue,
            urlcolor=RoyalBlue,
            bookmarksopen=true}

\theoremstyle{plain}
\newtheorem{theorem}{Theorem}[section]

\newtheorem{lemma}[theorem]{Lemma}
\newtheorem{proposition}[theorem]{Proposition}
\newtheorem{corollary}[theorem]{Corollary}

\newtheorem*{main}{Main Theorem}

\newtheorem{remark}[theorem]{Remark}

\numberwithin{equation}{section}

\newcommand{\dt}{\partial_\tau}
\newcommand{\Dt}{\frac{d}{d\tau}}
\newcommand{\dy}{\partial_y}
\newcommand{\dz}{\partial_\theta}
\newcommand{\ls}{\lesssim}
\newcommand{\lan}{\langle}
\newcommand{\ran}{\rangle}
\newcommand{\ve}{\varepsilon}
\newcommand{\vp}{\varphi}

\newcommand{\Dy}{\mathrm{d}y}
\newcommand{\Dz}{\mathrm{d}\theta}
\newcommand{\es}{\simeq}
\newcommand{\lp}{\langle}
\newcommand{\rp}{\rangle_\sigma}

\newcommand{\V}[2]{
    \ifnum0=#1  
        \ifnum1=#2
            (v_{0,1,2})   
        \else
            (v_{0,#2,#2}) 
        \fi
    \else   
        \ifnum0=#2  
            \ifnum1=#1
                (v_{1,0,1})   
            \else
                (v_{#1,0,0})  
            \fi
        \else   
            (v_{#1,#2,#2})
        \fi
    \fi}

\begin{document}
\title[Neckpinch dynamics for asymmetric surfaces]
{Neckpinch dynamics for asymmetric surfaces\\ evolving by mean curvature flow}

\author{Zhou Gang}
\address[Zhou Gang]{California Institute of Technology}
\email{gzhou@caltech.edu}

\author{Dan Knopf} \address[Dan Knopf]{University of Texas at Austin}
\email{danknopf@math.utexas.edu}
\urladdr{http://www.ma.utexas.edu/users/danknopf/}

\author{Israel Michael Sigal}
\address[Israel Michael Sigal]{University of Toronto}
\email{im.sigal@utoronto.ca}
\urladdr{http://www.math.toronto.edu/sigal/}

\thanks{ZG thanks NSF for supports in DMS-1308985.
	DK thanks NSF for support in DMS-0545984.
	IMS thanks NSERC for support in NA7901.}

\begin{abstract}
We study noncompact surfaces evolving by mean curvature flow (\textsc{mcf}). For an open set of initial
data that are $C^3$-close to round, but without assuming rotational symmetry or positive mean curvature,
we show that \textsc{mcf} solutions become singular in finite time by forming neckpinches, and we
obtain detailed asymptotics of that singularity formation. Our results show in a precise way that
\textsc{mcf} solutions become asymptotically rotationally symmetric near a neckpinch singularity.

\end{abstract}
\maketitle

\setcounter{tocdepth}{2}
\tableofcontents

\section{Introduction}

In this paper, we study the motion of cylindrical surfaces under mean
curvature flow (\textsc{mcf}). Specifically, we investigate the asymptotics of
finite-time singularity formation for such surfaces, without assuming rotational
symmetry or positive mean curvature.

There is a rich literature on \textsc{mcf} singularity formation, one far too vast
to be fully acknowledged here. Neckpinch singularities for rotationally symmetric
$2$-dimensional surfaces were first observed by Huisken \cite{Huisken90}; this
was generalized to higher dimensions by Simon \cite{Simon90}. Asymptotic
properties of a solution approaching a singularity were studied by
Angenent--Vel\'{a}zquez \cite{AV97}, who obtained rigorous asymptotics for
nongeneric Type-II ``degenerate neckpinches,'' and derived formal matched
asymptotics for (conjecturally generic) Type-I singularities. Rigorous asymptotics
for the Type-I case were recently obtained by two of the authors \cite{GS09}.
The corresponding result for Type-I Ricci flow neckpinches was obtained by Angenent
and another of the authors \cite{AK07}. For the Type-II Ricci flow case, see
\cite{AIK11}.\footnote{Note that what we call the ``inner region'' in this
paper is called the ``intermediate region'' in \cite{AIK11} and \cite{AK07}.}
These asymptotic analyses all use the hypothesis of rotational symmetry in
essential ways.

A different approach is the surgery program
of Huisken and Sinestrari \cite{HS09}, which shows that singularities of $2$-convex
immersions $\mathcal{M}^n\subset\mathbb{R}^{n+1}$ with $n\geq3$ are either
close to round $\mathbb{S}^n$ components or close to ``necklike''
$\mathbb{S}^{n-1}\times I$ components; this program does not involve asymptotics
or require symmetry hypothesis. Yet another approach studies weak solutions that
extend past the singularity time. See existence results of Brakke \cite{Brakke78},
Evans--Spruck \cite{ES91}, and Chen--Giga--Goto \cite{CGG91}, as well as recent
classification results by Colding--Minicozzi \cite{CM09}.

In this paper, we consider the evolution of non-symmetric, noncompact surfaces embedded
in $\mathbb{R}^3$. To analyze their singularity formation, we study rescaled solutions
defined with respect to adaptive blowup variables.
We remove the hypothesis of rotational symmetry and instead consider initial data in a
Sobolev space $H^5 (\mathbb{S}^1\times\mathbb{R};\Dz\,\sigma\,\Dy)$
(where the weighted measure $\sigma\sim|y|^{-\frac{6}{5}}$ as $|y|\rightarrow\infty$)
which are $C^3$-close to round necks and possess weaker symmetries. (These symmetries
are retained to reduce technicalities but are not essential to our approach.) The
assumptions are made precise in Section~\ref{Basic}. Because we do not allow fully asymmetric
initial data, our results do not establish uniqueness of the limiting cylinder for general
\textsc{mcf} neckpinch singularities (a conjecture that we learned of from Klaus Ecker).
However, our results do provide rigorous evidence in favor of the heuristic expectation
that rotational symmetry is stable in a suitable (quasi-isometric) sense. In light of the
result \cite{CM09} of Colding--Minicozzi that shrinking spheres and cylinders are the only
generic \textsc{mcf} singularity models, our work suggests that the rotationally-symmetric
neckpinch is in fact a ``universal'' singularity profile.

Our analysis consists of two largely independent parts, both of which are organized
as bootstrap machines. The first machine takes as its input certain (weak) estimates
that follow from our assumptions on the initial data; by parabolic regularization
for quasilinear equations, we may assume that these hold for a sufficiently short time
interval. These estimates are detailed in Section~\ref{FirstBootstrap}. The output
of the machine consists of improved \emph{a priori} estimates for the same time interval,
which may then be propagated forward in time.
The machine is constructed in Sections~\ref{FirstBootstrap}--\ref{ProveLast}. Its
construction employs Lyapunov functionals and Sobolev embedding arguments near the
center of the neck, where the most critical analysis is needed, together with
maximum-principle arguments away from the developing singularity.

The second bootstrap machine takes for its input similar (weak) conditions detailed
in Section~~\ref{SecondBootstrap}, along with the improved \emph{a priori} estimates
output by the first machine. It further improves those estimates by showing that for
correct choices of adaptive scaling parameters, it is possible to decompose a solution
into an asymptotically dominant profile and a far smaller ``remainder'' term. The
methods employed here are close to those used in \cite{GS09}, and are collected in
Sections~\ref{SecondBootstrap}--\ref{FullStop}.

Connecting the two machines yields a (non-circular) sequence of arguments that establishes
finite-time extinction of the unrescaled solution, its singular collapse, the nature of the singular
set, and the solution's asymptotic behavior in a space-time neighborhood of the singularity. Thus when
combined, these two bootstrap arguments imply our main theorem. We prove this in Section~\ref{RepeatMainProof}.
\smallskip

As we recall in Section~\ref{Basic}, \textsc{mcf} of a normal graph over a
$2$-dimensional cylinder is determined up to tangential diffeomorphisms
by a radius function $u(x,\theta,t)>0$. 
The assumptions of the theorem are satisfied by any smooth reflection-symmetric
and $\pi$-periodic initial surface $u_0(x,\theta)=u(x,\theta,0)$ that is
a sufficiently small perturbation of the surfaces studied in \cite{GS09}.

\begin{main}
Let $u(x,\theta,t)$ be a solution of \textsc{mcf} whose initial surface satisfies the
Main Assumptions stated in Section~\ref{Basic} for sufficiently small $0<b_0,c_0\ll1$.
Then there exists a time $T<\infty$ such that $u(x,\theta,t)$ develops a
neckpinch singularity at time $T$, with $u(0,\cdot,T)=0$.
There exist functions $\lambda(t)$, $b(t)$, $c(t)$, and $\phi(x,\theta,t)$ such that
with $y(x,t):=\lambda(t)^{-1}x$, one has
\[\frac{u(x,\theta,t)}{\lambda(t)}
=\sqrt{\frac{2+b(t)y^2}{c(t)}}+\phi(y,\theta,t),\]
where as $t\nearrow T$,
\begin{align*}
\lambda(t)  &=[1+o(1)]\lambda_0\sqrt{T-t},\\
b(t)    &=\left[1+\mathcal{O}((-\log(T-t))^{-\frac{1}{2}})\right](-\log(T-t))^{-1},\\
c(t)    &=1+\left[1+\mathcal{O}((-\log(T-t))^{-1})\right](-\log(T-t))^{-1}.
\end{align*}
The solution is asymptotically rotationally symmetric near the neckpinch singularity
in the precise sense that, with $v(y,\theta,\tau):=\lambda(t)^{-1}u(x,\theta,t)$, the estimates
\[v^{-2}|\dz\phi|+v^{-1}|\dy\dz\phi|+v^{-2}|\dz^2\phi|+v^{-1}|\dy^2\dz\phi|
    +v^{-2}|\dy\dz^2\phi|+v^{-3}|\dz^3\phi|=\mathcal{O}(b(t)^{\frac{33}{20}})\]
and
\[\frac{|\phi(y,\theta,t)|}{(1+y^2)^{\frac{3}{2}}}=\mathcal{O}(b(t)^{\frac{8}{5}})
    \quad\text{and}\quad
\frac{|\phi(y,\theta,t)|}{(1+y^2)^{\frac{11}{20}}}=\mathcal{O}(b(t)^{\frac{13}{20}})\]
all hold uniformly as $t \nearrow T$.
\end{main}

\section{Basic evolution equations}\label{Basic}
In this paper, we study the evolution of graphs over a cylinder
$\mathbb{S}^1\times\mathbb{R}$ embedded in $\mathbb{R}^3$.
In coordinates $(x,y,z)$ for $\mathbb{R}^3$, we take as an initial datum a surface $M_0$
around the $x$-axis, given by a map $\sqrt{y^2+z^2}=u_0 (x,\theta)$, where $\theta$ denotes
the angle from the ray $y>0$ in the $(y,z)$-plane. Then for as long as the flow remains
a graph, all $M_{t}$ are given by $\sqrt{y^2+z^2}=u(x,\theta,t)$. It follows from
equations~\eqref{MCF-general} and \eqref{Q-general}, derived in Appendix~\ref{Graphs},
that $u$ evolves by
\begin{equation}\label{MCF-u}
\partial_t u =
\frac{[1+(\frac{\dz u}{u})^2]\partial_x^2 u
      +\frac{1+(\partial_x u)^2}{u^2}\dz^2 u
      -2\frac{(\partial_x u)(\dz u)^2}{u^3} \partial_x \dz u
      -\frac{(\dz u)^2}{u^3}}
    {1+(\partial_x u)^2 + (\frac{\dz u}{u})^2}
-\frac{1}{u}
 \end{equation}
with initial condition $u(x,\theta,0)=u_0(x,\theta)$.

Analysis of rotationally symmetric neckpinch formation \cite{GS09} leads one to
expect that solutions of \eqref{MCF-general} will become singular in finite
time, resembling spatially homogeneous \textsc{ode} solutions $\sqrt{2(T-t)}$
in an suitable space-time neighborhood of the developing singularity.

Accordingly, we apply adaptive rescaling, transforming the original space-time variables
$x$ and $t$ into rescaled blowup variables
\begin{equation} \label{Define-y}
y(x,t):=\lambda^{-1}(t) [x-x_{0}(t)]
\end{equation}
and
\begin{equation} \label{Define-tau}
\tau(t):=\int_0^t \lambda^{-2}(s)\,\mathrm{d}s,
\end{equation}
respectively, where $x_0(t)$ marks the center of the neck.
What distinguishes this approach from standard parabolic rescaling
(see, e.g., \cite{AV97} or \cite{AK07}) is that we do not fix $\lambda(t)$
but instead consider it as a free parameter to be determined from the evolution of
$u$ in equation~\eqref{MCF-u}.\footnote{For fully general neckpinches, one would
also regard $x_0(t)$ as a free parameter; below, we impose reflection symmetry to
fix $x_0(t)\equiv 0$.}
By \cite{GS09}, one expects that $\lambda\approx\sqrt{T-t}$
and $\tau\approx-\log(T-t)$, where $T>0$ is the singularity time.
(We confirm this in Section~\ref{RepeatMainProof} below.)

Consider a solution $u(x,\theta,t)$ of \eqref{MCF-u} with initial condition $u_{0}(x,\theta)$.
We define a rescaled radius $v(y,\theta,\tau)$ by
\begin{equation}\label{Define-v}
v(y(x,t),\theta,\tau(t)):=\lambda^{-1}(t)\,u(x,\theta,t).
\end{equation}
Then $v$ initially satisfies $v(y,\theta,0)=v_0(y,\theta)$,
where $v_0(y,\theta):=\lambda^{-1}_{0} u_{0}(\lambda_{0}y,\theta)$,
with $\lambda_0$ the initial value of the scaling parameter $\lambda$.

In commuting $(y,\theta,\tau)$ variables, the quantity $v$ evolves by
\begin{equation}\label{MCF-v}
\dt v = A_v v +av -v^{-1},
\end{equation}
where $A_v$ is the quasilinear elliptic operator
\begin{equation} \label{Define-A}
A_v :=
    F_1(p,q)\dy^2 + v^{-2}F_2(p,q)\dz^2 + v^{-1}F_3(p,q)\dy\dz
    +v^{-2}F_4(p,q)\dz -ay\dy
\end{equation}
with $v$-dependent coefficients defined with respect to
\begin{equation}\label{Define-apq}
a:=-\lambda\partial_t \lambda,\quad p:=\dy v,\quad\text{and}\quad q:=v^{-1}\dz v,
\end{equation}
by
\begin{equation}\label{Define-Fi}
\begin{array}{llcllllc}
F_1(p,q) &:= &\displaystyle\frac{1+q^2}{1+p^2+q^2}, &\quad &\quad
F_2(p,q) &:= &\displaystyle\frac{1+p^2}{1+p^2+q^2},\\
\\
F_3(p,q) &:= &\displaystyle-\frac{2pq}{1+p^2+q^2}, &\quad &\quad
F_4(p,q) &:= &\displaystyle\frac{q}{1+p^2+q^2}.
\end{array}
\end{equation}

\medskip
In order to state our assumptions precisely, we introduce some further notation.
We denote the formal solution of the adiabatic approximation to equation~\eqref{MCF-v} by
\begin{equation}\label{adiabatic}
V_{r,s}(y):=\sqrt{\frac{2+sy^2}{\frac 12+r}},
\end{equation}
where $r$ and $s$ are positive parameters. We introduce a step
function\,\footnote{One motivation for this choice of $g$ may be found in the results of \cite{GS09},
which prove that $v=\big(1+o(1)\big)\sqrt{2+b(\tau)y^2}$ as $\tau\rightarrow\infty$
for the rotationally symmetric \textsc{mcf} neckpinch. For another motivation,
see Remark~\ref{OneReasonFor-g} below.}
\begin{equation}\label{Define-g}
g(y,s):=\left\{    \begin{array}{ccc}
    \frac{9}{10}\sqrt{2} & \text{if} & s y^{2}< 20,\\
    \\
    4 & \text{if} & s y^{2} \geq 20.
    \end{array}
\right.
\end{equation}
We define a norm $\|\cdot\|_{m,n}$ by
\[\|\phi\|_{m,n} := \left\| (1+y^2)^{-\frac m2} \dy^n \phi\right\|_{L^\infty}.\]
We also introduce a Hilbert space
$L^2_{\sigma}\equiv L^2(\mathbb{S}^1\times\mathbb{R};\Dz\,\sigma\,\Dy)$
with norm $\|\cdot\|_{\sigma}$, whose weighted measure is defined with
respect to $\Sigma\gg 1$ (to be fixed below) by
\begin{equation}\label{Define-sigma}
\sigma(y):=(\Sigma+y^2)^{-\frac{3}{5}}.
\end{equation}

Here are our assumptions. We state them for $v_{0}(y,\theta)=v(y,\theta,0)$,
but they easily translate to $u_{0}(x,\theta)=u(x,\theta,0)$ using the
relation $u_0(x,\theta)=\lambda_0 v(\lambda_0^{-1}x,\theta)$.

\medskip
\textbf{Main Assumptions.} There exist small positive constants $b_0,c_0$ such that:
\begin{itemize}
\item[{[A1]}] The initial surface is a graph over $\mathbb{S}^1\times\mathbb{R}$
determined by a smooth function $v_0(y,\theta)>0$ with the symmetries
$v_0(y,\theta)=v_0(-y,\theta)=v_0(y,\theta+\pi)$.

\item[{[A2]}] The initial surface satisfies $v_0(y,\cdot) > g(y, b_0)$.

\item[{[A3]}] The initial surface is a small deformation of a formal solution $V_{a_0,b_0}$
in the sense that for $(m,n)\in\left\{(3,0),\ (11/10,0),\ (2,1), \ (1,1)\right\}$, one has
\[ \|v_0-V_{a_0,b_0}\|_{m,n}<b_0^{\frac{m+n}{2}+\frac{1}{10}}.\]

\item[{[A4]}] The parameter $a_0=a(0)$ obeys the bound $|a_0-1/2|<c_0$.

\item[{[A5]}] The initial surface obeys the further derivative bounds:
\begin{align*}
\textstyle\sum_{n\neq0,\, 2\leq m+n \leq3} v_0^{-n}|\dy^m\dz^n v_0|&<b_0^2,\\
b_0v_0^{-\frac12}|\dy v_0| + b_0^{\frac12}|\dy^2v_0|+|\dy^3 v_0|&<b_0^{\frac32},\\
|\dy\dz^2 v_0|+v_0^{-1}|\dz^3 v_0|&<c_0.
\end{align*}

\item[{[A6]}] $b_0^{\frac45}\|\dy^4 v_0\|_\sigma + \|\dy^5v_0\|_\sigma%
+ \sum_{n\neq0,\, 4\leq m+n \leq5} \|v_0^{-n}\dy^m\dz^n v_0\|_\sigma < b_0^4$.
\end{itemize}
\medskip

Our method works because \textsc{mcf} preserves the symmetries in Assumption~[A1].
\begin{lemma}\label{Symmetries}
If $v_0$ has the symmetries $v_0(y,\theta)=v_0(-y,\theta)=v_0(y,\theta+\pi)$,
then the solution $v$ of equation~\eqref{MCF-v} shares the same symmetries,
$v(y,\theta,\tau)=v(-y,\theta,\tau)=v(y,\theta+\pi,\tau)$, for as long as it
remains a graph.
\end{lemma}
\begin{proof}
If $v(y,\theta,\tau)$ solves equation~\eqref{MCF-v}, so do $v(-y,\theta,\tau)$ and $v(y,\theta+\pi,\tau)$.
Thus the result follows from local well-posedness of the equation.
\end{proof}

\section{Implied evolution equations}

Recall from equation~\eqref{MCF-v} in Section~\ref{Basic} that
$\dt v = A_v v +av -v^{-1}$, where $A_v$ is the quasilinear operator
defined in equation~\eqref{Define-A}, with coefficients $F_\ell$ introduced
in definitions~\eqref{Define-apq}--\eqref{Define-Fi}. Our goal
in this section is to derive evolution equations for quantities of the
form
\begin{equation} \label{Define-vmnk}
v_{m,n,k}:= v^{-k}(\dy^m \dz^n v),
\end{equation}
defined with respect to integers $m,n\geq 0$ and a real number $k\geq 0$.

\begin{lemma}\label{ComputeOnceAndForAll}
The quantity $v_{m,n,k}$ evolves by
\begin{equation}	\label{CorrectEvolutionEquation}
\dt v_{m,n,k} = \left\{A_v + (k+1)v^{-2} - (m+k-1)a \right\}v_{m,n,k} + E_{m,n,k},
\end{equation}
where the individual terms of the nonlinear quantity $E_{m,n,k}=\sum_{\ell=0}^5 E_{m,n,k,\ell}$
are defined by
\begin{equation} \label{Define-Xmnkj}
\begin{array} {lll}
E_{m,n,k,0} & := &-kv^{-1}v_{m,n,k}(A_v v+ay\dy v),\\ \\
E_{m,n,k,1} & := & v^{-k}\dy^m\dz^n(F_1\dy^2v)-F_1\dy^2v_{m,n,k},\\ \\
E_{m,n,k,2} & := & v^{-k}\dy^m\dz^n(v^{-2}F_2\dz^2v)-v^{-2}F_2\dz^2v_{m,n,k},\\ \\
E_{m,n,k,3} & := & v^{-k}\dy^m\dz^n(v^{-1}F_3\dy\dz v)-v^{-1}F_3\dy\dz v_{m,n,k},\\ \\
E_{m,n,k,4} & := & v^{-k}\dy^m\dz^n(v^{-2}F_4\dz v)-v^{-2}F_4\dz v_{m,n,k},\\ \\
E_{m,n,k,5} & := & -v^{-k}\dy^m\dz^n(v^{-1})-v^{-2}v_{m,n,k}.
\end{array}
\end{equation}
\end{lemma}

\begin{proof}
We begin by using equations~\eqref{MCF-v} and \eqref{Define-A} to compute that
\begin{align*}
\dt\big(v^{-k}\dy^m\dz^n v\big)&=-kv_{m,n,k}v^{-1}\dt v + v^{-k}\dy^m\dz^n(\dt v)\\
&=-kv^{-1}v_{m,n,k}(A_v v)-kav_{m,n,k}+kv^{-2}v_{m,n,k}\\
&\quad+v^{-k}\dy^m\dz^n
\big\{F_1\dy^2v+v^{-2}F_2\dz^2v+v^{-1}F_3\dy\dz v+v^{-2}F_4\dz v\big\}\\
&\quad-av^{-k}\dy^m\dz^n(y\dy v)+av_{m,n,k}-v^{-k}\dy^m\dz^n(v^{-1}).
\end{align*}
We then rewrite this in the form
\begin{align*}
\dt v_{m,n,k}&=A_v v_{m,n,k}+\big\{(k+1)v^{-2}+a(1-k)\big\}v_{m,n,k}\\
&\quad-kv^{-1}v_{m,n,k}(A_vv)\\
&\quad+v^{-k}\dy^m\dz^n(F_1\dy^2v)-F_1\dy^2v_{m,n,k}\\
&\quad+v^{-k}\dy^m\dz^n(v^{-2}F_2\dz^2v)-v^{-2}F_2\dz^2v_{m,n,k}\\
&\quad+v^{-k}\dy^m\dz^n(v^{-1}F_3\dy\dz v)-v^{-1}F_3\dy\dz v_{m,n,k}\\
&\quad+v^{-k}\dy^m\dz^n(v^{-2}F_4\dz v)-v^{-2}F_4\dz v_{m,n,k}\\
&\quad-v^{-k}\dy^m\dz^n(v^{-1})-v^{-2}v_{m,n,k}\\
&\quad+a\big\{y\dy v_{m,n,k}-v^{-k}\dy^m\dz^n(y\dy v)\big\}.
\end{align*}
Comparing this to equations~\eqref{CorrectEvolutionEquation} and \eqref{Define-Xmnkj}
shows that to complete the proof, it suffices to observe that the final line above simplifies
as follows:
\[
a\big\{y\dy v_{m,n,k}-v^{-k}\dy^m\dz^n(y\dy v)\big\}
=-mav_{m,n,k}-kv^{-1}v_{m,n,k}(ay\dy v).
\]
\end{proof}

\begin{remark}\label{OneReasonFor-g}
An important feature of the calculation above is its linear reaction term,
which reveals how the evolution equation for $v_{m,n,k}$ ``improves''
as $m$ and $k$ increase, provided that $v$ is suitably bounded from below.
This partly explains our choice of step function $g$ in equation~\eqref{Define-g}.
\end{remark}

\begin{corollary}\label{GeneralEquation}
The quantity $v_{m,n,k}^2$ evolves by
\begin{equation} \label{Evolve-vmnk2}
\begin{array}{lrl}
\dt v_{m,n,k}^2 &= &A_v (v_{m,n,k}^2) + 2\left[(k+1)v^{-2}-(m+k-1)a\right]v_{m,n,k}^2\\ \\
                &  &- B_{m,n,k} + 2E_{m,n,k}v_{m,n,k},
\end{array}
\end{equation}
where $B_{m,n,k}:= A_v(v_{m,n,k}^2)-2v_{m,n,k} A_v v_{m,n,k}$ satisfies $B\geq0$
along with
\begin{equation}\label{BisGood}
B_{m,n,k} \geq \frac{v_{m+1,n,k}^2 + v_{m,n+1,k+1}^2}{1+p^2+q^2}
	-k^2 (v_{1,0,1}^2+v_{0,1,1}^2) v_{m,n,k}^2.
\end{equation}
\end{corollary}
\begin{proof}
Recall the useful fact that $v\dy^2 v =\frac12 \dy^2(v^2)-(\dy v)^2$.
Using this idea, the result follows by applying Cauchy--Schwarz to obtain first
\begin{align*}
B_{m,n,k}&= 2[ F_1(\dy v_{m,n,k})^2+v^{-2}F_2 (\dz v_{m,n,k})^2
   +v^{-1}F_3 (\dy v_{m,n,k})(\dz v_{m,n,k})]\\
   &\geq2\frac{(\dy v_{m,n,k})^2+(\dz v_{m,n,k})^2}{1+p^2+q^2}\geq0,
\end{align*}
and then applying it in case $k>0$ to the identities
\begin{align*}
(\dy v_{m,n,k})^2&=v_{m+1,n,k}^2 - 2v_{m+1,n,k} kv_{1,0,1}v_{m,n,k}+k^2v_{1,0,1}^2 v_{m,n,k}^2,\\
(\dz v_{m,n,k})^2&=v_{m,n+1,k}^2 - 2v_{m,n+1,k} kv_{0,1,1}v_{m,n,k}+k^2v_{0,1,1}^2 v_{m,n,k}^2.
\end{align*}
\end{proof}

\begin{remark}\label{WhyBisGood}
Once one has suitable first-order estimates for $v$, one can bound the quantity
$1+p^2+q^2$ from above, whereupon $B_{m,n,k}$ contributes valuable
higher-order terms of the form $-\ve(v_{m+1,n,k}^2 + v_{m,n+1,k+1}^2)$
to the evolution equation satisfied by $v_{m,n,k}$. The same first-order
estimates let us easily control the potentially bad terms in \eqref{BisGood} by
$k^2 (v_{1,0,1}^2+v_{0,1,1}^2) v_{m,n,k}^2\leq\ve v_{m,n,k}^2$.
\end{remark}

In estimating the ``commutators'' $E_{m,n,k,\ell}$ defined in \eqref{Define-Xmnkj}, the following
simple observation will be used many times, often implicitly. We omit its easy proof.

\begin{lemma}\label{UniformBounds}
For any $i,j\geq 0$ and $\ell=1,\dots,4$,
there exist constants $C_{i,j,\ell}$ such that
\[ |\partial_p^i\partial_q^j F_{\ell}(p,q)|\leq C_{i,j,\ell} \]
for all $p,q\in\mathbb{R}$.
\end{lemma}

We also freely use the following facts, usually without comment.
\begin{remark}
The quantity $E_{m,n,k,0}$ vanishes if $k=0$.
\end{remark}

\begin{remark}
The quantity $E_{m,n,k,5}$ vanishes if $m+n=1$.
\end{remark}

\section{The first bootstrap machine}\label{FirstBootstrap}

\subsection{Input}\label{FBMI}
Here are the conditions that constitute the input to our first bootstrap machine, whose
structure we describe below. By standard regularity theory for quasilinear parabolic
equations, if the initial data satisfy the Main Assumptions in Section~\ref{Basic} for
$b_0$ and $c_0$ sufficiently small, then the solution will satisfy the properties below
up to a short time $\tau_1>0$. (Note that to obtain some of the derivative bounds below,
one uses the general interpolation result, Lemma~\ref{InterpEst}.)

Some of these properties are global, while others are local in nature. In these conditions,
and in many of the arguments that follow, we separately treat the \emph{inner region}
$\{\beta y^2\leq20\}$ and the \emph{outer region} $\{\beta y^2\geq20\}$, both defined with
respect to
\begin{equation}
\beta(\tau):=(\kappa+\tau)^{-1},
\end{equation}
where $\kappa=\kappa(b_0,c_0)\gg1$. Note that our Main Assumptions imply that slightly
stronger conditions hold for the inner region. This is unsurprising: it is natural to
expect that the solution of equation~\eqref{MCF-v} is sufficiently close in that region
to the solution $V_{\frac12,\beta}$ of the equation $\frac{1}{2}y\dy V-\frac{1}{2}V+V^{-1}=0$
that is an adiabatic approximation of \eqref{MCF-v} there. (Compare~\cite{AK07} and \cite{GS09}.)
\medskip

Here are the global conditions:\footnote{\textsc{Notation}: In the remainder
of this paper, we write $\vp\ls\psi$ if there exists a uniform constant $C>0$
such that $\vp\leq C\psi$, and we define $\langle x \rangle := \sqrt{1+|x|^2}$.}
\begin{itemize}
\item[{[C0]}] For $\tau\in[0,\tau_1]$, the solution has
the uniform lower bound $v(\cdot,\cdot,\tau)\geq\kappa^{-1}$.

\item[{[C1]}] For $\tau\in[0,\tau_1]$, the solution satisfies
the first-order estimates
\[  |\dy v| \ls \beta^{\frac{2}{5}}v^{\frac{1}{2}},\quad
    |\dz v| \ls \beta^{\frac{3}{2}}v^2,\quad\text{and}\quad
    |\dz v| \ls v. \]

\item[{[C2]}] For $\tau\in[0,\tau_1]$, the solution satisfies
the second-order estimates
\[  |\dy^2 v|  \ls \beta^{\frac{3}{5}},\quad
    |\dy\dz v| \ls \beta^{\frac{3}{2}}v,\quad
    |\dy\dz v| \ls 1,\quad\text{and}\quad
    |\dz^2 v|  \ls \beta^{\frac{3}{2}}v^2. \]

\item[{[C3]}] For $\tau\in[0,\tau_1]$ the solution satisfies the
third-order decay estimates
\[|\dy^3 v|\ls\beta\quad\text{and}\quad
    v^{-n}|\dy^m\dz^n v|\ls\beta^{\frac{3}{2}}\]
for $m+n=3$ with $n\geq 1$, as well as the ``smallness estimate'' that
\[\beta^{-\frac{11}{20}}\left(|\dy^3 v|+|\dy^2\dz v|\right)
    +|\dy\dz^2 v|+v^{-1}|\dz^3 v|\ls(\beta_0+\ve_0)^{\frac{1}{40}}\]
for some $\ve_0=\ve_0(b_0,c_0)\ll1$.\footnote{The smallness estimate will only be
used in the proof of Theorem~\ref{SmallnessEstimates}.}

\item[{[Ca]}] For $\tau\in[0,\tau_1]$, the parameter $a$
satisfies\,\footnote{Proposition~\ref{IFT} will establish that $a$ is a $C^1$ function of time.}
\[\frac12-\kappa^{-1}\leq a \leq \frac12+\kappa^{-1}.\]
\end{itemize}

To state the remaining global conditions, we decompose the solution into
$\theta$-independent and $\theta$-dependent parts $v_1,v_2$, respectively,
defined by
\begin{equation}\label{Decomposition}
    v_1(y,\tau):=\frac{1}{2\pi}\int_0^{2\pi}v(y,\theta,\tau)\,\mathrm{d}\theta
    \quad\text{and}\quad
    v_2(y,\theta,\tau):=v(y,\theta,\tau)-v_1(y,\tau).
\end{equation}
We denote the norm in the Hilbert space $L^2_{\sigma}$ introduced in definition~\eqref{Define-sigma}
by $\|\cdot\|_{\sigma}$, and the inner product by
\[\lp\vp,\psi\rp :=
    \int_{\mathbb{R}}\int_{\mathbb{S}^1} \vp \psi\,\Dz\,\sigma\,\Dy.\]

Our remaining inputs are:
\begin{itemize}
\item[{[Cs]}] For $\tau\in[0,\tau_1]$, one has Sobolev bounds
$\|v^{-n}\dy^m\dz^n v\|_{\sigma}<\infty$ whenever
$4\leq m+n\leq 7$.

\item[{[Cr]}] There exists $0<\delta\ll 1$ such that the scale-invariant
bound $|v_2|\leq\delta v_1$ holds everywhere for
$\tau\in[0,\tau_1]$.

\item[{[Cg]}] For $\tau\in[0,\tau_1]$, one has
$\lan y \ran^{-1}|\dy v|\leq\Sigma^{\frac{1}{4}}\beta$.
\end{itemize}

\begin{remark}
Note that our global gradient Condition~[Cg] posits the $\beta$ decay rate
of the formal solution $V(y,\tau)=\sqrt{2+\beta y^2}$ but with a large constant
$\Sigma^{\frac{1}{4}}\gg1$. The effect of our bootstrap argument will be to
sharpen that constant.
\end{remark}

Remark~\ref{NoCircleHere} (below) shows that Condition~[Cb] in Section~\ref{SecondBootstrap},
which is directly implied by our Main Assumptions, in turn implies that
[Cg] holds, along with extra properties that are local to the inner region:
\begin{itemize}
\item[{[C0i]}] For $\tau\in[0,\tau_1]$ and $\beta y^2\leq20$, the quantity $v$
is uniformly bounded from above and below, so that for $\beta y^2\leq 20$,
one has bounds
\[\frac{9}{10}\sqrt 2\equiv g(y,\beta)\leq v(y,\cdot,\cdot)\leq C_0.\]

\item[{[C1i]}] For $\tau\in[0,\tau_1]$ and $\beta y^2\leq20$, the solution satisfies
the stronger first-order estimates
\[ |\dy v| \ls \beta^{\frac{1}{2}}v^{\frac{1}{2}}\quad\text{and}\quad
    |\dz v|\ls\kappa^{-\frac{1}{2}}v.\]
\end{itemize}

\subsection{Output}\label{FirstOutput}
The output of this machine consists of the following estimates, which
collectively improve Conditions~[C0]--[C3], [Cs], [Cr], [Cg], and [C0i]--[C1i]:
\[\begin{array}{c}
    v(y,\theta,\tau)\geq g(y,\beta),\quad
    v=\mathcal{O}(\lan y \ran)\,\text{ as }\, |y|\rightarrow\infty;\\ \\
    v^{-\frac{1}{2}}|\dy v|\ls \beta^{\frac{1}{2}},\quad
    |\dy v|\ls 1,\quad
    v^{-2}|\dz v|\ls\beta^{\frac{33}{20}},\quad
    v^{-1}|\dz v|\ls\kappa^{-\frac{1}{2}};\\ \\
    \beta|\dy^2 v|+v^{-1}|\dy\dz v|+v^{-2}|\dz^2 v|\ls\beta^{\frac{33}{20}};\\ \\
    |\dy\dz v|+v^{-1}|\dz^2 v|\ls(\beta_0+\ve_0)^{\frac{1}{20}};\\ \\
    \beta^{\frac{1}{2}}|\dy^3 v|+v^{-1}|\dy^2\dz v|+v^{-2}|\dy\dz^2 v|+v^{-3}|\dz^3v|
    \ls\beta^{\frac{33}{20}};\\ \\
    \beta^{-\frac{11}{20}}\left(|\dy^3 v|+|\dy^2\dz v|\right)
    +|\dy\dz^2 v|+v^{-1}|\dz^3 v|\ls(\beta_0+\ve_0)^{\frac{1}{20}}.
\end{array}\]
Here, $\beta_0\equiv\beta(0)$ and $\ve_0$ are independent of $\tau_1$. So these
improvements will allow us to propagate the assumptions above forward in time.
(See Section~\ref{RepeatMainProof} below.)

\subsection{Structure}

Here is how we establish these stronger estimates.

We first derive improved estimates for $v$ and its first derivatives in the outer region,
using our Main Assumptions and their implications [Ca] and [C0]--[C2] for $v$ and its
first and second derivatives, and using the stronger estimates [C0i]--[C1i] that hold on the
boundary circles $\beta y^2=20$. The reason why second-order conditions are needed to prove improved
first-order estimates is that the nonlinear terms in the evolution equation for $D^N v$ (a spatial
derivative of total order $N$) will in general contain derivatives of order $N+1$. This occurs
because first derivatives of $v$ appear in the quasilinear operator $A_v$ that controls the
evolution equations studied here.

\begin{theorem}\label{FirstOrderEstimates}
Suppose a solution $v=v(y,\theta,\tau)$ of equation~\eqref{MCF-v}
satisfies Assumption~[A1] at $\tau=0$, and Conditions~[Ca] and [C0]--[C2]
for $\beta y^2\geq20$ and $\tau\in [0,\tau_1]$. If Conditions~[C0i]--[C1i]
hold on the boundary of this region for $\tau\in [0,\tau_1]$, then for the
same time interval, the solution satisfies the following estimates throughout
the outer region $\{\beta y^2\geq20\}$:
\begin{equation}\label{v-below}
    v(y,\cdot,\cdot)\geq 4,
\end{equation}
\begin{equation}\label{v_y-est}
    |\dy v|\ls \beta^{\frac{1}{2}}v^{\frac{1}{2}},
\end{equation}
\begin{equation}\label{v_y-bound}
    |\dy v|\ls 1,
\end{equation}
\begin{equation}\label{v_z-est}
    |\dz v|\leq C_0\kappa^{-\frac{1}{2}} v,
\end{equation}
where $C_0$ depends only on the initial data and not on $\tau_1$.
\end{theorem}

Note that estimate~\eqref{v_y-bound} provides $C$ such that
$|\dy v|\leq C$ in the outer region for $\tau\in[0,\tau_1]$. By
Condition~[C0i], one has $v(\pm\sqrt{20}\beta^{-\frac{1}{2}})\leq c^2$
for some $c>0$. Hence by quadrature, one obtains the immediate corollary that
\begin{equation}\label{v-above}
    v=\mathcal{O}(\lan y \ran)\quad\text{as}\quad |y|\rightarrow\infty.
\end{equation}
Theorem~\ref{FirstOrderEstimates} is proved in Section~\ref{ProveFirst}.
\medskip

In the second step, we derive estimates for second and third derivatives of $v$. It is
reasonable to expect that one could bound second derivatives  without assumptions on
third derivatives, thus allowing us to ``close the loop'' at two derivatives in our
bootstrap arguments. A reason for this expectation is that once one has suitable
estimates on first derivatives, one can show that higher-order derivatives occur in
combinations $\phi\cdot(D^{N+1}v)-\psi\cdot(D^{N+1}v)^2$, where $|\phi|\ll\psi$. Indeed,
Theorem~\ref{FirstOrderEstimates} suffices to bound the quantity $\psi>0$ from below.
This expectation is correct in the outer region, where one can obtain the needed
second-derivative bounds using only maximum-principle arguments. But in the inner
region, more complicated machinery is needed. To construct improved second-order
estimates there, we introduce and bound suitable Lyapunov functionals and then apply
Sobolev embedding theorems. This method requires us to assume (and subsequently improve)
pointwise bounds on third-order derivatives and $L^2_\sigma$ bounds on derivatives of
orders four and five. What makes this method more effective than the maximum principle
in the inner region is the fact that the Lyapunov functionals allow integration by parts.

\begin{remark}\label{FixAxis}
The main reason that integration improves our estimates is
the fact that for any smooth function $f$ which is orthogonal to
constants in $L_2(\{y\}\times\mathbb{S}^1)$, our assumption
of $\pi$-periodicity implies that
\[\int_{\{y\}\times\mathbb{S}^1}(\dz f)^2 \,\Dz
    \geq 4\int_{\{y\}\times\mathbb{S}^1} f^2 \,\Dz.\]
\end{remark}

\begin{remark}
Here is one reason why we need the pointwise estimates on third-order derivatives in [C3].
(Another reason will become clear in Section~\ref{SecondBootstrap}.)  In estimating nonlinear terms
in the evolution of the Lyapunov functionals introduced below, one encounters quantities of the sort
$\lp D^N(F_\ell),(D^2v)(D^Nv)\rp$, where the coefficients $F_\ell$ are given in definition~\eqref{Define-Fi}.
After integration by parts, one gets, schematically,
$\lp D^{N-1}(F_\ell),(D^3v)(D^Nv)+(D^2v)(D^{N+1}v)\rp$.
Such terms are comparable to $\lp D^{N}(v),(D^3v)(D^Nv)+(D^2v)(D^{N+1}v)\rp$.
So we need $L^\infty$ control on third derivatives in order to impose only
$L^2_\sigma$ assumptions on higher derivatives.
\end{remark}

In Section~\ref{InnerProof}, we prove the following result.

\begin{theorem}\label{InnerEstimates}
Suppose that a solution $v=v(y,\theta,\tau)$ of equation~\eqref{MCF-v}
satisfies Assumption~[A1] at $\tau=0$, and Conditions~[Ca], [C0]--[C3],
[Cs], [Cr], [Cg], and [C0i]--[C1i] for $\tau\in [0,\tau_1]$. Then for
the same time interval, the solution satisfies the following pointwise
bounds throughout the inner region $\{\beta y^2\leq20\}$:
\begin{equation}\label{InnerSecondOrder}
\beta|\dy^2 v|+v^{-1}|\dy\dz v|+v^{-2}|\dz^2 v|\ls\beta^{\frac{33}{20}}
\end{equation}
and
\begin{equation}\label{InnerThirdOrder}
\beta^{\frac{1}{2}}|\dy^3 v|+v^{-1}|\dy^2\dz v|+v^{-2}|\dy\dz^2 v|
    +v^{-3}|\dz^3 v|\ls\beta^{\frac{33}{20}},
\end{equation}
\end{theorem}
The fact that we do not achieve the expected $\beta^2$ decay on the
\textsc{rhs} of estimates~\eqref{InnerSecondOrder}--\eqref{InnerThirdOrder}
is due to our use of Sobolev embedding with respect to the weighed measure
introduced in equation~\eqref{Define-sigma}.
\medskip

Then using Theorem~\ref{InnerEstimates} to ensure that they hold on
the boundary of the inner region, we extend its estimates to the outer region.

\begin{theorem}\label{OuterEstimates}
Suppose that a solution $v=v(y,\theta,\tau)$ of equation~\eqref{MCF-v}
satisfies Assumption~[A1] at $\tau=0$, and Conditions~[Ca] and [C0]--[C3]
for $\beta y^2\geq20$ and $\tau\in [0,\tau_1]$. Then the estimates
\begin{equation}\label{OuterSecondOrder}
    \beta|\dy^2 v|+v^{-1}|\dy\dz v|+v^{-2}|\dz^2 v|\ls\beta^{\frac{33}{20}}
\end{equation}
and
\begin{equation}\label{OuterThirdOrder}
\beta^{\frac{1}{2}}|\dy^3 v|+v^{-1}|\dy^2\dz v|+v^{-2}|\dy\dz^2 v|+v^{-3}|\dz^3v|
    \ls\beta^{\frac{33}{20}}
\end{equation}
hold throughout the entire outer region $\{\beta y^2\geq20\}$ during the
same time interval, provided that they hold on the boundary $\beta y^2=20$.
\end{theorem}

As an easy corollary, we apply the interpolation result Lemma~\ref{InterpEst} to get
our final first-order estimate claimed in subsection~\ref{FirstOutput}, namely
\begin{align*}
\frac{|\dz v|}{v^2}
    \leq \frac{1}{(1-\delta)^2}\frac{|\dz v_2|}{v_1^2}
    &\leq \frac{C_1}{(1-\delta)^2}\max_{\theta\in[0,2\pi]}\frac{|\dz^2 v_2|}{v_1^2}\\
    &\leq C_2\frac{(1+\delta)^2}{(1-\delta)^2}\max_{\theta\in[0,2\pi]}\frac{|\dz^2 v|}{v^2}
    \leq C_3 \beta^{\frac{33}{20}}.
\end{align*}
\smallskip

Theorem~\ref{OuterEstimates} is proved in Section~\ref{ProveLast}.

\medskip
Finally, we improve the ``smallness estimates'' in Condition~[C3],
producing improved bounds for $|\dy\dz^2 v|$ and $v^{-1}|\dz^3 v|$;
these serve as inputs to the second bootstrap machine constructed in
Section~\ref{SecondBootstrap} below.

\begin{theorem}\label{SmallnessEstimates}
Suppose that a solution $v=v(y,\theta,\tau)$ of equation~\eqref{MCF-v}
satisfies Assumption~[A1] at $\tau=0$, and Conditions [Ca], [C0]--[C3],
[Cs], [Cr], [Cg], and [C0i]--[C1i] for $\tau\in [0,\tau_1]$. Then for
the same time interval, the solution satisfies
\[\beta^{-\frac{11}{10}}\left[(\dy^3 v)^2+(\dy^2\dz v)^2\right]
    +(\dy\dz^2 v)^2+v^{-2}(\dz^3 v)^2
    \ls(\beta_0+\ve_0)^{\frac{1}{10}}.\]
\end{theorem}
As an easy corollary, we apply Lemma~\ref{InterpEst} to get our final second-order
estimates claimed in subsection~\ref{FirstOutput}, namely
\[|\dy\dz v|+v^{-1}|\dz^2 v|\ls(\beta_0+\ve_0)^{\frac{1}{20}}.\]
\smallskip

Theorem~\ref{SmallnessEstimates} is proved in Section~\ref{Smallness}.
Its proof completes our construction of the first bootstrap machine.

\section{Estimates of first-order derivatives}\label{ProveFirst}

In this section, we prove the estimates that constitute Theorem~\ref{FirstOrderEstimates}.
Our arguments use a version of the parabolic maximum principle adapted to noncompact
domains, which we state as Proposition~\ref{PMP} and prove in Appendix~\ref{ProvePMP}.
\smallskip

We start with a simple observation illustrating how one applies Proposition~\ref{PMP}
to control $\inf v$ for large $|y|$.

\begin{lemma}\label{NotTooSmall}
If there exist constants $\ve>0$ and
$c\geq\sup_{0\leq\tau\leq\tau_1}a^{-1/2}$,
and a continuous function $b(\tau)\geq 0$ such that
(\textsf{a}) $v_0(y,\cdot)\geq c$ for $|y|\geq b(0)$,
(\textsf{b}) $v(\pm b(\tau),\cdot,\tau)\geq c$ for $0\leq\tau\leq\tau_1$, and
(\textsf{c}) $v(\cdot,\cdot,\tau)\geq\ve$ for $|y|\geq b(\tau)$ and $0\leq\tau\leq\tau_1$,
then \[v(y,\theta,\tau)\geq c\] for $|y|\geq b(\tau)$,
$0\leq\theta\leq2\pi$, and $0\leq\tau\leq\tau_1$.
\end{lemma}

\begin{proof}
Apply the maximum principle in the form of Proposition~\ref{PMP} to $c-v$
in the region $\Omega:=\{|y|\geq b(\tau)\}$, with $D:=\dt-A_v-B$ and $B:=a+(cv)^{-1}$.
By (\textsf{c}), $B$ is bounded from above.
Computing $D(c-v)=c^{-1}(1-ac^2)\leq 0$ establishes property (\textsc{i}).
Hypotheses (\textsf{a}) and (\textsf{b}) ensure that property (\textsc{ii}) is satisfied.
Property (\textsc{iii}) follows from Conditions~[C0i] and [C1] by quadrature.
Hence Proposition~\ref{PMP} implies that $v\geq c$.
\end{proof}
\smallskip

We now establish the estimates that constitute Theorem~\ref{FirstOrderEstimates}.

\begin{proof}[Proof of estimate~\eqref{v-below}]
We apply Lemma~\ref{NotTooSmall} in the outer region $\beta y^2\geq 20$.
Let $\ve=\kappa^{-1}$, $c=4$ (large enough by Condition~[Ca]),
and $b(\tau)=\sqrt{20\beta(\tau)^{-1}}$. Then
Assumption~[A2] implies property (\textsf{a}); Condition~[C0i] gives (\textsf{b});
and [C0] yields (\textsf{c}). Thus Lemma~\ref{NotTooSmall} implies the estimate.
\end{proof}

\begin{proof}[Proof of estimate~\eqref{v_y-est}]
Set $w:=v_{1,0,\frac{1}{2}}\equiv v^{-\frac{1}{2}}\dy v$, using definition~\eqref{Define-vmnk}.
Then by equation~\eqref{Evolve-vmnk2} in Lemma~\ref{ComputeOnceAndForAll}, one has
\[
\dt w^2 = A_v(w^2)+(3v^{-2}-a)w^2-B_{1,0,\frac{1}{2}}+2w\sum_{\ell=0}^4 E_{1,0,\frac{1}{2},\ell},
\]
where the nonlinear commutator terms $E_{1,0,\frac{1}{2},\ell}$ are defined in equation~\eqref{DecomposeCommutators}
and display~\eqref{Define-Xmnkj}.

We now calculate and estimate the nonlinear terms above. Observe, for example, that
\begin{align*}
E_{1,0,\frac{1}{2},1}&:=v^{-\frac{1}{2}}\dy(F_1 \dy^2 v)-F_1 \dy^2(v^{-\frac{1}{2}}\dy v)\\
    &=(\dy F_1)\left[v^{-\frac{1}{2}}(\dy^2 v)\right]
    +F_1\left[\frac{3}{2}v^{-\frac{3}{2}}(\dy^2 v)(\dy v)-\frac{3}{4}v^{-\frac{5}{2}}(\dy v)^3\right],
\end{align*}
where\,\footnote{Recall from definition~\eqref{Define-apq} that $p=\dy v$ and $q=v^{-1}\dz v$.}
\[\dy F_\ell
    = (\partial_p F_\ell) (\dy^2 v)
    +(\partial_q F_\ell)[v^{-1}\dy\dz v - v^{-2}(\dy v)(\dz v)].\]
By Lemma~\ref{UniformBounds}, there are uniform bounds
$|\partial_p^i\partial_q^j F_{\ell}(p,q)|\leq C_{i,j,\ell}$.
Applying this fact (which we freely use below without comment)
one establishes by similar direct computations that
\[\begin{array}{ccl}
|E_{1,0,\frac{1}{2},0}|&\ls
    &v^{-\frac{1}{2}}|\dy v|\big[|\dy^2 v|+v^{-2}|\dz^2 v|+v^{-1}|\dy\dz v|+v^{-2}|\dz v|\big] \\ \\
|E_{1,0,\frac{1}{2},1}|&\ls
    &\big[v^{-\frac{1}{2}}|\dy^2 v|+v^{-\frac{3}{2}}|\dy^2 v||\dy v|+v^{-\frac{5}{2}}|\dy v|^3\big] \\ \\
|E_{1,0,\frac{1}{2},2}|&\ls
    &\big[v^{-\frac{5}{2}}|\dz^2 v|+v^{-\frac{7}{2}}(|\dy v||\dz^2 v|+|\dy\dz v||\dz v|)
    +v^{-\frac{9}{2}}|\dy v||\dz v|^2\big] \\ \\
|E_{1,0,\frac{1}{2},3}|&\ls
    & \big[v^{-\frac{3}{2}}|\dy^2 v||\dy\dz v|
    +v^{-\frac{5}{2}}(|\dy\dz v|^2+|\dy\dz v||\dy v|+|\dy^2 v|^|\dz v|)\\
    &&\, +v^{-\frac{7}{2}}(|\dy\dz v||\dy v|+|\dy v|^2|\dz v|)\big]\\ \\
|E_{1,0,\frac{1}{2},4}|&\ls
        & \big[v^{-\frac{5}{2}}|\dy^2 v||\dz v|+v^{-\frac{7}{2}}(|\dy v||\dz v|+|\dy\dz v||\dz v|)
        +v^{-\frac{9}{2}}|\dy v||\dz v|^2\big].
\end{array}\]
By Conditions~[C1]--[C2], one has
    $|\dy v|\ls v^{\frac{1}{2}}\beta^{\frac{2}{5}}$,
    $|\dz v|\ls v^2\beta^{\frac{3}{2}}$,
    $|\dy^2 v|\ls\beta^{\frac{3}{5}}$,
    $|\dy\dz v|\ls v \beta^{\frac{3}{2}}$,
    $|\dy\dz v|\ls 1$, and
    $|\dz^2 v|\ls v^2 \beta^{\frac{3}{2}}$.
Combining these inequalities and using estimate~\eqref{v-below}, one readily obtains
\[2\left|w\sum_{\ell=0}^4 E_{1,0,\frac{1}{2},\ell}\right|\leq C\beta^{\frac{3}{5}}|w|.\]
By Condition~[Ca] and estimate~\eqref{v-below},
one has $(3v^{-2}-a)w^2\leq -\frac{1}{4}w^2$ in the outer region,
for $\kappa$ large enough, whereupon completing the square shows that
\[C\beta^{\frac{3}{5}}|w|-\frac{1}{8}w^2\leq 2C^2 \beta^{\frac{6}{5}}.\]

Now let $\Gamma>0$ be a large constant to be chosen below. Using
Corollary~\ref{GeneralEquation} to see that $B_{1,0,\frac{1}{2}}\geq0$,
and applying the estimates above, one obtains
\[
\dt(w^2-\Gamma\beta)
    \leq A_v(w^2-\Gamma\beta)-\frac{1}{8}(w^2-\Gamma\beta)
    +\left(\beta+\frac{2C^2}{\Gamma}\beta^{\frac{1}{5}}-\frac{1}{8}\right)
    \Gamma\beta.
\]
We shall apply Proposition~\ref{PMP}. By taking $\kappa$ and $\Gamma$
large enough, we can ensure that
\[
\dt(w^2-\Gamma\beta)
    \leq A_v(w^2-\Gamma\beta)-\frac{1}{8}(w^2-\Gamma\beta),
\]
which is property~(\textsc{i}). Making $\Gamma$ larger if necessary,
we can by our Main Assumptions ensure that $w^2\leq\Gamma\beta$ at $\tau=0$,
whereupon property~(\textsc{ii}) follows from Condition~[C1i] on the boundary.
Property~(\textsc{iii}) is a consequence of Condition~[C0i] on the boundary
and Condition~[C1] in the outer regon. Hence we have $w^2\leq\Gamma\beta$ in the
outer region for $0\leq\tau\leq\tau_1$ by the maximum principle.
\end{proof}

\begin{proof}[Proof of estimate~\eqref{v_y-bound}]
Arguing as in the proof of Corollary~\ref{GeneralEquation},
one computes from equation~\eqref{Evolve-vmnk2}, equation~\eqref{DecomposeCommutators},
and definition~\eqref{Define-Xmnkj}, that
\[\dt(\dy v)^4
    \leq A_v(\dy v)^4 + 4v^{-2}(\dy v)^4
    +4(\dy v)^3\sum_{\ell=1}^4 E_{1,0,0,\ell}.\]
By estimate~\eqref{v_y-est} and the implications of
Conditions~[C1]--[C2] that
    $|\dz v|\ls v^2\beta^{\frac{3}{2}}$,
    $|\dz v|\ls v$,
    $|\dy^2 v|\ls\beta^{\frac{3}{5}}$,
    $|\dy\dz v|\ls v \beta^{\frac{3}{2}}$,
    $|\dy\dz v|\ls 1$, and
    $|\dz^2 v|\ls v^2 \beta^{\frac{3}{2}}$,
one can estimate that
$|E_{1,0,0,1}|\ls\beta^{\frac{6}{5}}$, with
$|E_{1,0,0,\ell}|\ls\beta^2$ for
$\ell=2,3,4$.

Define $w:=(\dy v)^4+1$. Then by estimate~\eqref{v_y-est}
and Young's inequality, one has
\begin{align*}
\dt w
    &\leq A_v(w)+C_1\beta^2
        +C_2\beta^{\frac{6}{5}}\left[(\dy v)^4+1\right]\\
    &\leq A_v(w)+C\beta^{\frac{6}{5}}w.
\end{align*}
Let $\vp(\tau)$ solve the \textsc{ode} $\vp'=C\beta^{\frac{6}{5}}\vp$
with initial condition $\vp(0)=C_0$, where $C_0$ may be
chosen by our Main Assumptions so that $w\leq C_0$ at $\tau=0$. Then
\[\dt(w-\vp)\leq A_v(w-\vp)+C\beta^{\frac{6}{5}}(w-\vp).\]
Properties~(\textsc{i})--(\textsc{ii}) of Proposition~\ref{PMP}
are clearly satisfied for this equation, while property~(\textsc{iii})
follows from estimate~\eqref{v_y-est}. Hence we have $w\leq\vp$ by the
maximum principle. Because $\vp$ is uniformly bounded in time, the
lemma follows.
\end{proof}

\begin{proof}[Proof of estimate~\eqref{v_z-est}]
Let $w:=v_{0,1,1}\equiv v^{-1}\dz v$.
By Corollary~\ref{GeneralEquation}, $w^2$ satisfies the differential inequality
\[ \dt w^2 \leq A_v(w^2)+4v^{-2}w^2+2\left|w\sum_{\ell=0}^4 E_{0,1,1,\ell}\right|.\]
We proceed to bound the reaction terms.
By Condition~[C1], one has \[v^{-2}w^2\ls\beta^3.\]
By Conditions~[C1]--[C2], one has $|E_{0,1,1,0}|\ls\beta^{\frac{21}{10}}$.
By direct computation (compare estimate~\eqref{dzF_ell-estimate} below),
one has $v^{-1}|\dz F_\ell|\ls\beta^{\frac{3}{2}}$. Thus
one gets $|E_{0,1,1,1}|\ls\beta^2$ and $|E_{0,1,1,\ell}|\ls\beta^3$
for $\ell=2,3,4$. By Condition~[C1], one has $|w|\ls 1$, and thus
$\dt w^2 \leq A_v(w^2)+C\beta^2$ for some $C<\infty$.

Now let $\vp(\tau)$ solve $\vp'=C\beta^2$ with $\vp(0)=C_*\kappa^{-1}$.
By our Main Assumptions, $w\leq\vp$ at $\tau=0$ for $C_*>0$ sufficiently large
and depending only on the initial data. Because
\[\dt (w^2-\vp) \leq A_v(w^2-\vp),\]
the maximum principle implies that $w^2\leq\vp$ for as long as the solution exists.
\end{proof}

The proof of Theorem~\ref{FirstOrderEstimates} is now complete.

\section{Decay estimates in the inner region}\label{InnerBootstrap}

In this section, we construct the machinery that will prove Theorem~\ref{InnerEstimates}.

Given integers $m,n\geq 0$, we define functionals
$\Omega_{m,n}=\Omega_{m,n}(\tau)$ by
\begin{equation}\label{Define-Omega-mn}
\Omega_{m,n}
    := \int_{-\infty}^{\infty}\int_{0}^{2\pi}
        v_{m,n,n}^2\,\Dz\,\sigma\,\Dy
    \equiv \int_{-\infty}^{\infty}\int_{0}^{2\pi}
        v^{-2n}|\dy^m\dz^n v|^2\,\Dz\,\sigma\,\Dy,
\end{equation}
where the notation $v_{m,n,k}$ appears in equation~\eqref{Define-vmnk}, and
$\sigma=\sigma(y)$ is defined in equation~\eqref{Define-sigma}.
By Conditions~[C2], [C3], and [Cs], these functionals and their
$\tau$-derivatives are well defined if $m+n\leq5$.

Our strategy to prove Theorem~\ref{InnerEstimates} consists of three steps.
\textsc{(i)} We bound weighted sums of $\Omega_{m,n}$ with $2\leq m+n\leq3$.
\textsc{(ii)} We bound weighted sums of $\Omega_{m,n}$ with $4\leq m+n\leq5$.
\textsc{(iii)} We apply Sobolev embedding, using the facts that
$|y|\ls\beta^{-\frac{1}{2}}$ in the inner region, and that
$\sigma\sim |y|^{-\frac{6}{5}}$ as $|y|\rightarrow\infty$.

A consequence of this strategy is that we only need strong estimates for
second-order derivatives. It suffices to show that higher derivatives decay
at the same rates as those of second order, rather than at the faster rates
one would expect from parabolic smoothing. This somewhat reduces the work
necessary to bound derivatives of orders three through five.

\subsection{Differential inequalities}\label{PointwiseSecond}

We first derive differential inequalities satisfied by the squares of the
second-order quantities appearing in Theorem~\ref{InnerEstimates}. To avoid
later redundancy, these estimates are designed to be useful in both the inner
and outer regions. They could be obtained using the more general techniques
developed in Appendix~\ref{GetHigh} and employed below to bound higher
derivatives. We chose to derive them explicitly here, both to obtain
the sharpest results possible and to help the reader by introducing our
methods in as transparent a manner as we can.

Observe that Conditions~[C0i] and [C1i] imply that inequalities~\eqref{v_y-est}--\eqref{v_z-est} of
Theorem~\ref{FirstOrderEstimates} hold in the inner region as well as the outer region. Thus we
assume in this subsection that those inequalities hold globally.

\begin{lemma}\label{v200-estimate}
There exist $0<\ve<C<\infty$ such that the quantity $|\dy^2 v|^2$ satisfies
\begin{align*}
\dt|\dy^2 v|^2
    &\leq A_v(|\dy^2 v|^2)+X|\dy^2 v|^2
    -Y\left(|\dy^3 v|^2+v^{-2}|\dy^2\dz v|^2\right)\\
    &+C\left(\beta^{\frac{11}{5}}
        +\beta|\dy^2 v|
        +\beta^{\frac{11}{10}}v^{-2}|\dy\dz^2 v|\right),
\end{align*}
where in the inner region $\{\beta y^2\leq 20\}$,
\[X=2(v^{-2}-a)\quad\text{and}\quad Y\geq\frac{2}{1+\ve},\]
while in the outer region,
\[X\leq\frac{1}{8}-2a\leq-\frac{3}{4}\quad\text{and}\quad Y\geq\ve.\]
\end{lemma}

\begin{proof}
Define $w:=v_{2,0,0}\equiv \dy^2 v$.
Then by equation~\eqref{Evolve-vmnk2} and Corollary~\ref{GeneralEquation}, one has
\[
\dt w^2
    \leq A_v(w^2)+X w^2-Y\left(|\dy^3 v|^2+v^{-2}|\dy^2\dz v|^2\right)
    +2\left|w\sum_{\ell=1}^5 E_{2,0,0,\ell}\right|,
\]
where $X:=2(v^{-2}-a)$ and $Y:=\frac{2}{1+p^2+q^2}$.

The estimates for $X$ in the outer region follow from estimate~\eqref{v-below}
for $v$ there, and Condition~[Ca].

Although we discarded analogous quantities $B_{m,n,k}$ in the proof
of Theorem~\ref{FirstOrderEstimates}, we retain them above to help us control
third-order derivatives of $v$.
By Corollary~\ref{GeneralEquation}, estimate~\eqref{v_y-bound} for $p=\dy v$, and
estimate~\eqref{v_z-est} for $q=v^{-1}\dz v$, there exists $\ve>0$ such that
$B_{2,0,0}\geq 2 \ve\left(|\dy^3 v|^2+v^{-2}|\dy^2\dz v|^2\right)$ in the outer
region. (Compare Remark~\ref{WhyBisGood}.)

In the inner region, one can do better. Condition~[C0i]
gives uniform bounds for $v$ there, whence Condition~[C1i] implies
that $|p|=|\dy v|\ls\beta^{\frac{1}{2}}$ and similarly that
$|q|=|v^{-1}\dz v|\ls\beta^{\frac{3}{2}}$ in the inner region. Hence
\[B_{2,0,0}\geq\left(\frac{2}{1+\ve}+\ve\right)\left(|\dy^3 v|^2+v^{-2}|\dy^2\dz v|^2\right)\]
in the inner region.

We next estimate the $E_{2,0,0,\ell}$ terms above.
By Conditions~[C1]--[C2] and estimate~\eqref{v_y-bound}, one has
\begin{equation}\label{dyF_ell-estimate}
|\dy F_\ell|
    = \big|(\partial_{p}F_\ell) (\dy^2 v)
        +(\partial_q F_\ell)(v^{-1}\dy\dz v - v^{-2}\dy v \dz v)\big|
    \ls\beta^{\frac{3}{5}}.
\end{equation}
Next one calculates that
\begin{align*}
\dy^2 F_\ell
    =&\,(\partial_p^2 F_\ell)(\dy^2 v)^2
    +2(\partial_p \partial_q F_\ell)(\dy^2 v)
            (v^{-1}\dy\dz v - v^{-2}\dy v \dz v)\\
    &+(\partial_q^2 F_\ell)
            (v^{-1}\dy\dz v - v^{-2}\dy v \dz v)^2
    +(\partial_p F_\ell)(\dy^3 v)\\
    &+(\partial_q F_\ell)[v^{-1}\dy^2 \dz v-2v^{-2}\dy\dz v\dy v-v^{-2}\dy^2 v\dz v+2v^{-3}(\dy v)^2\dz v]
\end{align*}
and uses Conditions~[C1]--[C2] with inequalities~\eqref{v_y-est}--\eqref{v_y-bound}
from Theorem~\ref{FirstOrderEstimates} to estimate
\begin{equation}\label{dy2F_ell-estimate}
|\dy^2 F_\ell|\ls \beta^{\frac{6}{5}}+|\dy^3 v|+v^{-1}|\dy^2 \dz v|.
\end{equation}
Collecting the estimates above and using Condition~[C2] again, one obtains
\[|E_{2,0,0,1}|
    =|(\dy^2 F_1)(\dy^2 v)+2(\dy F_1)(\dy^3 v)|
    \ls \beta^{\frac{9}{5}}
    +\beta^{\frac{3}{5}}(|\\dy^3 v|+v^{-1}|\dy^2 \dz v|).\]
In similar fashion, one derives
\[|E_{2,0,0,2}|\ls
    \beta^2
    +\beta^{\frac{3}{2}}(|\dy^3 v|+v^{-1}|\dy^2\dz v|)
    +\beta^{\frac{1}{2}}v^{-2}|\dy\dz^2 v|\]
and
\[|E_{2,0,0,3}|\ls
    \beta^{\frac{21}{10}}
    +\beta^{\frac{3}{2}}|\dy^3 v|
    +\beta^{\frac{1}{2}}v^{-1}|\dy^2\dz v|.\]
Noting that
\begin{align*}
E_{2,0,0,4}
    =&\,(\dy^2 F_4)(v^{-2}\dz v)
        +(\dy F_4)(2v^{-2}\dy \dz v-4v^{-3}\dy v\dz v)\\
    &-F_4
        [4v^{-3}\dy\dz v\dy v+2v^{-3}\dy^2 v\dz v-6v^{-4}(\dy v)^2\dz v],
\end{align*}
one applies Condition~[C2] and estimate~\eqref{v_y-est} to get
\[|E_{2,0,0,4}|\ls \beta^2
    +\beta^{\frac{3}{2}}(|\dy^3 v|+v^{-1}|\dy^2 \dz v|).\]
By estimate~\eqref{v_y-est} again, one has
\[|E_{2,0,0,5}|=2v^{-3}(\dy v)^2\ls\beta.\]

Collecting the estimates above and using Condition~[C2], we
conclude that
\[\left|w\sum_{\ell=1}^5 E_{2,0,0,\ell}\right|
    \ls\beta|w|
    +\beta^{\frac{11}{10}}
           \left(|\dy^3 v|+v^{-1}|\dy^2 \dz v|+v^{-2}|\dy\dz^2 v|\right).\]
Using Remark~\ref{WhyBisGood} and completing squares to see that
$\beta^{\frac{11}{10}}|\dy^3 v|-\ve|\dy^3 v|^2\leq\frac{1}{4\ve}\beta^{\frac{11}{5}}$
and
$\beta^{\frac{11}{10}}v^{-1}|\dy^2 \dz v|-\ve v^{-2} |\dy^2 \dz v|^2\leq\frac{1}{4\ve}\beta^{\frac{11}{5}}$,
we combine the estimates above to obtain the lemma.
\end{proof}

\begin{lemma}\label{v111-estimate}
There exist $0<\ve<C<\infty$ such that the
quantity $v^{-2}|\dy\dz v|^2$ satisfies
\begin{align*}
\dt(v^{-2}|\dy\dz v|^2)
    &\leq A_v(v^{-2}|\dy\dz v|^2)+X v^{-2}|\dy\dz v|^2
    -Y\left(v^{-2}|\dy^2\dz v|^2+v^{-4}|\dy\dz^2 v|^2\right)\\
    &+C\left(\beta^4+\beta^2v^{-1}|\dy\dz v|
        +\beta^3|\dy^3 v|+\beta^2v^{-3}|\dz^3 v|\right),
\end{align*}
where in the inner region $\{\beta y^2\leq 20\}$,
\[X=2(2v^{-2}-a)+\ve\quad\text{and}\quad Y\geq\frac{1}{1+\ve},\]
while in the outer region,
\[X\leq\frac{1}{4}-2a+\ve\leq-\frac{5}{8}\quad\text{and}\quad Y\geq\ve.\]
\end{lemma}

\begin{proof}
Define $w:=v_{1,1,1}\equiv v^{-1}\dy\dz v$.
Then by equation~\eqref{Evolve-vmnk2} and estimate~\eqref{BisGood}, one has
\[\dt w^2\leq A_v(w^2)+X w^2-Y\left(v^{-2}|\dy^2\dz v|^2+v^{-4}|\dy\dz^2 v|^2\right)
    +2\left|w\sum_{\ell=0}^5 E_{1,1,1,\ell}\right|,\]
where $X:=2(2v^{-2}-a)+\ve$ and $Y:=\frac{1}{1+p^2+q^2}$.
As explained in Remark~\ref{WhyBisGood}, existence of $0<\ve\ll1$ follows easily
from \eqref{BisGood} and the estimates for $|v_{1,0,1}|$ and $|v_{0,1,1}|$ implied by
Theorem~\ref{FirstOrderEstimates}. The estimates for $X$ in the outer region follow
from estimate~\eqref{v-below} for $v$ there, and Condition~[Ca].
The estimates for $Y$ are the same as those in Lemma~\ref{v200-estimate} above.

Using Condition~[C2], it is easy to see that
$|E_{1,1,1,0}|\ls\beta^{\frac{21}{10}}$.
We compute and use Conditions~[C1]--[C2] to estimate
\begin{equation}\label{dzF_ell-estimate}
v^{-1}|\dz F_\ell|=v^{-1}\big|(\partial_p F_\ell)(\dy\dz v)
    +(\partial_q F_\ell)[v^{-1}\dz^2 v-v^{-2}(\dz v)^2]\big|
    \ls\beta^{\frac{3}{2}}.
\end{equation}
As we did in the proof of Lemma~\ref{v200-estimate}, we next
compute and estimate
\begin{equation}\label{dydzF_ell-estimate}
v^{-1}|\dy\dz F_\ell|\ls\beta^2
    +v^{-1}|\dy^2\dz v|+v^{-2}|\dy\dz^2 v|.
\end{equation}
Here we used inequality~\eqref{v_y-est} as well.

Using the inequalities above and estimate~\eqref{dyF_ell-estimate},
one sees, for example, that
\begin{align*}
E_{1,1,1,3}
    =&\,(v^{-1}\dy\dz F_3)(v^{-1}\dy\dz v)\\
    &+(v^{-1}\dz F_3)(v^{-1}\dy^2\dz v-v^{-2}\dy\dz v \dy v)\\
    &+(\dy F_3)(v^{-2}\dy\dz^2 v-v^{-3}\dy\dz v \dz v)
\end{align*}
may be estimated by
\[ |E_{1,1,1,3}|\ls
    \beta^{\frac{7}{2}}
    +\beta^{\frac{3}{2}}v^{-1}|\dy^2\dz v|
    +\beta^{\frac{3}{5}}v^{-2}|\dy\dz^2 v|. \]
In like fashion, we derive the estimates
\begin{align*}
|E_{1,1,1,1}|&\ls
    \beta^{\frac{21}{10}}
    +\beta^{\frac{3}{2}}|\dy^3 v|
    +\beta^{\frac{1}{2}}(v^{-1}|\dy^2\dz v|+v^{-2}|\dy\dz^2 v|),\\ \\
|E_{1,1,1,2}|&\ls
    \beta^3
    +\beta^{\frac{3}{2}}
        (v^{-1}|\dy^2\dz v|+v^{-2}|\dy\dz^2 v|)
    +\beta^{\frac{1}{2}}v^{-3}|\dz^3 v|,\\ \\
|E_{1,1,1,4}|&\ls
    \beta^2
    +\beta^{\frac{3}{2}}(v^{-1}|\dy^2\dz v| + v^{-2}|\dy\dz^2 v|).
\end{align*}
We omit some details which are entirely analogous to those shown
above. Finally, Condition~[C1] and estimate~\eqref{v_y-est} give
\[ |E_{1,1,1,5}|=2v^{-3}|\dy v||\dz v|\ls\beta^2. \]

Collecting the estimates above and using Condition~[C2], we
conclude that
\[\left|w\sum_{\ell=0}^5 E_{1,1,1,\ell}\right|
    \ls\beta^2|w|
    +\beta^3|\dy^3 v|
    +\beta^2 (v^{-1}|\dy^2\dz v|+v^{-2}|\dy\dz^2v|+v^{-3}|\dz^3 v|).\]
Using Remark~\ref{WhyBisGood} and completing squares to get
$\beta^2v^{-1}|\dy^2\dz v| -\ve v^{-2}|\dy^2\dz v|^2\leq\frac{1}{4\ve}\beta^4$
and
$\beta^2v^{-2}|\dy \dz^2 v|-\ve v^{-4} |\dy \dz^2 v|^2\leq\frac{1}{4\ve}\beta^4$,
we combine the estimates above to obtain the lemma.
\end{proof}

\begin{lemma}\label{v022-estimate}
There exist $0<\ve<C<\infty$ such that the
quantity $v^{-4}|\dz^2 v|^2$ satisfies
\begin{align*}
\dt(v^{-4}|\dz^2 v|^2)
    &\leq A_v(v^{-4}|\dz^2 v|^2)+Xv^{-4}|\dz^2 v|^2
    -Y\left(v^{-4}|\dy\dz^2 v|^2+v^{-6}|\dz^3 v|^2\right)\\
    &+C\left(\beta^4
        +\beta^{\frac{21}{10}}v^{-2}|\dz^2 v|
        +\beta^3 v^{-1}|\dy^2\dz v|\right),
\end{align*}
where in the inner region $\{\beta y^2\leq 20\}$,
\[X=2(3v^{-2}-a)+\ve\quad\text{and}\quad Y\geq\frac{1}{1+\ve},\]
while in the outer region,
\[X\leq\frac{3}{8}-2a+\ve\leq-\frac{1}{2}\quad\text{and}\quad Y\geq\ve.\]
\end{lemma}
\begin{proof}
Define $w:=v_{0,2,2}\equiv v^{-2}\dz^2 v$.
Then by equation~\eqref{Evolve-vmnk2} and estimate~\eqref{BisGood}, one has
\[\dt w^2\leq A_v(w^2)+Xw^2-Y\left(v^{-4}|\dy\dz^2 v|^2+v^{-6}|\dz^3 v|^2\right)
    +2\left|w\sum_{\ell=0}^5 E_{0,2,2,\ell}\right|,\]
where $X:=2(3v^{-2}-a)+\ve$ and $Y:=\frac{1}{1+p^2+q^2}$.
The estimates for $X$ in the outer region follow from estimate~\eqref{v-below}
for $v$ there, and Condition~[Ca].
The estimates for $Y$ are the same as those in Lemma~\ref{v200-estimate}.

To proceed, we compute
\begin{align*}
\dz^2 F_\ell
    =&\,(\partial_p^2 F_\ell)(\dy\dz v)^2\\
    &+2(\partial_p\partial_q F_\ell)(\dy\dz v)
        [v^{-1}\dz^2 v-v^{-2}(\dz v)^2]\\
    &+(\partial_q^2 F_\ell)[v^{-1}\dz^2 v-v^{-2}(\dz v)^2]^2\\
    &+(\partial_p F_\ell)(\dy\dz^2 v)
        +(\partial_q F_\ell)
        \dz [v^{-1}\dz^2 v-v^{-2}(\dz v)^2],
\end{align*}
which we estimate in the form
\begin{equation}\label{dz2F_ell-estimate}
v^{-2}|\dz^2F_\ell|\ls
    \beta^3+v^{-2}|\dy\dz^2 v|+v^{-3}|\dz^3 v|.
\end{equation}

By Conditions~[C1]--[C2], it is easy to see that
$|E_{0,2,2,0}|\ls\beta^{\frac{21}{10}}$.
Using [C1]--[C2], estimate~\eqref{dzF_ell-estimate}
for $v^{-1} |\dz F_1|$, and estimate~\eqref{dz2F_ell-estimate} for
$v^{-2}|\dz^2 F_1|$, one can estimate
\begin{align*}
E_{0,2,2,1}=&\,(v^{-2}\dz^2 F_1)(\dy^2 v)
        +2(v^{-1}\dz F_1)(v^{-1}\dy^2\dz v)\\
    &-F_1[2\dy(v^{-2})(\dy\dz^2 v)
        +\dy^2(v^{-2})\dz^2 v]
\end{align*}
by
\[|E_{0,2,2,1}|\ls   \beta^{\frac{21}{10}}
    +\beta^{\frac{3}{2}}v^{-1}|\dy^2\dz v|
    +\beta^{\frac{1}{2}}v^{-2}|\dy\dz^2 v|
    +\beta^{\frac{3}{5}}v^{-3}|\dz^3 v|.\]
Continuing in this fashion, one proceeds to estimate
the remaining terms,
\begin{align*}
|E_{0,2,2,2}|&\ls   \beta^{\frac{9}{2}}
    +\beta^{\frac{3}{2}}
    \left(v^{-2}|\dy\dz^2 v|+v^{-3}|\dz^3 v|\right),\\ \\
|E_{0,2,2,3}|&\ls   \beta^3
    +\beta^{\frac{3}{2}}v^{-2}|\dy\dz^2 v|
    +\beta^{\frac{1}{2}}v^{-3}|\dz^3 v|,\\ \\
|E_{0,2,2,4}|&\ls   \beta^3
    +\beta^{\frac{3}{2}}
    \left(v^{-2}|\dy\dz^2 v|+v^{-3}|\dz^3 v|\right),\\ \\
|E_{0,2,2,5}|&\ls   \beta^3.
\end{align*}

Collecting the estimates above and using Condition~[C2], we
conclude that
\[\left|w\sum_{\ell=0}^5 E_{0,2,2,\ell}\right|
    \ls\beta^{\frac{21}{10}}|w|
    +\beta^3v^{-1}|\dy^2\dz v|
    +\beta^2
        \left(v^{-2}|\dy\dz^2 v|+v^{-3}|\dz^3 v|\right).\]
The lemma again follows by recalling Remark~\ref{WhyBisGood}
and completing the squares.
\end{proof}

\subsection{Lyapunov functionals of second and third order}\label{Lyapunov23}

In this subsection, we derive differential inequalities satisfied by the
$\Omega_{m,n}$ with $2\leq m+n\leq 3$. In the next subsection, we treat the cases
$4\leq m+n\leq 5$. Our arguments are slightly different for  $m+n=2$ than they are
for $3\leq m+n\leq5$. But in all cases, we shall exploit the fact that as
$\Sigma\rightarrow\infty$, one has
\begin{equation}\label{gamma-decay}
\|\sigma^{-1}\dy^2\sigma\|_\infty=\mathcal{O}(\Sigma^{-1})
\quad\text{and}\quad
\|\sigma^{-1}\dy\sigma\|_\infty=\mathcal{O}(\Sigma^{-\frac{1}{2}}),
\end{equation}
where $\sigma(y):=(\Sigma+y^2)^{-\frac{3}{5}}$ is defined in
equation~\eqref{Define-sigma}.

\begin{remark}
To derive good differential inequalities for the $\Omega_{m,n}$, we shall
need to use \eqref{gamma-decay} by taking $\Sigma$ sufficiently large. See
in particular the estimates derived in display~\eqref{Commutators} below.
On the other hand, the fact that one can make $\rho$ arbitrarily small in
Lemmata~\ref{Omega11}--\ref{Omega30} and \ref{OmegaLevel4}--\ref{OmegaLevel5}
by making $\Sigma$ large is an interesting feature of the proof but is not
needed for the arguments elsewhere in this paper.
\end{remark}

To estimate the evolution of the second-order functionals
$\Omega_{m,n}$ with $m+n=2$, we can recycle the estimates
derived in Lemmata~\ref{v200-estimate}--\ref{v022-estimate}.

\begin{lemma}\label{Omega11}
There exist constants $\ve>0$ and $\rho=\rho(\Sigma)>0$,
with $\rho(\Sigma)\searrow0$ as $\Sigma\rightarrow\infty$, such that
\[\Dt\Omega_{1,1}\leq\
    -\ve\left(\Omega_{1,1}+\Omega_{2,1}+\Omega_{1,2}\right)
    +\rho\beta^2\left(\Omega_{1,1}^{\frac{1}{2}}
        +\beta\Omega_{3,0}^{\frac{1}{2}}+\Omega_{2,1}^{\frac{1}{2}}
        +\Omega_{1,2}^{\frac{1}{2}}
        +\Omega_{0,3}^{\frac{1}{2}}\right).\]
\end{lemma}
\begin{proof}
Let $w:=v_{1,1,1}\equiv v^{-1}\dy\dz v$.
Condition~[Cs] lets us apply dominated convergence to compute
$\frac{d}{d\tau}\Omega_{1,1}$, whence Lemma~\ref{ComputeOnceAndForAll} gives
\begin{align*}
\frac{1}{2}\frac{d}{d\tau}\Omega_{1,1}
    =&\,\int_{\mathbb{R}}\int_{\mathbb{S}^1} wA_v(w)\,\Dz\,\sigma\,\Dy\\
    &+\int_{\mathbb{R}}\int_{\mathbb{S}^1} (2v^{-2}-a)w^2\,\Dz\,\sigma\,\Dy\\
    &+\int_{\mathbb{R}}\int_{\mathbb{S}^1}
        \left(\sum_{\ell=0}^5 E_{1,1,1,\ell}\right)w\,\Dz\,\sigma\,\Dy\\
    =&: I_1+I_2+I_3.
\end{align*}
By Lemma~\ref{v111-estimate} and its proof, one has
\[I_2 \leq
    \int_{\beta y^2\leq 20}\int_{\mathbb{S}^1}
        (2v^{-2}-a)w^2\,\Dz\,\sigma\,\Dy
    +\left(\frac{1}{8}-a\right)\int_{\beta y^2\geq20}\int_{\mathbb{S}^1}
        w^2\,\Dz\,\sigma\,\Dy\]
and
\[|I_3| \leq C\beta^2 \int_{\mathbb{R}}\int_{\mathbb{S}^1}\left(
    |w|+\beta|\dy^3 v|+v^{-1}|\dy^2\dz v|
    +v^{-2}|\dy\dz^2v|+v^{-3}|\dz^3 v|\right)\Dz\,\sigma\,\Dy.\]
Using Cauchy--Schwarz, the latter estimate becomes
\[|I_3|\leq\rho\beta^2\left(\Omega_{1,1}^{\frac{1}{2}}
        +\beta\Omega_{3,0}^{\frac{1}{2}}+\Omega_{2,1}^{\frac{1}{2}}
        +\Omega_{1,2}^{\frac{1}{2}}
        +\Omega_{0,3}^{\frac{1}{2}}\right).\]
To complete the proof, we will estimate $I_1:=I_{1,1}+I_{1,2}$, where
\begin{align*}
I_{1,1}&:=\int_{\mathbb{R}}\int_{\mathbb{S}^1}
    w\left(F_1 \dy^2 w +v^{-2}F_2\dz^2 w +v^{-1}F_3\dy\dz w\right)\Dz\,\sigma\,\Dy,\\
I_{1,2}&:=\int_{\mathbb{R}}\int_{\mathbb{S}^1}
    w\left(v^{-2}F_4\dz w - ay\dy w\right)\Dz\,\sigma\,\Dy.
\end{align*}
Because $w$ and $\sigma$ are even functions of $y$, one may integrate by parts
in $y$ as well as $\theta$, whereupon applying Cauchy--Schwarz pointwise (as in
Corollary~\ref{GeneralEquation}) shows that $I_{1,1}\leq I_{1,3}+I_{1,4}$, where
\begin{align*}
    I_{1,3}&:=-\int_{\mathbb{R}}\int_{\mathbb{S}^1}
        \frac{(\dy w)^2+v^{-2}(\dz w)^2}{1+p^2+q^2}\,\Dz\,\sigma\,\Dy,\\
    I_{1,4}&:=\frac{1}{2}\int_{\mathbb{R}}\int_{\mathbb{S}^1}w^2
        \left[\sigma^{-1}\dy^2(\sigma F_1)+\dz^2(v^{-2}F_2)
        +\sigma^{-1}\dy\dz(\sigma v^{-1}F_3)\right]\Dz\,\sigma\,\Dy.
\end{align*}
One estimates
$(\dy w)^2\geq v_{2,1,1}^2-C\beta w^2$
and
$v^{-2}(\dz w)^2\geq v_{1,2,2}^2-C\beta^3 w^2$
using Theorem~\ref{FirstOrderEstimates} and Condition~[C1i].
Using those facts and Condition~[C0i], one has $p^2+q^2\leq C$ everywhere
and $p^2+q^2\leq\ve$ in the inner region. Hence
\begin{align*}
I_{1,3}
    \leq&-\ve\left(\Omega_{2,1}+\Omega_{1,2}\right)
        +C\beta\Omega_{1,1}\\
    &-4(1-2\ve)\left(\frac{1-\delta}{1+\delta}\right)^2
        \int_{\beta y^2\leq20}\int_{\mathbb{S}^1} v^{-2} w^2\,\Dz\,\sigma\,\Dy.
\end{align*}
To get the final term, we applied Condition~[Cr] to the decomposition
$v=v_1+v_2$ introduced in definition~\eqref{Decomposition} and then exploited
the key idea behind Remark~\ref{FixAxis}. Then combining Theorem~\ref{FirstOrderEstimates}
and Conditions~[C1i] and [C2] with
estimates~\eqref{dyF_ell-estimate}--\eqref{dydzF_ell-estimate}
and \eqref{gamma-decay} lets us control the terms in $I_{1,4}$ using
\begin{subequations}\label{Commutators}
\begin{align}
\sigma^{-1}|\dy^2(\sigma F_1)|&\leq C
    \left(|\dy^3 v|+v^{-1}|\dy^2\dz v|+\beta^{\frac{6}{5}}
        +\beta^{\frac{3}{5}}\Sigma^{-\frac{1}{2}}+\Sigma^{-1}\right),\\
|\dz^2(v^{-2}F_2)|&\leq C
    \left(v^{-2}|\dy\dz^2 v|+v^{-3}|\dz^3 v|+\beta^{\frac{3}{2}}\right),\\
\sigma^{-1}|\dy\dz(\sigma v^{-1}F_3)|&\leq C
    \left(v^{-1}|\dy^2\dz v|+v^{-2}|\dy\dz^2 v|
    +\beta^{\frac{3}{2}}+\beta^{\frac{3}{2}}\Sigma^{-\frac{1}{2}}\right).
\end{align}
\end{subequations}
Notice that here is where one needs $\Sigma$ to be sufficiently large.
To control $I_{1,2}$, one integrates by parts in $\theta$, combining
Conditions~[C1] and [C1i] with estimate~\eqref{dzF_ell-estimate} to obtain
\[\left|\int_{\mathbb{R}}\int_{\mathbb{S}^1} v^{-2}F_4 (w\dz w)\,\Dz\,\sigma\,\Dy\right|
    \leq C\beta^{\frac{3}{2}}\Omega_{1,1}.\]
Finally, one uses the fact that $y\dy\sigma\leq 0$ to get
\[-a\int_{\mathbb{R}}\int_{\mathbb{S}^1} yw\dy w\,\Dz\,\sigma\,\Dy
    \leq\frac{a}{2}\Omega_{1,1}.\]
The conclusion of the lemma follows by collecting the estimates above, using
the consequence of Condition~[C2] that $|w|\ls\beta^{\frac{3}{2}}$
pointwise. By Condition~[Ca], the coefficient multiplying
$\Omega_{1,1}$ can be chosen to be less than
$\frac{1}{8}-\frac{a}{2}\leq-\frac{1}{8}+(2\kappa)^{-1}$.
\end{proof}

\begin{lemma}\label{Omega02}
There exist constants $\ve>0$ and $\rho=\rho(\Sigma)>0$,
with $\rho(\Sigma)\searrow0$ as $\Sigma\rightarrow\infty$, such that
\[\Dt\Omega_{0,2}\leq\
    -\ve\left(\Omega_{0,2}+\Omega_{1,2}+\Omega_{0,3}\right)
    +\rho\beta^2\left(\beta^{\frac{1}{10}}\Omega_{0,2}^{\frac{1}{2}}
        +\beta\Omega_{2,1}^{\frac{1}{2}}+\Omega_{1,2}^{\frac{1}{2}}
        +\Omega_{0,3}^{\frac{1}{2}}\right).\]
\end{lemma}

\begin{proof}
Because Remark~\ref{FixAxis} applies to $v_{0,2,2}\equiv v^{-2}\dz^2 v$,
the proof is virtually identical, \emph{mutatis mutandis} to that of
Lemma~\ref{Omega11}, when one takes as input the estimates in
Lemma~\ref{v022-estimate}. The final observation is that the
coefficient multiplying $\Omega_{0,2}$ can be chosen less than
$\frac{3}{16}-\frac{a}{2}\leq-\frac{1}{16}+(2\kappa)^{-1}$.
We omit further details.
\end{proof}

\begin{lemma}\label{Omega20}
There exist constants $\ve>0$ and $\rho=\rho(\Sigma)>0$,
with $\rho(\Sigma)\searrow0$ as $\Sigma\rightarrow\infty$, such that
\[\Dt\Omega_{2,0}\leq\
    -\ve\left(\Omega_{2,0}+\Omega_{3,0}+\Omega_{2,1}\right)
    +\rho\left[\beta\Omega_{2,0}^{\frac{1}{2}}
    +\beta^{\frac{1}{2}}\Omega_{3,0}^{\frac{1}{2}}
    +\beta^{\frac{11}{10}}
        \left(\Omega_{2,1}^{\frac{1}{2}}+\Omega_{1,2}^{\frac{1}{2}}\right)\right].\]
\end{lemma}

\begin{proof}
The argument here contains two key differences from that used above.
The estimate for $I_2$ in Lemma~\ref{Omega11} is replaced by
\[I_2\leq\left(\frac{1}{16}-a\right)\Omega_{2,0}+2I_{2,2},\]
where
\begin{align*}
I_{2,2}&:=\int_{\mathbb{R}}\int_{\mathbb{S}^1} v^{-2}w^2\,\Dz\,\sigma\,\Dy\\
    &=\int_{\mathbb{R}}\int_{\mathbb{S}^1}
    (\dy v)\left[
    2v^{-3}(\dy^2 v)(\dy v)-v^{-2}(\dy^2v)\sigma^{-1}\dy\sigma-v^{-2}\dy^3 v
    \right]\,\Dz\,\sigma\,\Dy.
\end{align*}
Note that $\sigma^{-1}|\dy\sigma|=\mathcal{O}(\lan y\ran^{-1})$ as
$|y|\rightarrow\infty$. Hence one may apply Condition~[Cg] and
estimate~\eqref{gamma-decay} to see that
\[I_{2,2}\leq \rho\left(\beta\Omega_{2,0}^{\frac{1}{2}}
    +\beta^{\frac{1}{2}}\Omega_{3,0}^{\frac{1}{2}}\right).\]

Remark~\ref{FixAxis} does not apply to the function $\dy^2 v$, so
one gets only a weaker estimate for $I_{1,3}$, namely
\[I_{1,3}\leq-\ve\left(\Omega_{3,0}+\Omega_{2,1}\right)
        +\rho\left(\beta^2\Omega_{3,0}^{\frac{1}{2}}
        +\beta^3\Omega_{2,1}^{\frac{1}{2}}\right).\]

The remainder of the proof goes through as in Lemma~\ref{Omega11},
using the pointwise estimates derived in Lemma~\ref{v200-estimate}.
Here the coefficient multiplying $\Omega_{2,0}$ can be chosen less than
$\frac{1}{16}-\frac{a}{2}\leq-\frac{3}{16}+(2\kappa)^{-1}$.
\end{proof}
\medskip

We use a modified technique to derive differential inequalities
for $\Omega_{m,n}$ with $3\leq m+n\leq 5$. This technique applies
integration by parts to the nonlinear commutators, using
important but tedious estimates derived in Appendix~\ref{GetHigh}.
In the remainder of this subsection, we bound the functionals of order $m+n=3$.
We treat the cases $4\leq m+n\leq 5$ in Section~\ref{Lyapunov45} below.

\begin{lemma}\label{OmegaLevel3}
There exist constants $0<\ve<C<\infty$ and $\rho=\rho(\Sigma)>0$,
with $\rho(\Sigma)\searrow0$ as $\Sigma\rightarrow\infty$, such that
for all $m+n=3$ with $n\geq1$, one has
\begin{align*}
\Dt\Omega_{m,n}\leq&\,
    -\ve\left(\Omega_{m,n}+\Omega_{m+1,n}+\Omega_{m,n+1}\right)
    +C\beta^{\frac{3}{5}}\left(\Omega_{2,1}+\Omega_{1,2}+\Omega_{0,3}\right)\\
    &+\rho\beta^2\left[\Omega_{m,n}^{\frac{1}{2}}
    +\beta\Omega_{4,0}^{\frac{1}{2}}+\Omega_{3,1}^{\frac{1}{2}}
    +\Omega_{2,2}^{\frac{1}{2}}+\Omega_{1,3}^{\frac{1}{2}}
    +\Omega_{0,4}^{\frac{1}{2}}\right].
\end{align*}
\end{lemma}

\begin{proof}
Let $w:=v_{m,n,n}\equiv v^{-n}\dy^m\dz^n v$. Then using dominated convergence
and Lemma~\ref{ComputeOnceAndForAll}, one obtains
$\frac{1}{2}\,\frac{d}{d\tau}\Omega_{m,n}=I_1+I_2+I_3$, with terms $I_j$
defined below. We closely follow Lemma~\ref{Omega11} in estimating the
first two terms, indicating why the method used there applies here. We
use a somewhat different method to estimate the final term $I_3$.
For the first term, integration by parts yields
\begin{align*}
I_1&:=\int_{\mathbb{R}}\int_{\mathbb{S}^1}
    w\left(F_1 \dy^2 w +v^{-2}F_2\dz^2 w +v^{-1}F_3\dy\dz w
    +v^{-2}F_4\dz w - ay\dy w\right)\Dz\,\sigma\,\Dy\\
   &\leq-\ve\left(\Omega_{m+1,n}+\Omega_{m,n+1}\right)
   +\left(\frac{a}{2}+\ve\right)\Omega_{m,n}
   -(4-\ve)\int_{\beta y^2\leq 20}\int_{\mathbb{S}^1} v^{-2} w^2\,\Dz\,\sigma\,\Dy.
\end{align*}
Here we followed Lemma~\ref{Omega11}, estimating
\[(\dy w)^2+(\dz w)^2\geq v_{m+1,n,n}^2+v_{m,n+1,n+1}^2-C\beta w^2,\] using
Condition~[C3] to bound the third-order derivatives in display~\eqref{Commutators}
by $C\beta$, and again exploiting the idea behind Remark~\ref{FixAxis} to get
the useful integral over the inner region. By Lemma~\ref{ComputeOnceAndForAll} and
estimate~\eqref{v-below}, one has
\begin{align*}
I_2 &:=\int_{\mathbb{R}}\int_{\mathbb{S}^1} [(n+1)v^{-2}-2a]w^2\,\Dz\,\sigma\,\Dy\\
    &\leq\left(\frac{1}{4}-\frac{3a}{2}\right)\Omega_{m,n}+
    \left(4-\frac{a}{2}\right)\int_{\beta y^2\leq 20}\int_{\mathbb{S}^1}
        v^{-2} w^2\,\Dz\,\sigma\,\Dy.
\end{align*}
In Appendix~\ref{GetHigh}, we estimate
\[I_3 :=\int_{\mathbb{R}}\int_{\mathbb{S}^1}
        \left(\sum_{\ell=0}^5 E_{m,n,n,\ell}\right)w\,\Dz\,\sigma\,\Dy\]
using integration by parts. By Lemma~\ref{E21estimate} there,
one has
\[|I_3|\leq
\rho\beta^2\left[\Omega_{m,n}^{\frac{1}{2}}
    +\beta\Omega_{4,0}^{\frac{1}{2}}
    +\Omega_{3,1}^{\frac{1}{2}}
    +\Omega_{2,2}^{\frac{1}{2}}
    +\Omega_{1,3}^{\frac{1}{2}}
    +\Omega_{0,4}^{\frac{1}{2}}\right]
    +C\beta^{\frac{3}{5}}\left(\Omega_{2,1}+\Omega_{1,2}+\Omega_{0,3}\right).\]
Collecting these estimates while recalling Condition~[Ca] completes the proof.
\end{proof}

\begin{lemma}\label{Omega30}
There exist constants $0<\ve<C<\infty$ and $\rho=\rho(\Sigma)>0$,
with $\rho(\Sigma)\searrow0$ as $\Sigma\rightarrow\infty$, such that
\begin{align*}
\Dt\Omega_{3,0}\leq&\,
    -\ve\left(\Omega_{3,0}+\Omega_{4,0}+\Omega_{3,1}\right)
    +\rho\beta^{\frac{3}{2}}\left[
    \Omega_{3,0}^{\frac{1}{2}}
    +\beta^{\frac{1}{10}}\Omega_{4,0}^{\frac{1}{2}}
    +\Omega_{3,1}^{\frac{1}{2}}
    +\Omega_{2,2}^{\frac{1}{2}}\right]\\
    &+C\beta\left(\sum_{i+j=3}\Omega_{i,j}\right)
    +C\beta^{\frac{1}{2}}\Omega_{3,0}^{\frac{1}{2}}\Omega_{2,0}^{\frac{1}{2}}.
\end{align*}
\end{lemma}

\begin{proof}
We again write $\frac{1}{2}\,\frac{d}{d\tau}\Omega_{3,0}=I_1+I_2+I_3$
as in the proof of Lemma~\ref{OmegaLevel3}. The only differences
from the proof there are \textsc{(i)} we use Lemma~\ref{E30estimate} to
estimate $I_3\equiv E_{3,0}$, and \textsc{(ii)} we cannot apply
Remark~\ref{FixAxis} to extract a useful integral over the inner region
from $I_1$. However, that integral is not needed. Indeed, by Conditions~[C0i]
and [Ca], one has $v^{-2}-\frac{4a}{3}\leq\frac{50}{81}-\frac{2}{3}+\ve<0$.
(Compare Remark~\ref{OneReasonFor-g}.) Therefore,
\[I_2 :=\int_{\mathbb{R}}\int_{\mathbb{S}^1} (v^{-2}-2a){\V 30}^2\,\Dz\,\sigma\,\Dy
    \leq-\frac{2a}{3}\Omega_{3,0}.\]
The remainder of the proof exactly follows that of Lemma~\ref{OmegaLevel3}.
(Note that we deal with the term
$C\beta^{\frac{1}{2}}\Omega_{3,0}^{\frac{1}{2}}\Omega_{2,0}^{\frac{1}{2}}$
in the proof of Proposition~\ref{SharpLyapunov} below).\end{proof}

We are now ready to bound the Lyapunov functionals $\Omega_{m,n}$ with $2\leq m+n\leq3$.

\begin{proposition}\label{SharpLyapunov}
Suppose that a solution $v=v(y,\theta,\tau)$ of equation~\eqref{MCF-v}
satisfies Assumption~[A1] at $\tau=0$, as well as [Ca], [C0]--[C3],
[C0i]--[C1i], [Cg], [Cr], and [Cs] for $\tau\in [0,\tau_1]$. Then for
the same time interval, one has
\begin{equation}\label{SharpLyapunovEstimates}
\beta^2\Omega_{2,0}+\Omega_{1,1}+\Omega_{0,2}+\beta\Omega_{3,0}
    +\Omega_{2,1}+\Omega_{1,2}+\Omega_{0,3}\ls\beta^4.
\end{equation}
\end{proposition}

\begin{proof}
For $R>0$ to be chosen, define
\[\Upsilon:=R\left(\beta^2\Omega_{2,0}+\Omega_{1,1}+\Omega_{0,2}\right)
+\beta\Omega_{3,0}+\Omega_{2,1}+\Omega_{1,2}+\Omega_{0,3},\]
noting the $\beta$-weights on the first and fourth terms.
Observe that $\frac{d}{d\tau}\beta<0$. Thus by Lemmas~\ref{Omega20}
and \ref{Omega30}, respectively, there exist
$\ve_2,\ve_3>0$ such that
\[\frac{d}{d\tau}(\beta^2\Omega_{2,0})\leq
    -\ve_2(\beta^2\Omega_{2,0})
    +\rho\beta^2\left[(\beta^2\Omega_{2,0})^{\frac{1}{2}}
        +(\beta\Omega_{3,0})^{\frac{1}{2}}
        +\Omega_{2,1}^{\frac{1}{2}}
        +\Omega_{1,2}^{\frac{1}{2}}\right]\]
and
\begin{align*}
\frac{d}{d\tau}(\beta\Omega_{3,0})\leq&\,
    -\ve_3(\beta\Omega_{3,0})
    +\rho\beta^2\left[(\beta\Omega_{3,0})^{\frac{1}{2}}
        +\beta^{\frac{3}{5}}\Omega_{4,0}^{\frac{1}{2}}
        +\Omega_{3,1}^{\frac{1}{2}}
        +\Omega_{2,2}^{\frac{1}{2}}\right]\\
    &+C\beta\left[(\beta\Omega_{3,0})
        +\beta(\Omega_{2,1}+\Omega_{1,2}+\Omega_{0,3})\right]\\
    &+\left\{\frac{\ve_3}{2}(\beta\Omega_{3,0})
    +\frac{C^2}{2\ve_3}(\beta^2\Omega_{2,0})\right\},
\end{align*}
where one applies weighted Cauchy--Schwarz to
$C(\beta^2\Omega_{2,0})^{\frac{1}{2}}(\beta\Omega_{3,0})^{\frac{1}{2}}$
to obtain the terms in braces.

Set $R=\frac{C^2}{\ve_2\ve_3}$. Then collecting the remaining estimates in
Lemmata~\ref{Omega11}--\ref{OmegaLevel3} and applying Cauchy--Schwarz yet again
proves that there exist $\ve_0>0$ and $C_0<\infty$, both independent of $R$, such that
\[\frac{d}{d\tau}\Upsilon\leq
    -\ve_0\Upsilon+C_0\left(R\rho\beta^2 \Upsilon^{\frac{1}{2}}+\rho^2\beta^4\right),\]
whence it follows easily that
$\Upsilon\ls\beta^4$.
\end{proof}

\subsection{Lyapunov functionals of fourth and fifth order}\label{Lyapunov45}
In this subsection, we derive differential inequalities satisfied by the
$\Omega_{m,n}$ with $4\leq m+n\leq5$. As noted above, our  estimates for them
do not need to be sharp, which somewhat reduces the technical work needed to bound the relevant
nonlinear terms. This work, which is done in Lemmata~\ref{pq}--\ref{E5estimate} of
Appendix~\ref{GetHigh}, is also substantially reduced by the availability of
Proposition~\ref{SharpLyapunov}.

\begin{lemma}\label{OmegaLevel4}
There exist $0<\ve<C<\infty$ such that whenever $m+n=4$,
\begin{align*}
\Dt\Omega_{m,n}\leq&\,
    -\ve\left(\Omega_{m,n}+\Omega_{m+1,n}+\Omega_{m,n+1}\right)
    +C\beta^{\frac{1}{2}}\left(\sum_{4\leq i+j\leq 5}\Omega_{i,j}\right)\\
    &+C\beta^r\left(\sum_{4\leq i+j\leq 5}\Omega_{i,j}^{\frac{1}{2}}\right),
\end{align*}
where $r=\frac{8}{5}$ if $n=0$ and $r=2$ otherwise.
\end{lemma}

\begin{proof}
Following the proof of Lemma~\ref{OmegaLevel3}, the result follows easily when one
observes that
\begin{multline*}
\int_{\mathbb{R}}\int_{\mathbb{S}^1} [(n+1)v^{-2}-(m+n-1)a]v_{m,n,n}^2\,\Dz\,\sigma\,\Dy\\
    \leq\left(\frac{3}{8}-\frac{5a}{2}\right)\Omega_{m,n}+
    \left(4-\frac{a}{2}\right)\int_{\beta y^2\leq 20}\int_{\mathbb{S}^1}
        v^{-2} w^2\,\Dz\,\sigma\,\Dy
\end{multline*}
and applies the conclusion of Lemma~\ref{E4estimate} from Appendix~\ref{GetHigh}.
\end{proof}

\begin{lemma}\label{OmegaLevel5}
There exist $0<\ve<C<\infty$ such that whenever $m+n=5$,
\begin{align*}
\Dt\Omega_{m,n}\leq&\,
    -\ve\left(\Omega_{m,n}+\Omega_{m+1,n}+\Omega_{m,n+1}\right)
    +C\beta^{\frac{1}{2}}\left(\sum_{4\leq i+j\leq 6}\Omega_{i,j}\right)\\
    &+C\beta^{\frac{21}{10}}\left(\sum_{5\leq i+j\leq 6}\Omega_{i,j}^{\frac{1}{2}}\right).
\end{align*}
\end{lemma}

\begin{proof}
The proof exactly parallels that of Lemma~\ref{OmegaLevel3} when one observes that
\[\int_{\mathbb{R}}\int_{\mathbb{S}^1} [v^{-2}-(m-1)a]v_{m,0,0}^2\,\Dz\,\sigma\,\Dy
    \leq-\frac{5a}{3}\Omega_{m,0}\]
and applies Lemma~\ref{E5estimate} from Appendix~\ref{GetHigh}.
\end{proof}

\subsection{Estimates of second- and third-order derivatives}\label{InnerProof}

\begin{proof}[Proof of Theorem~\ref{InnerEstimates}]
As indicated above, the proof is in three steps.
\smallskip

\textsc{(Step i)} This was accomplished in Proposition~\ref{SharpLyapunov},
where we proved that
\[\beta^2\Omega_{2,0}+\Omega_{1,1}+\Omega_{0,2}+\beta\Omega_{3,0}
    +\Omega_{2,1}+\Omega_{1,2}+\Omega_{0,3}\ls\beta^4.\]

\textsc{(Step ii)} Define
\[\Upsilon:=\beta^{\frac{4}{5}}\Omega_{4,0}+\Omega_{3,1}+\Omega_{2,2}+\Omega_{1,3}+\Omega_{0,4}
    +\sum_{m+n=5}\Omega_{m,n},\]
noting the $\beta$-weight imposed on the first term.
By Lemmas~\ref{OmegaLevel4}--\ref{OmegaLevel5} and Cauchy--Schwarz, there exist
$0<\ve<C<\infty$ such that
\[\frac{d}{d\tau}\Upsilon\leq-\ve\Upsilon+C\left(\beta^2\Upsilon^{\frac{1}{2}}+\beta^4\right).\]
Assumption~[A6] bounds $\Upsilon$ at $\tau=0$. It follows that $\Upsilon\ls\beta^4$, namely that
\[\beta^{\frac{4}{5}}\Omega_{4,0}+\Omega_{3,1}+\Omega_{2,2}+\Omega_{1,3}+\Omega_{0,4}
    +\sum_{m+n=5}\Omega_{m,n}\ls\beta^4.\]

\textsc{(Step iii)}
Now that we have $L^2_\sigma$ bounds on derivatives of orders two through five, Sobolev embedding
gives pointwise bounds on derivatives of orders two through three. Because of the weighted norm
$\|\cdot\|_\sigma$ these pointwise bounds are not uniform in $y$. Using the facts
that $\sigma\sim\lan y \ran^{-\frac{6}{5}}$ as $|y|\rightarrow\infty$ and that
$|y|\ls\beta^{-\frac{1}{2}}$ in the inner region, one obtains
\[\beta^2(\dy^2 v)^2+v^{-2}(\dy\dz v)^2+v^{-4}(\dz^2 v)^2\ls
    \beta^{4-\frac{1}{20}-\frac{3}{5}}\]
and
\[\beta(\dy^3 v)^2+v^{-2}(\dy^2\dz v)^2+v^{-4}(\dy\dz^2 v)^2+v^{-6}(\dz^3 v)^2
    \ls\beta^{4-\frac{1}{20}-\frac{3}{5}}\]
in the inner region. These inequalities are equivalent to
estimates~\eqref{InnerSecondOrder}--\eqref{InnerThirdOrder}.
\end{proof}

\section{Estimates in the outer region}\label{ProveLast}

\subsection{Second-order decay estimates}

As noted above, it is easy in the outer region to ``close the loop'' by estimating
second-order derivatives without needing assumptions on higher derivatives.

\begin{proof}[Proof of estimate~\eqref{OuterSecondOrder}]
By Lemmata~\ref{v200-estimate}--\ref{v022-estimate} and Cauchy--Schwarz,
there exist constants $C_1,C_2,C_3$ such that in the outer region, one has
\begin{align*}
\dt|\dy^2 v|^2
    &\leq A_v(|\dy^2 v|^2)-\frac{5}{8}|\dy^2 v|^2
    -\ve\left(|\dy^3 v|^2+v^{-2}|\dy^2\dz v|^2\right)\\
    &+C_1\left(\beta^2
                +\beta^{\frac{11}{10}}v^{-2}|\dy\dz^2 v|\right),
\end{align*}
and
\begin{align*}
\dt(v^{-2}|\dy\dz v|^2)
    &\leq A_v(v^{-2}|\dy\dz v|^2)-\frac{1}{2}v^{-2}|\dy\dz v|^2
    -\ve\left(v^{-2}|\dy^2\dz v|^2+v^{-4}|\dy\dz^2 v|^2\right)\\
    &+C_2\left(\beta^4
        +\beta^3|\dy^3 v|+\beta^2 v^{-3}|\dz^3 v|\right),
\end{align*}
and
\begin{align*}
\dt(v^{-4}|\dz^2 v|^2)
    &\leq A_v(v^{-4}|\dz^2 v|^2)-\frac{3}{8}v^{-4}|\dz^2 v|^2
    -\ve\left(v^{-4}|\dy\dz^2 v|^2+v^{-6}|\dz^3 v|^2\right)\\
    &+C_3\left(\beta^4+\beta^3 v^{-1}|\dy^2\dz v|\right).
\end{align*}

Define
\[ \Upsilon :=\beta^2|\dy^2 v|^2+v^{-2}|\dy\dz v|^2+v^{-4}|\dz^2 v|^2,\]
noting the $\beta$-weight imposed on the first term.
Then because $\frac{d}{d\tau}\beta<0$, there exists $C_0<\infty$ such that
\[\dt\Upsilon\leq A_v(\Upsilon)-\frac{3}{8}\Upsilon
    +C_0\beta^4+E,\]
where
\begin{align*}
E   &:= \left(C_2\beta^3|\dy^3 v|
        -\ve \beta|\dy^3 v|^2\right)
    +\left(C_3\beta^3 v^{-1}|\dy^2\dz v|
        -\ve v^{-2}|\dy^2\dz v|^2\right)\\
    &+\left(C_1\beta^{\frac{31}{10}} v^{-2}|\dy\dz^2 v|
        -\ve v^{-4}|\dy\dz^2 v|^2\right)
    +\left(C_2\beta^2 v^{-3}|\dz^3 v|
        -\ve v^{-6}|\dz^3 v|^2\right).
\end{align*}
Cauchy--Schwarz proves that $E\ls\beta^4$.
Hence there exists $C<\infty$ such that
\[\dt(\Upsilon-\Gamma\beta^{\frac{33}{10}})
    \leq A_v(\Upsilon-\Gamma\beta^{\frac{33}{10}})
    -\frac{3}{8}(\Upsilon-\Gamma\beta^{\frac{33}{10}})
    +\left[\frac{33}{10}\beta+\frac{C}{\Gamma}\beta^{\frac{7}{10}}-\frac{3}{8}\right]
    \Gamma\beta^{\frac{33}{10}}.\]
Choosing $\kappa$ and $\Gamma$ sufficiently large ensures
both that the quantity in brackets above is negative, and
that $\Upsilon\leq\Gamma\beta^{\frac{33}{10}}$ at $\tau=0$.
By Theorem~\ref{InnerEstimates}, one has $\Upsilon\leq\Gamma\beta^{\frac{33}{10}}$
on the rest of the parabolic boundary. Hence the result follows from
the maximum principle in the form of Proposition~\ref{PMP}.
\end{proof}

\subsection{Third-order decay estimates}\label{ProveThird}

In this subsection, we complete the proof of Theorem~\ref{OuterEstimates}
by establishing estimate~\eqref{OuterThirdOrder} in the outer region.

As we did in the previous section, we begin by deriving differential
inequalities satisfied by the squares of the quantities that appear
in \eqref{OuterThirdOrder}. These estimates do not have to be sharp,
because we only need relatively weak third-order estimates for our
application in Section~\ref{SecondBootstrap}. To keep the notation from becoming
too cumbersome, we consistently write $(v_{m,n,k})$ for
$(v^{-k}\dy^m \dz^n v)$ in the remainder of this section.

\begin{lemma}\label{Outer-v30}
There exist $0<\ve<C<\infty$ such that in the outer region,
\begin{align*}
\dt{\V 30}^2\leq&\, A_v({\V 30}^2)-\frac{7}{4}{\V 30}^2
    -\ve\left[{\V 40}^2+{\V 31}^2\right]\\
    &+C\left(\beta^3+\beta^{\frac{6}{5}}|\V 30|+\beta^{\frac{3}{2}}|\V 22|
        +\beta\sum_{m+n=3}v_{m,n,n}^2\right).
\end{align*}
\end{lemma}

\begin{proof}
Following the proof of Lemma~\ref{v200-estimate}, one applies Condition~[Ca]
and estimates~\eqref{v-below}, \eqref{v_y-est}, and \eqref{v_z-est} to
equations~\eqref{Evolve-vmnk2}--\eqref{BisGood} to bound the linear terms above.
So it suffices to derive a pointwise bound for $|\V 30\sum_{\ell=1}^5 E_{3,0,0,\ell}|$.

For this, we can use some but not all of the estimates derived in
Lemma~\ref{E30estimate}. Four terms there were estimated using
integration by parts. Here, we instead combine the pointwise bounds
derived in Lemma~\ref{Fell3} with
estimates~\eqref{v200-estimate}--\eqref{v022-estimate}
to obtain the bounds
\[|\V30(\dy^3 F_1)\V 20|\ls\beta^{\frac{7}{10}}|\V 30|
    \left(\beta^{\frac{8}{5}}+|\V 40|+|\V 31|\right)\]
and
\[|\V30(\dy^3 F_2)\V 02|\ls\beta^{\frac{3}{2}}|\V 30|
    \left(\beta^{\frac{8}{5}}+|\V 40|+|\V 31|\right),\]
with similar estimates holding for the contributions
from $F_3$ and $F_4$.
Again using estimates~\eqref{v200-estimate}--\eqref{v022-estimate}
instead of the method used in Lemma~\ref{E30estimate},
we obtain the critical bound
\[|\V30 E_{3,0,0,5}|\ls|\V 30|\,|\V 20 \V 10+{\V 10}^3|\ls
    \beta^{\frac{6}{5}}|\V 30|.\]
The result follows using Cauchy--Schwarz and our assumption
$|\V 30|\ls\beta$.
\end{proof}

\begin{lemma}\label{Outer-vmn}
There exist $0<\ve<C<\infty$ such that for $m+n=3$ with $n\geq1$,
one estimates in the outer region that
\begin{align*}
\dt(v_{m,n,n})^2\leq&\, A_v((v_{m,n,n})^2)-\frac{5}{4}(v_{m,n,n})^2
    -\ve\left[(v_{m+1,n,n})^2+(v_{m,n+1,n+1})^2\right]\\
    &+C\beta^2\left(|v_{m,n,n}|+\beta|v_{4,0,0}|+|v_{3,1,1}|+|v_{2,2,2}|
        +|v_{1,3,3}|+|v_{0,4,4}|\right)\\
    &+C\beta^{\frac{3}{5}}[(v_{2,1,1})^2+(v_{1,2,2})^2+(v_{0,3,3})^2].
\end{align*}
\end{lemma}

\begin{proof}
As in the proof of Lemma~\ref{Outer-v30}, we may apply Condition~[Ca]
and estimates~\eqref{v-below}, \eqref{v_y-est}, and \eqref{v_z-est} to
equations~\eqref{Evolve-vmnk2}--\eqref{BisGood} in order to bound the
linear terms above.

To derive a pointwise bound for $|(v_{m,n,n})\sum_{\ell=0}^5 E_{m,n,n,\ell}|$,
we use some but not all of the estimates derived in Lemma~\ref{E21estimate}.
To replace the estimates that were obtained there using integration by parts,
we again combine the pointwise bounds derived in Lemma~\ref{Fell3} with
estimates~\eqref{v200-estimate}--\eqref{v022-estimate}, obtaining, for example,
\[|(v_{m,n,n})(v^{-n}\dy^m\dz^n F_1)\V 20|\ls\beta^{\frac{7}{10}}|(v_{m,n,n})|
    \left(\beta^2+|(v_{m+1,n,n})|+|(v_{m,n+1,n+1})|\right),\]
with stronger estimates holding for the contributions from $F_\ell$, $\ell=2,3,4$.
Here, the term $| (v_{m,n,n})E_{m,n,n,5}|$ is easy to estimate. Thus the result
again follows from Cauchy--Schwarz and our assumption that
$|(v_{m,n,n})|\ls\beta^{\frac{3}{2}}$.
\end{proof}

Now we are ready to prove our third-order decay estimates in the outer region.

\begin{proof}[Proof of estimate~\eqref{OuterThirdOrder}]
Define
\[\Upsilon:=\beta{\V30}^2+{\V 21}^2+{\V 12}^2+{\V 03}^2.\]
Then, just as in the proof of estimate~\eqref{OuterSecondOrder}, it
follows from Lemmas~\ref{Outer-v30}--\ref{Outer-vmn} and Cauchy--Schwarz that
\begin{align*}
\dt\Upsilon&\leq A_v(\Upsilon)-\Upsilon
    +C_0\left(\beta^{\frac{17}{10}}\Upsilon^{\frac{1}{2}}+\beta^4\right)\\
    &\leq A_v(\Upsilon)-\frac{7}{8}\Upsilon+C\beta^{\frac{34}{10}}.
\end{align*}
By hypothesis, one has $\Upsilon\ls\beta^{\frac{33}{10}}$ on the parabolic
boundary of the outer region. Thus applying the parabolic maximum
principle in the form of Proposition~\ref{PMP}, exactly as in the proof
of estimate~\eqref{OuterSecondOrder}, shows that
$\Upsilon\ls\beta^{\frac{33}{10}}$ throughout the outer region for
$\tau\in[0,\tau_1]$.
\end{proof}

This completes our proof of Theorem~\ref{OuterEstimates}.

\subsection{Third-order smallness estimates}\label{Smallness}
Now we prove Theorem~\ref{SmallnessEstimates}, whose
purpose is to bound $|\dy\dz^2 v|$ and $v^{-1}|\dz^3 v|$ for use
in Section~\ref{SecondBootstrap} below.\footnote{Note that Lemma~\ref{InterpEst}
lets us control $|\dy\dz v|$ and $v^{-1}|\dz^k v|$ for $k=1,2$ by
bounding $|\dy\dz^2 v|$ and $v^{-1}|\dz^3 v|$, respectively.}

Define
\begin{equation}\label{eq:newUps}
\Upsilon:= (v^{-1}\dz^3 v)^2+(\dy\dz^2 v)^2
    +\beta^{-\frac{11}{10}}\left[(\dy^2 \dz v)^2+(\dy^3 v)^2\right].
\end{equation}
The factor $\beta^{-\frac{11}{10}}$ will be used in estimate~\eqref{eq:D1} below.
In this section, we prove:
\begin{proposition}\label{prop:small}
There exists a constant $C$ such that in the outer region $\beta y^2\geq 20$,
\[\dt\Upsilon\leq A_{v}\Upsilon+C\beta^{\frac{11}{10}}.\]
\end{proposition}

We claim that Theorem~\ref{SmallnessEstimates} is an easy corollary of this result.
Indeed, by our Main Assumptions, one has
\[\Upsilon(\cdot,\cdot,0)\leq\ve_0\]
for some $\ve_0\ll 1$. In the inner region $\beta y^2\leq 20$, one combines the
consequence $v\leq C_0$ of Condition~[C0i] with the results of Theorem~\ref{InnerEstimates}
to see that
\[\Upsilon|_{\beta y^2\leq 20}\ls \beta^{\frac{1}{2}}.\]
Together, these estimates control $\Upsilon$ on the parabolic boundary of the outer region.
Thus by Proposition~\ref{prop:small} and the parabolic maximum principle, one has
\begin{align*}
\Upsilon(\cdot,\cdot,\tau)
    &\leq\ve_0+\beta^{\frac{1}{2}}(\tau)
    +C\int_{0}^{\tau}\beta^{\frac{11}{10}}(\tau)\,\mathrm{d}\tau\\
    &\ls(\ve_0+\beta_0)^{\frac{1}{10}},
\end{align*}
which proves Theorem~\ref{SmallnessEstimates}.

\begin{proof}[Proof of Proposition~\ref{prop:small}]
We start by studying $\Upsilon_2:=(\dy\dz^2 v)^2$, since it forces us to
define $\Upsilon$ as in ~\eqref{eq:newUps}. Apply Corollary~\ref{GeneralEquation}
to obtain
\[\dt\Upsilon_2=A_{v}\Upsilon_2+2v^{-2}\Upsilon_2-B_{1,2,0}+2 E_{1,2,0}(\dy \dz^2 v).\]
For the term $2v^{-2}\Upsilon_2$, we write
$v^{-1}|\dy\dz^2 v|\leq (v^{-2}|\dy\dz^2 v|)^{\frac{1}{2}}|\dy\dz^2 v|^{\frac{1}{2}}$
and employ Theorem~\ref{OuterEstimates} and the smallness estimate in Condition~[C3] to estimate
\[2v^{-2}|\dy\dz^2 v|^2\leq \beta^{\frac{3}{2}}.\]
In this way, we derive the differential inequality
\begin{equation}\label{eq:Upsilon2Sec}
\dt\Upsilon_2\leq A_{v}\Upsilon_2
-B_{1,2,0}+2 E_{1,2,0}(\dy \dz^2 v)+\beta^{\frac{3}{2}}.
\end{equation}
Next we consider the term $E_{1,2,0}(\dy \dz^2 v)$. In the present situation, the difficulties
encountered are similar to those surmounted in the proof of Theorem~\ref{OuterEstimates}; we
handle these in a similar manner.

We treat the highest-order terms first.
Among the many terms that make up $E_{1,2,0}$, some contain a factor (at most one) of
$\dy^m\dz^n v$ with $m+n=4$. If $(m,n)=(2,2),\ (1,3)$, this factor
is controlled by the favorable term $-B_{1,2,0}$ in inequality~\eqref{eq:Upsilon2Sec}.
The difficult cases are $(m,n)=(0,4),\ (3,1)$. There are only two such terms in
$E_{1,2,0}(\dy \dz^2 v)$, namely
\begin{align*}
D_1 &:=(\dy\dz^2 v)(\dz  F_1)(\dy^3\dz v)\\
    &=(\dy\dz^2 v)[(\partial_p F_1)(\dy\dz v)+(\partial_q F_1)(\dz q)](\dy^3\dz v)
\end{align*}
and
\begin{align*}
D_2 &:=(\dy\dz^2 v)(\dy F_2)(v^{-2}\dz^4 v)\\
    &= (\dy\dz^2 v)[(\partial_p F_2)(\dy^2 v)+(\partial_q F_2)(\dy q)](v^{-2}\dz^4 v).
\end{align*}

For $D_1$, combining Lemma~\ref{InterpEst} with the smallness estimate in Condition~[C3]
shows that the terms in brackets admit the estimate
\[ |(\partial_p F_1)(\dy\dz v)+(\partial_q F_1)(\dz q)|
    \leq\frac{(\beta_0+\ve_0)^{\frac{1}{20}}}{1+p^2+q^2}.\]
Using Cauchy--Schwarz and the smallness assumption again, one thus obtains
\begin{align}\label{eq:D1}
|D_1|   &\leq (\beta_0+\ve_0)^{\frac{1}{20}}
        \left[\beta^{\frac{11}{10}}|\dy\dz^2 v|^2+\frac{\beta^{-\frac{11}{10}}(\dy^3\dz v)^2}{1+p^2+q^2}\right]
        \nonumber\\
    &\leq (\beta_0+\ve_0)^{\frac{1}{10}}\beta^{\frac{11}{10}}
    +(\beta_0+\ve_0)^{\frac{1}{20}}\beta^{-\frac{11}{10}}B_{2,1,0},
\end{align}
where $\beta^{-\frac{11}{10}}B_{2,1,0}$ appears in the evolution equation for
$\beta^{-\frac{11}{10}}(\dy^2 \dz v)^2$ below.

For $D_2$, we proceed differently. Here, the decay estimate comes from the terms in brackets.
By the estimates in Theorem~\ref{OuterEstimates} and the interpolation Lemma~\ref{InterpEst},
one has
\[
|(\partial_p F_2)(\dy^2 v)+(\partial_q F_2)(\dy q)|
    \ls\frac{|\dy^2 v|+|\dy q|}{1+p^2+q^2}\\
    \ls \frac{\beta^{\frac{13}{20}}}{1+p^2+q^2}.
\]
Hence
\begin{equation}\label{eq:D2}
|D_2|\leq (\beta_0+\ve_0)^{\frac{1}{10}}\beta^{\frac{11}{10}}
    +\frac{(\beta_0+\ve_0)^{\frac{1}{20}}}{1+p^2+q^2} B_{0,3,1},
\end{equation}
where $B_{0,3,1}$ appears in the evolution equation for $v^{-2}(\dz^3 v)^2$ below.

The remaining terms in inequality~\eqref{eq:Upsilon2Sec} are estimated by techniques
very similar to those employed in the proof of Theorem~\ref{OuterEstimates}. Those methods
may be used because we have available the smallness estimates from Condition~[C3].
Together with estimate~\eqref{eq:D1} and estimate~\eqref{eq:D2}, we apply these techniques
here (omitting further details) to obtain
\begin{align}\label{eq:Upsilon2Fin}
\dt\Upsilon_2\leq A_{v}\Upsilon_2-\frac{1}{2}B_{1,2,0}
    +\frac{B_{0,3,0}+\beta^{-\frac{11}{10}}B_{2,1,0}}{1+p^2+q^2}+\beta^{\frac{11}{10}}.
\end{align}
\smallskip

Now we turn to $\Upsilon_3:=\beta^{-\frac{11}{10}}(\dy^2 \dz v)^2$.
We use Corollary ~\ref{GeneralEquation} to write
\[\dt\Upsilon_3=A_{v}\Upsilon_3
    +2\left[v^{-2}-a+\beta^{\frac{11}{10}}(\dt\beta^{-\frac{11}{10}})\right]\Upsilon_3
    -\beta^{-\frac{11}{10}}B_{2,1,0}
    +2\beta^{-\frac{11}{10}} E_{2,1,0}(\dy^2 \dz v).\]
From this, we use the implication of Theorem ~\ref{OuterEstimates} that
$v^{-2}(\dy^2 \dz v)^2\ls\beta^{\frac{33}{10}}$ and the fact that
$\beta^{\frac{11}{10}}(\dt\beta^{-\frac{11}{10}})=\frac{11}{10}\beta\ll\frac{1}{10}$
to derive the differential inequality
\begin{equation}\label{eq:Upsilon3sec}
\dt\Upsilon_3   \leq A_{v}\Upsilon_3-\frac{1}{2}\Upsilon_3-\beta^{-\frac{11}{10}}B_{2,1,0}
    +2\beta^{-\frac{11}{10}} E_{2,1,0}(\dy^2 \dz v+\beta^2).
\end{equation}
This differs from the corresponding inequality~\eqref{eq:Upsilon2Sec}
for $\dt\Upsilon_2$ in by the term $-\frac{1}{2}\Upsilon_3$.

As in the derivation of inequality~\eqref{eq:Upsilon2Fin}, there are two critical terms,
\[W_1:=\beta^{-\frac{11}{10}}(\dy^2\dz v)(\dz F_1)(\dy^4 v)\]
and
\[W_2:=\beta^{-\frac{11}{10}}(\dy^2\dz v)(v^{-1}\dy F_2)(v^{-1}\dy\dz^3 v),\]
which appear in $\beta^{-\frac{11}{10}} E_{2,1,0}(\dy^2 \dz v).$

For $W_1$, a suitable estimate for $\dz F_1$, namely
\[
|\dz F_1|   \ls |\partial_p F_1||\dy\dz v|+|\partial_q F_1||\dz q|
            \ls\frac{(\beta_0+\ve_0)^{\frac{1}{20}}}{1+p^2+q^2},\]
follows from the smallness component of Condition~[C3]. Applying
Cauchy--Schwarz then yields
\begin{align}\label{eq:w1}
|W_1|&\leq (\beta_0+\ve_0)^{\frac{1}{20}}\beta^{-\frac{11}{10}}
        \frac{(\dy^2\dz v)^2+(\dy^4 v)^2}{1+p^2+q^2}\nonumber\\
     &\leq (\beta_0+\ve_0)^{\frac{1}{20}}
        \left[\Upsilon_3+\beta^{-\frac{11}{10}}B_{3,0,0}\right],
\end{align}
where $\beta^{-\frac{11}{10}}B_{3,0,0}$ appears in the evolution equation
for $\beta^{-\frac{11}{10}}(\dy^3 v)^2$.

For $W_2$, we get good decay from the term $v^{-1}\dy F_2$.
By Theorem~\ref{OuterEstimates}, one has
\[
|v^{-1}\dy F_2|\leq\frac{\beta^{\frac{3}{5}}}{1+p^2+q^2},\]
and hence
\begin{equation}\label{eq:w2}
|W_2|\leq \beta^{-\frac{1}{2}}\frac{|\dy^2\dz v|\,|v^{-1}\dy\dz^3 v|}{1+p^2+q^2}
    \leq \beta^{\frac{1}{20}}\left[\Upsilon_3+B_{1,2,0}\right],
\end{equation}
where $B_{1,2,0}$ appears in inequality~\eqref{eq:Upsilon2Fin}.

The remaining terms in inequality~\eqref{eq:Upsilon3sec} are estimated by the same
methods we employed many times in previous sections. Combining these estimates with
inequalities~\eqref{eq:w1}--\eqref{eq:w2} yields
\begin{equation}\label{eq:Upsilon3Fin}
\dt\Upsilon_3\leq A_{v}\Upsilon_3-\frac{1}{2}\beta^{-\frac{11}{10}}B_{2,1,0}
        +\beta^{\frac{1}{20}}B_{1,2,0}
        +(\beta_0+\ve_0)^{\frac{1}{20}}\beta^{-\frac{11}{10}}B_{3,0,0}+\beta^{\frac{11}{10}},
\end{equation}
where $-\frac{1}{2}\Upsilon_3$ was used to control various terms in $W_1$ and $W_2$.
\smallskip

By similar techniques, we estimate that $\Upsilon_1:=(v^{-1}\dz^3 v)^2$ and
$\Upsilon_4:=\beta^{-\frac{11}{10}}(\dy^3 v)^2$ satisfy the differential inequalities
\begin{equation}\label{eq:Upsilon1Fin}
\dt\Upsilon_1\leq A_{v}\Upsilon_1-\frac{1}{2}B_{0,3,1}
    +(\beta_0+\ve_0)^{\frac{1}{20}}B_{1,2,0}+\beta^{\frac{11}{10}}
\end{equation}
and
\begin{equation}\label{eq:Upsilon4Fin}
\dt\Upsilon_4\leq A_{v}\Upsilon_4-\frac{1}{2}\beta^{-\frac{11}{10}}B_{3,0,0}
    +(\beta_0+\ve_0)^{\frac{1}{20}}\beta^{-\frac{11}{10}}B_{2,1,0}+\beta^{\frac{11}{10}},
\end{equation}
respectively.
Adding equations~\eqref{eq:Upsilon1Fin}, \eqref{eq:Upsilon2Fin}, \eqref{eq:Upsilon3Fin},
and \eqref{eq:Upsilon4Fin} completes the proof.
\end{proof}


\section{The second bootstrap machine}\label{sec:asymp}
\label{SecondBootstrap}

In this section and those that follow, we describe the asymptotic behavior of solutions.
Specifically, we show that a solution $v$ to equation~\eqref{MCF-v}, which is a rescaling
of a solution $u$ to equation~\eqref{MCF-u}, may be decomposed into a slowly-changing main component
and a rapidly-decaying small component. We accomplish this by building a second bootstrap machine,
following \cite{GS09}.

\subsection{Input}\label{SBMI}
Our first input to this bootstrap machine is  that $u(x,t)$ is a solution of
equation~\eqref{MCF-u} satisfying the following conditions:
\begin{itemize}
\item[{[Cd]}]
There exists $t_{\#}>0$ such that for $0\le t\le t_{\#}$,
there exist $C^1$ functions $a(t)$ and $b(t)$ such that
$u(x,\theta, t)$ admits the decomposition
\begin{equation}\label{eqn:split2}
u(x,\theta,t) = \lambda(t) v(y,\theta,\tau) = \lambda (t)
\left[
\left(\frac{2+b(t)y^{2}}{a(t) +\frac12}\right)^{\frac12}
+ \phi (y,\theta,\tau)
\right]
\end{equation}
with the $L^2$ orthogonality properties
$$\phi (\cdot,\cdot, \tau) \perp\,e^{-\frac{a(t)}2  y^2},\
\left(1-a(t) y^2 \right) e^{-\frac{a(t)}2  y^2},$$
where $a(t):= -\lambda (t)\partial_t \lambda(t)$,
$y:=\lambda^{-1}(t)x$, and
$\tau(t) := \int_0^t \lambda^{-2} (s)\,\mathrm{d}s$.
\end{itemize}
Proposition~\ref{IFT} will show that Condition~[Cd] follows from our Main Assumptions.
\medskip

To state the second set of inputs, we define estimating functions to control the quantities
$\phi(y,\theta,\tau)$, $a(t(\tau))$, and $b(t(\tau))$ appearing in equation~\eqref{eqn:split2}, namely:
\begin{align}
M_{m,n}(T) := &\displaystyle\max_{\tau \le T}
\beta^{-\frac{m+n}2 - \frac1{10}}  (\tau)
\|\phi(\cdot,\cdot,\tau)\|_{m,n}, \label{eq:majorMmn}\\
\noalign{\vskip6pt}
A(T) := &\displaystyle \max_{\tau\le T}  \beta^{-2} (\tau)
\left| a(t(\tau)) - \frac12 + b(t(\tau))\right|, \label{eq:majorA}\\
\noalign{\vskip6pt}
B(T) :=& \displaystyle\max_{\tau\le T} \beta^{-\frac32}  (\tau) |b(t(\tau)) - \beta(\tau)|.\label{eq:majorB}
\end{align}
where $(m,n)\in\left\{(3,0),\ (11/10,0),\ (2,1), \ (1,1)\right\}$.
Here we used the definitions
$$\|\phi\|_{m,n} := \left\| \langle y\rangle^{-m} \dy^n \phi\right\|_{L^\infty}
\quad\text{and}\quad\beta(\tau) := \frac1{\frac1{b(0)} + \tau}.$$
\smallskip

By standard regularity theory for quasilinear parabolic equations, if the
initial data satisfy the Main Assumptions in Section~\ref{Basic} for $b_0$
and $c_0$ sufficiently small, then (making $t_{\#}>0$ smaller if necessary)
the solution will satisfy the second set of inputs for this bootstrap argument, namely:
\begin{itemize}
\item[{[Cb]}] For any $\tau \le \tau(t_{\#})$, one has
\begin{align}\label{eq:assuMAB}
A(\tau)+B(\tau)+|M(\tau)|\ls \beta^{-\frac1{20}} (\tau),
\end{align}
\end{itemize}
where $M$ denotes the vector
\begin{equation}\label{eq:vector}
M := (M_{i,j}), \ (i,j)\in\left\{(3,0),\ (11/10,0),\ (2,1),\ (1,1)\right\}.
\end{equation}
\begin{remark}\label{NoCircleHere}
Condition~[Cb] implies estimates on $v$ and $v^{-\frac12} \dy v$ in the
inner region $\beta y^2\le  20$, which in turn imply Condition~[C0i], the estimate
$|\dy v|\ls\beta^{\frac12} v^{\frac12}$ of Condition~[C1i], and the estimate
$\langle y\rangle^{-1} |\dy v |\ls \beta$ of Condition~[Cg].
\end{remark}
\smallskip

The final inputs to this bootstrap machine are the following estimates. They
follow from the outputs of the first bootstrap machine, as summarized in
Section~\ref{FirstOutput}. For any $\tau \in [0, \tau(t_{\#})]$, the estimates
proved in Sections~\ref{FirstBootstrap}--\ref{ProveLast} show that $v$ satisfies
\begin{align}\label{eq:lower}
v(y,\theta,\tau) \ge 1;
\end{align}
and that there exist constants $\epsilon_0\ll 1$ and $C$, independent of $\tau(t_{\#})$, such that
\begin{align}
|\dy  v| \le  & C; \label{eq:upperBounddV}\\
|v^{-1} \dz^2 v|,\ |\dy \dz^2 v| \le & \epsilon_0 \ll 1;
\label{eq:smallness}\\
v^{-1} |\dy v| \le C\beta^{\frac12},\ |\dy^2 v| \le  C\beta^{\frac{13}{20}},
&\ |\dy^3  v|  \le  C\beta^{\frac{23}{20}}; \label{eq:yDe}\\
v^{-2} |\dy \dz^2 v|, \ v^{-1} |\dy \dz v|,\ & v^{-1}
|\dy^2 \dz v|, \ v^{-2}|\dz^2 v|\le C \beta^{\frac{33}{20}}.\label{eq:thetaDe}
\end{align}

\subsection{Output}
The main result of this section is:

\begin{theorem}\label{THM:aprior}
Suppose that Conditions~[Cd] and [Cb] and estimates~\eqref{eq:lower}--\eqref{eq:thetaDe}
hold in an interval $\tau\in [0, \tau_1]$. Then there exists $C$ independent of $\tau_1$
such that for the same time interval, the parameters $a$ and $b$ and the function $\phi$
are such that
\begin{equation}\label{EstABM}
A(\tau)+B(\tau)+|M(\tau)|\le  C.
\end{equation}
\end{theorem}

This theorem will be reformulated into Propositions~\ref{Prop:Gam12} and \ref{Prop:Main3},
and then proved in Section~\ref{eqn:split3}.

\subsection{Structure}
Before giving the details of the second bootstrap argument, we discuss the general
strategy of its proof.
The first observation is that our Main Assumptions imply that there exist $\lambda(t)$
and $b(t)$ such that the solution $v(\cdot,\cdot,\tau)$ remains close to the
adiabatic approximation $V_{a(t(\tau)),b(t(\tau))}$ defined in equation~\eqref{adiabatic}
for at least a short time. Moreover, we can choose those parameters so that the solution
admits the decomposition in equation~\eqref{eqn:split2}, subject to the orthogonality
stipulations of Condition~[Cd].

It is shown in \cite{GS09} that $\lambda(t)$
and $b(t)$ can be chosen so that Condition~[Cd] is satisfied. To state this precisely,
we need some definitions. Given any time $t_0$ and $\delta>0$, we define
$I_{t_0,\delta}:=[t_0,t_0+\delta]$. We say that $\lambda(t)$ is \emph{admissible} on
$I_{t_0,\delta}$ if $\lambda\in C^2(I_{t_0,\delta},\,\mathbb{R}_+)$
and $a(t):=-\lambda\partial_t\lambda\in[1/4,1]$.
Proposition~5.3 and Lemma~5.4 of \cite{GS09} imply the following result.

\begin{proposition}\label{IFT}
Given $t_*>0$, fix $t_0\in[0,t_*)$ and $\lambda_0>0$. Then there exist $\delta,\ve>0$
and a function $\lambda(t)$, admissible on $I_{t_0,\delta}$, such that if

\textsc{(i)}
$\langle x \rangle^{-3} u\in C^1\left([0,t_*),\,L^\infty(\mathbb{S}^1\times\mathbb{R})\right)$,

\textsc{(ii)} $\inf u>0$, and

\textsc{(iii)} $\|v(\cdot,\cdot,t_0)-V_{a_0,b_0}\|_{3,0}\ll b_0$ for some
$a_0\in[1/4,1]$ and $b_0\in(0,\ve]$,

\noindent where $v(y,\theta,\tau)=\lambda(t)^{-1}\,u(\lambda(t)x,\theta,t)$,
then there exist
\[  a(\tau(t))\in C^1\left(I_{t_0,\delta},\,[1/4,1]\right)\quad\text{and}\quad
    b(\tau(t))\in C^1\left(I_{t_0,\delta},\,(0,\ve]\right)\]
such that the rescaled solution $v(y,\theta,\tau)$ admits the decomposition and orthogonality relations
of Condition~[Cd], with $\lambda(0)=\lambda_0$ and $a(\tau(t))=-\lambda(t)\partial_t\lambda(t)$.
\end{proposition}

The second bootstrap machine will establish that
$V_{a(t(\tau)),b(t(\tau))}$ is the large, slowly-changing part of the solution, while
$\phi$ is the small, rapidly-decaying part. The utility of the stipulated orthogonality
will become clear in Section~\ref{SEC:Rescale}, when we compute the linearization of an
equation closely related to \eqref{MCF-v}.
\smallskip

In what follows, we will usually convert from the time scale $t$ to $\tau$. To study the asymptotics
of the solution, we derive equations for $a_\tau$, $b_\tau$, and $\phi$ in  equations~\eqref{eq:xi},
\eqref{eq:F0ab} and \eqref{eq:F2ab}, respectively. Then we analyze those equations to show that
\[
a(\tau) \to  \frac12,\quad
b(\tau) \approx \frac1{\frac1{b_0} + \tau} =: \beta (\tau),\quad\text{and}\quad
\|\langle y\rangle^{-m} \dy^n \phi \|_\infty
\ls \beta^{\frac{m+n}2 + \frac1{10}}
\]
for $(m,n)=(3,0),\ (11/10,0),\ (2,1)$ and $(1,2)$.

The first two estimates above are proved in Section~\ref{Sec:Splitting}. Their
proofs are straightforward generalizations of those in \cite{GS09}. The third
estimate, which establishes the fast decay of $\phi$, is the most critical. It
is proved in Section~\ref{SEC:Rescale}. For all three results, the crucial new
ingredients from \cite{GS09} are the $\theta$-derivative bounds appearing in
estimates~\eqref{eq:upperBounddV}--\eqref{eq:thetaDe}, which follow from the
first bootstrap machine constructed in Sections~\ref{FirstBootstrap}--\ref{ProveLast}.

\section{Evolution equations for the decomposition}\label{SlowlyDecompose}

Equation~\eqref{MCF-v} is not self-adjoint with respect to its linearization,
but it can be made so by a suitable gauge transformation.
By Condition~[Cd] in Section~\ref{sec:asymp}, there exists a ($t$-scale) time
$0<t_{\#} \le \infty$ such that the gauge-fixed quantity
$$w(y,\theta,\tau) := v(y,\theta,\tau) e^{-\frac{a}4 y^2}$$ can be decomposed as
\begin{equation}
w = w_{ab}(y) + \xi(y,\theta,\tau),\quad\text{with}\quad \xi \perp  \phi_{0,a},\ \phi_{2 a}. \label{eqn:split3}
\end{equation}
Here
\begin{equation}\label{BadEvals}
\phi_{0,a} :=  \left( \frac{a}{2\pi}\right)^{\frac14} e^{-\frac{ay^2}4}\quad\text{and}\quad
\phi_{2,a} := \left( \frac{a}{8\pi}\right)^{\frac14}  (1-ay^2) e^{-\frac{ay^2}4},
\end{equation}
and the orthogonality is with respect to the $L^2(\mathbb{S}^1\times\mathbb{R})$ inner product.
The parameters $a$ and $b$ are $C^1$ functions of $t$; the (almost stationary) part is
$$w_{ab} : = \sqrt{2+by^2/a+\frac12}\ e^{-\frac{a}4 y^2};$$
and the (rapidly-decaying) fluctuation is
$$\xi := e^{-\frac{ay^2}4} \phi.$$
To simplify notation, we will write $a(\tau)$ and $b(\tau)$ for $a(t(\tau))$ and
$b(t(\tau))$, respectively. (This will not cause confusion, as the original functions
$a(t)$ and $b(t)$ are not needed until Section~\ref{RepeatMainProof}.) In this
section, we derive evolution equations for the parameters $a(\tau)$ and $b(\tau)$,
and the fluctuation $\xi(y,\theta,\tau)$.
\medskip

We substitute equation~\eqref{eqn:split3} into equation~\eqref{MCF-v} to obtain the
following equation for $\xi$,
\begin{equation}\label{eq:xi}
\dt \xi = - L(a,b) \xi + F(a,b) + N_1 (a,b,\xi) +N_2 (a,b,\xi) + N_3 (a,b,\xi),
\end{equation}
where $L(a,b)$ is the linear operator given by
$$L(a,b) := - \dy^2 + \frac{a^2 + \partial_\tau a }4 y^2 -\frac{3a}2
- \frac{\frac12 +a}{2+b y^2} - \frac12 \dz^2,$$
and the functions $F(a,b)$ and $N_i (a,b,\xi)$, $(i=1,2,3)$, are given below.
We define
\begin{equation}\label{eq:source}
F(a,b) := \frac12  e^{-\frac{a y^2}4}
\left(\frac{2+by^2}{a+\frac12}\right)^{\frac12}
\left[\Gamma_1 +\Gamma_2 \frac{y^2}{2+by^2} - \frac{b^3 y^4}{(2+by^2)^2}\right],
\end{equation}
with
\begin{align}
\Gamma_1 := & \frac{ \partial_\tau a }{a+\frac12} + a - \frac12 + b , \nonumber\\
\Gamma_2 := & - \partial_\tau b  - b \left(a-\frac12 + b\right) - b^2 , \nonumber\\
N_1 (a,b,\xi) := & - \frac1v   \frac{a+\frac12}{2+by^2} e^{\frac{ay^2}4} \xi^2 , \label{eq:defN1}\\
N_2 (a,b,\xi) := & - e^{-\frac{ay^2}4} \frac{p^2}{1+p^2+q^2} \dy^2 v . \label{eq:defN2}
\end{align}
The final term, $N_3$, did not appear in \cite{GS09}; it is
\begin{equation}\label{eq:difN3}
\begin{array}{rl}
N_3 (a,b,\xi) :=  & \displaystyle \left[v^{-2} \frac{1+p^2}{1+q^2+p^2} - \frac12 \right]
\dz^2 v e^{-\frac{ay^2}4}\\
\noalign{\vskip6pt}
&\quad \displaystyle + e^{-\frac{ay^2}4}  v^{-1} \frac{2pq}{1+p^2+q^2}
\dz  \dy v\\
\noalign{\vskip6pt}
&\quad \displaystyle + e^{-\frac{ay^2}4}  v^{-2}  \frac{q}{1+p^2+q^2} \dz v .
\end{array}
\end{equation}
\medskip

Now we derive differential equations for the parameters $a$ and $b$.
Taking inner products on equation~\eqref{eq:xi} with the functions $\phi_{k,a}$, $(k=0,2)$, and
using the orthogonality conditions $\xi \perp \phi_{0,a},\ \phi_{2,a}$ in \eqref{eqn:split3},
one obtains two equations:
\begin{equation}\label{eq:F0ab}
\begin{array}{rcl}
\left\langle F(a,b),\phi_{0,a}\right\rangle
&=&\displaystyle -\left\langle \xi,\partial_\tau \phi_{0,a}\right\rangle
-\left\langle \left(\frac{\frac12 + a}{2+by^2} - \frac{\frac12 + a}2\right)
\xi, \phi_{0,a}\right\rangle\\
\noalign{\vskip6pt}
&& \displaystyle - \sum_{k=1}^3  \left\langle N_k, \phi_{0,a}\right\rangle
+\frac{a_\tau}4 \left\langle \xi,y^2 \phi_{0,a}\right\rangle,
\end{array}
\end{equation}
\begin{equation}\label{eq:F2ab}
\begin{array}{rcl}
\left\langle F(a,b),\phi_{2,a}\right\rangle
&=& \displaystyle - \left\langle \xi,\partial_\tau  \phi_{2,a}\right\rangle
- \left\langle \left(\frac{\frac12 + a}{2+by^2} - \frac{\frac12 + a}2\right )
\xi, \phi_{2,a}\right\rangle\\
\noalign{\vskip6pt}
&& \displaystyle - \sum_{k=1}^3 \left\langle N_k, \phi_{2,a}\right\rangle
+\frac{a_\tau}4 \left\langle \xi,y^2 \phi_{2,a}\right\rangle.
\end{array}
\end{equation}
Note that the terms on the right-hand sides of equations~\eqref{eq:F0ab}--\eqref{eq:F2ab}
depend on $\xi$, while $F(a,b)$ depends on $a_\tau,\ b_\tau,\ a,\ b$ and $y$.
\medskip

We now show that Theorem~\ref{THM:aprior} is a consequence of the following two results.

\begin{proposition}\label{Prop:Gam12}
Suppose Conditions~[Cd] and [Cb] and
estimates~\eqref{eq:upperBounddV}--\eqref{eq:thetaDe} hold.
Then there exists a nondecreasing polynomial $P(M,A)$ of $A$ and the components of $M$ such that
\begin{equation}\label{eq:B}
B(\tau) \ls 1 + P \left(M(\tau),A(\tau)\right)
\end{equation}
and
\begin{equation}\label{eq:A}
A(\tau) \ls A(0) + \beta^{\frac7{10}} (0) P \left(M(\tau),A(\tau)\right).
\end{equation}
\end{proposition}

This will be proved in Section~\ref{Sec:Splitting}.

\begin{proposition}\label{Prop:Main3}
Suppose Conditions~[Cd] and [Cb] and
estimates~\eqref{eq:upperBounddV}--\eqref{eq:thetaDe} hold.
Then the function $\phi$ satisfies the estimates
\begin{equation}\label{eq:M30}
M_{3,0}(\tau) \ls M_{3,0}(0) + \beta^{ \frac1{20}} (0) P \left(M(\tau),A(\tau)\right),
\end{equation}
\begin{equation}\label{eq:M20}
M_{\frac{11}{10},0}(\tau) \ls
M_{\frac{11}{10},0}(0) + 1 + \epsilon_0 M_{\frac{11}{10},0}+M_{3,0}(\tau)
+ \beta^{ \frac{1}{20}} (0) P \left( M(\tau),A(\tau)\right),
\end{equation}
\begin{equation}\label{eq:M21}
M_{2,1}(\tau) \ls M_{2,1}(0) + M_{3,0}(\tau)
+ \beta^{ \frac1{20}} (0)P  \left(M(\tau),A(\tau)\right),
\end{equation}and
\begin{equation}\label{eq:M12}
M_{1,1}(\tau) \ls M_{1,2}(0) + M_{3,0} (\tau) + M_{2,1}(\tau)
+ \beta^{\frac{1}{20} }(0)P \left(M(\tau),A(\tau)\right),
\end{equation}
for any $\tau\in [0,\tau(t_{\#})]$, where $P(M,A)$ is a nondecreasing
polynomial of $A$ and the components of $M$.
\end{proposition}
This will be proved in Section~\ref{SEC:Rescale}.
\medskip

These two propositions imply Theorem~\ref{THM:aprior}, as we now show.

\begin{proof}[Proof of Theorem~\ref{THM:aprior}]
Observe that $M_{3,0}(\tau)$ and $M_{2,1}(\tau)$ are present on the right-hand sides
of equations~\eqref{eq:M20}--\eqref{eq:M12}.
To remove these, we use estimates~\eqref{eq:M30} and \eqref{eq:M21}
to recast estimates~\eqref{eq:B}, \eqref{eq:A}, and \eqref{eq:M30}--\eqref{eq:M12} as
\begin{align*}
A(\tau) + |M(\tau)| \ls & A(0) + |M(0)| + \beta^{\frac1{10}}
(0)P \left(|M(\tau)|,A(\tau)\right),\\ \\
B(\tau) \ls &1 + P  \left(|M(\tau)|,A(\tau)\right),
\end{align*}
with $P$ being some polynomial.
By the boundedness of $|M(0)|,\ A(0)$ and the smallness of $\beta_0$,
we obtain the desired estimate \eqref{EstABM}.
\end{proof}

\section{Estimates to control the parameters $a$ and $b$}\label{Sec:Splitting}

We start by stating a preliminary estimate, which is proved later in this section.

\begin{lemma}\label{LM:Gam12}
The functions $\Gamma_1$ and $\Gamma_2$ appearing in definition~\eqref{eq:source} satisfy
\begin{align}
|\Gamma_1|,\ |\Gamma_2|  \ls
\beta^{\frac{27}{10}}  P(M,A).\label{eq:estI1I2}
\end{align}
\end{lemma}

Now we use estimate~\eqref{eq:estI1I2} to prove Proposition~\ref{Prop:Gam12}.

\begin{proof}[Proof of Proposition~\ref{Prop:Gam12}]

We start by proving estimate~\eqref{eq:B}.

Recall that $\Gamma_1$ and $\Gamma_2$ are defined in terms of $a_\tau,\ b_\tau,$ $a$ and $b$.
Thus we begin by rewriting $\Gamma_2$ in estimate~\eqref{eq:estI1I2} as
\begin{equation*}
\left| \partial_\tau  b  + b^2\right|
\ls b  \left|\frac12 - a + b \right| + \beta^{\frac{27}{10}}   P(M,A).
\end{equation*}
The first term on the \textsc{rhs} is bounded by
$b \beta^2  A  \ls \beta^3  A$ by definition of $A$.
Hence
\begin{align}\label{eqn:Simplifiedb}
\left| \partial_\tau b  + b^2\right| & \ls \beta^{\frac{27}{10}}P(M,A).
\end{align}
Divide estimate~\eqref{eqn:Simplifiedb} by $b^2$, and use the inequality
$\beta\ls b$ implied by the condition $B\le \beta^{-\frac1{20}}$, to obtain
\begin{equation} \label{est:InversebDE}
\left| \partial_\tau  \frac1{b}-1 \right|  \ls \beta^{\frac7{10}}  P(M,A).
\end{equation}
Then use the fact that $\beta$ satisfies $-\partial_\tau \beta^{-1} + 1 = 0$ to get
\begin{equation*}
\left|\partial_\tau \left( \frac1b  -\frac1{\beta}\right) \right|
\ls  \beta^{\frac7{10}} P(M,A).
\end{equation*}
Integrating this estimate over $[0,\tau]$ and using the facts that $b \ls \beta$ and
$\beta(0) = b(0)$, we obtain
\begin{equation*}
\beta^{-\frac32}  |\beta - b|
\ls \beta^{\frac12}  \int_0^\tau \beta^{\frac7{10}}  (s) P\left  ( M(s),A(s)\right)\,\mathrm{d}s
\ls \beta^{\frac12} P\left(M(\tau),A(\tau)\right),
\end{equation*}
which together with the definitions of $\beta$ and $B$ implies estimate~\eqref{eq:B}.

Now we turn to estimate~\eqref{eq:A}.
To facilitate later discussions, we define
\begin{align}\label{eq:difGam}
\Gamma := a - \frac12 + b.
\end{align}
Differentiating $\Gamma$ with respect to $\tau$,
writing $ \partial_\tau b  $ and $\partial_\tau a$ in terms of $\Gamma_1$ and $\Gamma_2$,
and using estimate~\eqref{eq:estI1I2}, we obtain
\begin{equation*}
\partial_\tau \Gamma + \left( a+\frac12 + b\right)\Gamma = - b^2+\mathcal{R}_b,
\end{equation*}
where $\mathcal{R}_b$ obeys the bound
$$|\mathcal{R}_b|  \le \beta^{\frac{27}{10}} P(M,A).$$
Let $\mu=e^{\int_0^\tau a(s) + \frac12 + b(s)\,\mathrm{d}s}$.
Then the equation above implies that
\begin{equation*}
\mu\Gamma = \Gamma_0  -\int_0^\tau \mu b^2\,\mathrm{d}s+\int_0^\tau \mu \mathcal{R}_b\,\mathrm{d}s.
\end{equation*}
We now use the inequality $b\ls \beta$ and the bound for
$\mathcal{R}_b$ to estimate over $[0,\tau]\le [0,T]$ that
\begin{equation*}
|\Gamma|  \ls \mu^{-1} \Gamma_0 + \mu^{-1} \int_0^\tau \mu \beta^2\,\mathrm{d}s
    + \mu^{-1}\left[\int_0^\tau \mu \beta^{\frac{27}{10}}\,\mathrm{d}s\right] P(M,A).
\end{equation*}
For our purposes, it is sufficient to use the weaker inequality
\begin{equation*}
|\Gamma| \ls \mu^{-1} \Gamma(0) + \mu^{-1} \int_0^\tau \mu \beta^2\,\mathrm{d}s
    \left[1+ \beta^{\frac7{10}}(0) P(M,A)\right].
\end{equation*}
The conditions $A(\tau),\ B(\tau)\le \beta^{-\frac1{20}} (\tau)$ imply that
$a + \frac12 + b\ge \frac12.$
Thus, it is not difficult to show that $\beta^{-2} \mu^{-1} \Gamma(0) \le A(0)$ and
$\beta^{-2} \mu^{-1} \int_0^\tau \mu \beta^2\,\mathrm{d}s$ are bounded. Hence we have
$$A \ls A(0) + 1 + \beta^{\frac7{10}}(0)P(M,A),$$
which is estimate~\eqref{eq:A}.

This completes the proof of Proposition~\ref{Prop:Gam12}.
\end{proof}
\medskip

Now we present:

\begin{proof}[Proof of Lemma \ref{LM:Gam12}]
We begin by analyzing equations~\eqref{eq:F0ab}--\eqref{eq:F2ab}.
For notational purpose, we denote the terms on their right-hand sides by
$G_1$ and $G_2$, namely
\begin{align}
\left\langle F(a,b),\ \phi_{0,a}\right\rangle = &\ G_1,\label{eq:Projection1}\\
\left\langle F(a,b),\ \phi_{2,a}\right\rangle = &\ G_2.\label{eq:Projection2}
\end{align}

We start by bounding the terms in $G_1$ and $G_2$.
Using estimates~\eqref{eq:thetaDe} and \eqref{v-above}, both of which follow
from the first bootstrap machine, one finds for $k=0,2$ that
\begin{align*}
\left| \left\langle v^{-1} F_3 \dy \dz v,\ \phi_{k,a}\right\rangle\right|
  +&\left| \left\langle v^{-2}  F_4 \dz  v,\ \phi_{k,a}\right\rangle\right| \\
\noalign{\vskip6pt}
&\ls |v^{-2} \dz  v|
 \left(|v^{-2} \dy  \dz  v| + |v^{-2} \dz v| \right)
 \ls \beta^{\frac{33}{10}}.
\end{align*}
For the term $\langle v^{-2}F_2 \dz^2 v,\ \phi_{k,a}\rangle$, we
integrate by parts in $\theta$ to generate sufficient decay estimates, obtaining
\begin{align}
 & \left| \left\langle v^{-2}\frac{1+p^2}{1+p^2+q^2} \dz^2 v,
 e^{-\frac{a}4 y^2}\right\rangle\right| \nonumber\\
 \noalign{\vskip6pt}
&= \left| \left\langle \dz v, \dz
[v^{-2} \frac{1+p^2}{1+p^2+q^2}  e^{-\frac{a}4 y^2}]\right\rangle\right| \nonumber\\
 \noalign{\vskip6pt}
&= \left|\left\langle 2v^{-3} \frac{1+p^2}{1+p^2+q^2} (\dz v)^2
+ v^{-2} \frac{2pq \dy \dz v - 2(1 + p^2)(v^{-1}\dz^2 v-2 q^2)}{(1+p^2+q^2)^2} \dz v,
    e^{-\frac{a}4 y^2}\right\rangle\right|\nonumber\\
 \noalign{\vskip6pt}
&\ls  \beta^{\frac{33}{10}}\nonumber.
\end{align}
In the last step, we applied estimate~\eqref{eq:thetaDe}.
By similar methods, i.e.~integration by parts, we get
$$\left| \left\langle v^{-2} \frac{1+p^2}{1+p^2+q^2} \dz^2 v,
\left( ay^2-1\right)e^{-\frac{a}4 y^2}\right\rangle \right| \ls \beta^{\frac{33}{10}}.$$

For the terms $N_1$ and $N_2$, we have
$$|\langle N_1,\ \phi_{k,a}\rangle |
\ls \|\langle y\rangle^{-6} N_1\|_\infty
\le  M_{3,0}^2 \beta^{\frac{16}5}, \qquad (k=0,2).$$
Then we use the estimate $|\dy^2 v| \ls \beta^{\frac{13}{20}}$ of \eqref{eq:yDe} to obtain
$$\begin{array}{rcl}
\displaystyle \left| \left\langle N_2,\ \phi_{k,a}\right\rangle\right|\,
\left\| \langle y\rangle^{-2} p^2 \dy^2 v\right\|_\infty
&\le  & \left\| \langle y\rangle^{-3} \dy v \right\|_\infty^2
\left\| \dy^2 v\right\|_\infty\\
\noalign{\vskip6pt}
&\ls &
\left(\beta+M_{2,1}\beta^{\frac{16}{10}}\right)^2 \beta^{\frac{13}{20}},\qquad (k=0,2).
\end{array}
$$
In estimating $\langle y\rangle^{-3} \dy v$, we used
$v=\sqrt{{2+by^2}/{a+\frac12}} + e^{\frac{a}4  y^2}$
and the facts $a\approx \frac12$, $b\approx \beta$ implied by the consequences
$A,\ B\ls \beta^{-\frac1{20}}$ of  estimate~\eqref{eq:assuMAB} and the definition of
$M_{3,0}$ in \eqref{eq:majorMmn}.

We collect the estimates above, using our conditions on $A(\tau)$ and $B(\tau)$, to get
\begin{align}\label{eq:G1G2}
G_1,G_2 \ls | \partial_\tau a |\beta^2  M_{3,0}
+ \beta^{\frac{53}{20}} \left(1 + M_{2,1}^2 + M_{3,0}^2\right).
\end{align}

Now we turn to the terms on the left-hand sides of equations~\eqref{eq:Projection1}--\eqref{eq:Projection2}.
We decompose the function $F(a,b)$ defined in equation~\eqref{eq:source} into components
approximately parallel or orthogonal to $\phi_{0,a}$ and $\phi_{2,a}$,
plus a term of order $b^3$, as follows:
$$F(a,b) = \left( \frac{2+by^2}{a+\frac12} \right)^{\frac12}
\frac12  e^{-\frac{ay^2}4}
\left[ \Gamma_1 + \Gamma_2  \frac1{a[2+by^2]}
+ \Gamma_2  \frac{ay^2 -1}{a[2+by^2]} - \frac{b^3 y^4}{(2+by^2)^2} \right].$$
Consequently, one has
$$|\langle F(a,b),\phi_{0,a}\rangle|
\gtrsim \left| \Gamma_1 + \frac1{2a} \Gamma_2\right|
- |b| \left(|\Gamma_1| + |\Gamma_2|\right)-b^3$$
and
$$|\langle F(a,b), \phi_{2,a}\rangle|
\gtrsim |\Gamma_2| - |b| (|\Gamma_1| + |\Gamma_2|) - b^3,$$
which together with estimate~\eqref{eq:G1G2} imply that
\begin{align*}
|\Gamma_1|,\ |\Gamma_2|
\ls \beta^{\frac{53}{20}} \left(1+M_{2,1}^2 + M_{3,0}^{2}\right)
+  |\partial_\tau  a |\beta^2  M_{3,0}.
\end{align*}
To control the term $|\partial_\tau a|$ on the \textsc{rhs}, we use
the definitions of $\Gamma_1$ and $ \partial_\tau a $ to obtain
$$| \partial_\tau a |  \ls |\Gamma_1| + \beta^2 A.$$
Hence we have
$$|\Gamma_1|,\ |\Gamma_2|
\ls  \beta^{\frac{53}{20}} \left(1+M_{2,1}^2 + M_{3,0}^2\right)
+ \beta^4 AM_{3,0}.$$
This together with the condition
$|M(\tau)| \le \beta^{- \frac1{10} }  (\tau)$ implies estimate~\eqref{eq:estI1I2}.
\end{proof}

To facilitate later discussions, we state some other useful estimates for $F$.

\begin{corollary}
If $(m,n)\in\left\{(3,0), (11/10,0),\ (2,1), \ (1,1)\right\}$
and $A(\tau),\ B(\tau) \ls \beta^{- \frac1{20} }(\tau)$, then
\begin{equation}\label{eq:estSource}
\Big\| e^{\frac{ay^2}4}  F(a,b)\Big\|_{m,n}
\ls  \beta^{\frac{m+n+1}2} (\tau) P(M,A).
\end{equation}
\end{corollary}
\begin{proof}
In what follows, we only prove the case $(m,n) = (11/10,0)$.
The proofs of the remaining cases are similar (see \cite{GS09}), hence omitted.

Recalling definition~\eqref{eq:source} for $F$, one observes that
$$\Big\|\langle  y\rangle^{-\frac{11}{10}}
\frac{y^2}{(2+by^2)^{ \frac12}} \Big\|_\infty
\le b^{-\frac9{20}}
\ls \beta^{-\frac9{20}}.$$
Thus the estimates for $\Gamma_1$ and $\Gamma_2$ in \eqref{eq:estI1I2} imply that
$$\left\|\langle y\rangle^{-\frac{11}{10}}
\left( \frac{2+by^2}{a+\frac12}\right)^{\frac12}
\left[ \Gamma_1 + \Gamma_2 \frac{y^2}{2+by^2}\right] \right\|_\infty
\ls  \left(|\Gamma_1| + |\Gamma_2|\right) \beta^{-\frac9{20}}
\ls \beta^{ \frac{11}5 }  P(M,A).
$$
By similar reasoning, one has
$$\left\| \langle y\rangle^{-\frac{11}{10}} \frac{b^3 y^4}{a(2+by^2)^2}
\left( \frac{2+by^2}{a+\frac12}\right)^{\frac12}\right \|_\infty
\ls   \beta^{ \frac{21}{20} }.$$
Combining the estimates above, we complete the estimate for $(m,n) = (11/10,0)$.
\end{proof}

\section{Estimates to control the fluctuation $\phi$}\label{SEC:Rescale}\label{FullStop}

In this section, we prove the estimates that make up Proposition~\ref{Prop:Main3}.
\smallskip

Recall the evolution equation~\eqref{eq:xi} satisfied by $\xi (y,\theta,\tau)$.
We begin our analysis by reparameterizing $\xi$ so that critical terms in the resulting
(new) evolution equation have time-independent coefficients.

Recall that $\tau(t) := \int_0^t  \lambda^{-2} (s)\,\mathrm{d}s$ for any $\tau\ge  0$,
and $a(\tau) := - \lambda (t(\tau)) \partial_t  \lambda(t(\tau))$.
Let $t(\tau)$ be the inverse function to $\tau(t)$, and pick $T>0$. We approximate
$\lambda (t(\tau))$ on the interval  $0\le \tau \le T$ by a new trajectory
$\lambda_1 (t(\tau))$, chosen so that $\lambda_1 (t(T)) = \lambda (t(T))$ and
$\alpha := - \lambda_1 (t(\tau)) \partial_t \lambda_1  (t(\tau)) = a(T)$ is constant.

Then we introduce new independent variables $z(x,t) := \lambda^{-1}_1  (t)x$ and
$\sigma(t) := \int_0^t  \lambda^{-2}_1(s)\,\mathrm{d}s$, together with a new function
$\eta(z,\theta,\sigma)$ defined by
\begin{equation}\label{NewFun}
\lambda_1 (t) e^{\frac{\alpha}4 z^2} \eta (z,\theta,\sigma)
    :=  \lambda (t) e^{\frac{a(\tau)}4 y^2} \xi (y,\theta,\tau)
    \equiv\lambda (t)\phi(y,\theta,\tau).
\end{equation}
One should keep in mind that the variables $z$ and $y$,
$\sigma$, $\tau$ and $t$ are related by
$z = \frac{\lambda(t)}{\lambda_1(t)}y$,
$\sigma(t) := \int_0^t  \lambda_1^{-2} (s)\,\mathrm{d}s$, and
$\tau = \int_0^t \lambda^{-2} (s)\,\mathrm{d}s$.

For any $\tau = \int_0^{t(\tau)} \lambda^{-2} (s)\,\mathrm{d}s$ with $t(\tau) \le t(T)$,
or equivalently $\tau\le  T$, we define a new function
$\sigma(\tau) := \int_0^{t(\tau)} \lambda_1^{-2} (s)\,\mathrm{d}s$.
Observing that the function $\sigma$ is invertible, we denote its inverse by $\tau(\sigma)$.

The new function $\eta$ satisfies the equation
\begin{equation}\label{eq:eta}
\begin{array}{rl}
\partial_\sigma \eta
\ = &- \mathcal{L}_\alpha \eta (\sigma)
+ \mathcal{W} (a,b) \eta (\sigma)
+ \mathcal{F} (a,b) (\sigma)\\
\noalign{\vskip6pt}
& + \mathcal{N}_{1} (a,b,\eta)
 + \mathcal{N}_{2} (a,b,\eta)
 + \mathcal{N}_{3} (a,b,\eta),
\end{array}
\end{equation}
where the operator $\mathcal{L}_\alpha$ is linear,
$$\mathcal{L}_\alpha  := L_\alpha + V,$$
with $L_\alpha$ and $V$ defined as
$$ L_\alpha :=
- \partial_z^2  + \frac{\alpha^2}4 z^2 - \frac{3\alpha}2 - \frac12 \dz^2,$$
$$V := - \frac{2 \alpha}{2+\beta(\tau(\sigma)) z^2},$$
respectively. The quantity $\mathcal{W}(a,b)$ is a small linear factor of order $\mathcal{O}(\beta)$,
namely
\begin{align}\label{eq:defineTW}
\mathcal{W}(a,b):
= & - \frac{\lambda^2}{\lambda_1^2}\
\frac{(a+\frac12)}{2+b(\tau(\sigma)) y^2}
+ \frac{2\alpha}{2+\beta(\tau(\sigma)) z^2}\\
\noalign{\vskip6pt}
= & \left[ 1 - \frac{\lambda^2}{\lambda_1^2} \right]
\frac{(a+\frac12)}{2+b(\tau(\sigma))y^2} + \frac{2\alpha-a-\frac12}
{2+\beta(\tau(\sigma)) z^2} \nonumber\\
\noalign{\vskip6pt}
& + \left( a+\frac12\right)
\frac{z^2[b(\tau(\sigma)) - \beta(\tau(\sigma))]}
	{(2+b(\tau(\sigma))y^2)(2+\beta(\tau(\sigma)) z^2)}\nonumber\\
\noalign{\vskip6pt}
& + \left( a+\frac12\right)
\frac{[y^2-z^2]b(\tau(\sigma))}{(2+b(\tau(\sigma))y^2)(2+\beta(\tau(\sigma)) z^2)}.\nonumber
\end{align}
The function $\mathcal{F}(a,b)$ is a variant of $F(a,b)$ from definition~\eqref{eq:source}, namely
\begin{equation}\label{ScriptF}
\mathcal{F}(a,b) :=  e^{-\frac{\alpha}4  z^2}  e^{\frac{a}4 y^2}
\frac{\lambda_1}{\lambda}F(a,b).
\end{equation}
The nonlinear terms $\mathcal{N}_k$, $(k=1,2,3)$,
are defined using \eqref{eq:defN1}--\eqref{eq:difN3} by
\begin{equation}\label{eq:DefNonlin}
\begin{array}{lll}
\mathcal{N}_1 (a,b,\eta)
&:=& \displaystyle \frac{\lambda_1}{\lambda}  e^{-\frac{\alpha}4 z^2}
e^{\frac{a}4 y^2} N_1 (a,b,\xi),\\
& &\\
\mathcal{N}_2 (a,b,\eta)
&:=& \displaystyle \frac{\lambda_1}{\lambda}  e^{-\frac{\alpha}4 z^2}
e^{\frac{a}4 y^2} N_2 (a,b,\xi),\\
& &\\
\mathcal{N}_3 (a,b,\eta)
&:=& \displaystyle \frac{\lambda_1}{\lambda}  e^{-\frac{\alpha}4 z^2}
e^{\frac{a}4 y^2} N_3 (a,b,\xi),
\end{array}
\end{equation}
with $\tau$ and $y$ expressed in terms of $\sigma$ and $z$.
Note that $\mathcal{N}_3$ was not needed in \cite{GS09}.

Analysis of equation~\eqref{eq:eta} will provide estimates for $\eta$, but what
we actually need are estimates for $\xi$. To provide this link, we prove in the
next proposition that the new trajectory is a good approximation of the old one.

\begin{proposition}\label{NewTrajectory}
For any $\tau\le T$, if $A(\tau)\ls \beta^{- \frac1{10} }(\tau)$, then
\begin{equation}\label{eq:compare}
\left| \frac{\lambda}{\lambda_1} (t(\tau)) - 1\right|
\leq c \beta(\tau)
\end{equation}
for some constant $c$ independent of $\tau$.
For $(m,n) = (3,0),\ (11/10,0),\ (2,1),\ (1,1)$,
one has
\begin{equation}\label{eq:compareXiEta}
\Big\| e^{\frac{\alpha}4 z^2} \eta(\cdot,\cdot,\sigma)\Big\|_{m,n}
\ls \beta^{\frac{m+n}2 + \frac1{10}} (\tau(\sigma)) M_{m,n} (\tau(\sigma)).
\end{equation}
\end{proposition}

\begin{proof}
We start by deriving a convenient expression for $\frac{\lambda}{\lambda_1} (t(\tau))-1$.
By properties of $\lambda$ and $\lambda_1$, we have
\begin{equation}\label{EstLambda}
\partial_\tau \left( \frac{\lambda}{\lambda_1} (t(\tau)) - 1\right)
= 2a (\tau) \left( \frac{\lambda}{\lambda_1} (t(\tau)) - 1\right) + G (\tau),
\end{equation}
where
$$G := \alpha-a + (\alpha-a) \left( \frac{\lambda}{\lambda_1} - 1\right)
\left[ \left( \frac{\lambda}{\lambda_1}\right)^2
+  \frac{\lambda}{\lambda_1} + 1\right]
+ a \left( \frac{\lambda}{\lambda_1} - 1\right)^2
\left[ \frac{\lambda}{\lambda_1}  + 2\right].$$
We use the fact that $\frac{\lambda}{\lambda_1} (t(\tau)) - 1 = 0$ when
$\tau = T$ to rewrite equation~\eqref{EstLambda} as
\begin{equation}\label{LambdaRe}
\frac{\lambda_1}{\lambda}(t(\tau)) - 1 =
- \int_\tau^T  e^{-\int^s_\tau  2a(t)\,\mathrm{d}t} G(s)\,\mathrm{d}s.
\end{equation}
In what follows, we rely on equation~\eqref{LambdaRe} to prove estimate~\eqref{eq:compare}.

We start by analyzing the integrand on the \textsc{rhs}.
The definition of $A(\tau)$ and the condition
$A(\tau)\ls \beta^{- \frac1{20}}(\tau)$ together imply that
$$ \left| a(\tau) - \frac12 \right| \ls \beta(\tau),$$
and hence that
\begin{equation}\label{CauchA}
|a(\tau) - \alpha|  \le  \left| a(\tau) - \frac12 \right| + \left|\alpha - \frac12\right| \le  \beta(\tau)
\end{equation}
in the time interval $\tau\in [0,T]$.
Thus
\begin{equation}\label{Rem}
|G| \ls
\beta + \left(\frac{\lambda}{\lambda_1} - 1\right)^2
+ \left| \frac{\lambda}{\lambda_1} - 1\right|^3
+ \beta \left| \frac{\lambda}{\lambda_1} -1\right|.
\end{equation}

We claim that inequalities~\eqref{CauchA}--\eqref{Rem} are sufficient to prove estimate~\eqref{eq:compare}.
Indeed, define an estimating function $\Lambda (\tau)$ by
$$\Lambda(\tau) := \sup_{\tau\le  s\le  T} \beta^{-1} (s)
\left| \frac{\lambda}{\lambda_1} (t(s)) - 1 \right |.$$
Then \eqref{LambdaRe} and the conditions
$A(\tau),\ B(\tau)\ls \beta^{- \frac1{20} }(\tau)$ imply that $2a\ge  \frac12$ and hence that
$$
\begin{array}{lll}
\displaystyle \left| \frac{\lambda}{\lambda_1} (t(\tau)) - 1\right|
&\ls &
\displaystyle \int_\tau^T  e^{-\frac12 (T-\tau)}
\left[ \beta(s) + \beta^2 (s) \Lambda^2 (\tau) + \beta^2 (s) \Lambda(\tau)\right ]\,\mathrm{d}s\\
\noalign{\vskip6pt}
&\ls
&\beta(\tau) + \beta^2 (\tau) \Lambda^2 (\tau)
+ \beta^3 (\tau) \Lambda^3 (\tau) + \beta^2 (\tau) \Lambda(\tau).
\end{array}
$$
Therefore,
$$\beta^{-1}(\tau)  \left| \frac{\lambda}{\lambda_1} (t(\tau))-1 \right|
\ls
1 + \beta(\tau) \Lambda^2 (\tau)
+ \beta^2 (\tau) \Lambda^3 (\tau) + \beta (\tau) \Lambda(\tau).$$
Now we use the facts that $\beta(\tau)$ and $\Lambda(\tau)$ are
decreasing functions to obtain
$$\Lambda(\tau) \ls
1 + \beta(\tau) \Lambda^2 (\tau) + \beta^2 (\tau) \Lambda^3 (\tau)
+ \beta(\tau) \Lambda(\tau),$$
which together with $\Lambda(T) = 0$ implies that $\Lambda(\tau) \ls 1$
for any time $\tau \in [0,T]$. Combining this with the definition of $\Lambda(\tau)$
gives estimate~\eqref{eq:compare}.
\end{proof}

In the rest of this section, we study the linear operator $\mathcal{L}_\alpha$.
By Lemma~\ref{Symmetries}, it suffices to consider its spectrum acting on
$\{w: \ \mathbb{R}\times\mathbb{S}^1: w(z,\theta) = w(z,\theta+\pi)\}$.
Due to the presence of the quadratic term $\frac14  \alpha z^2$, the operator
$\mathcal{L}_\alpha$ has a discrete spectrum.
For $\beta z^2  \ll 1$, it is close to the harmonic oscillator Hamiltonian
\begin{equation}\label{eq:spectrum}
L_\alpha - \alpha := - \partial_z^2 + \frac14 \alpha^2
z^2 - \frac{5\alpha}2 - \frac12  \dz^2.
\end{equation}
The spectrum of the operator $L_\alpha  - \alpha$ is
\begin{equation}
\mathrm{spec}(L_\alpha - \alpha) = \left\{ n \alpha + 2k^2 :
n = -2,-1,0,1,\ldots; \ k=0,1,2,\ldots\right\}.
\end{equation}
Thus it is essential that we can solve the evolution equation~\eqref{eq:eta}
on the subspace orthogonal to the first three eigenvectors of $L_{\alpha}$.
These eigenvectors, normalized, are
\begin{align}\label{eq:eigenvectors}
& \phi_{0,\alpha}(z) := \left(\frac{\alpha}{2\pi}\right)^\frac{1}{4}
e^{-\frac{\alpha}{4}z^2},\nonumber\\
\noalign{\vskip6pt}
& \phi_{1,\alpha}(z) := \left(\frac{\alpha}{2\pi}\right)^{\frac{1}{4}}\sqrt{\alpha}\ z
e^{-\frac{\alpha}{4}z^2},\\
\noalign{\vskip6pt}
& \phi_{2,\alpha}(z) := \left(\frac{\alpha}{8\pi}\right)^{\frac{1}{4}}(1-\alpha
z^2)e^{-\frac{\alpha}{4}z^2}.\nonumber
\end{align}
We define the orthogonal projection $\overline{P}^\alpha_n$ onto
the space spanned by the first $n$ eigenvectors of $L_\alpha$ by
\begin{equation}\label{eq:projection}
\overline{P}^\alpha_n w  :=  \sum_{m=0}^{n-1} \langle w, \phi_{m,\alpha} \rangle,
\end{equation}
and the orthogonal projection
\begin{align}\label{eq:finiProj}
P^\alpha_n w := 1 - \overline{P}^\alpha_n w\ , \qquad (n=1,2,3).
\end{align}
The following proposition provides useful decay estimates for the
propagators generated by $- L_\alpha$ and $-\mathcal{L}_\alpha$.

\begin{proposition}\label{PRO:propagator}
If $g:\ \mathbb{R}\times \mathbb{S}^1\to \mathbb{R}$ is a function satisfying
$g(y,\theta) = g(y,\theta + \pi)$ for all $y \in \mathbb{R}$ and $\theta\in \mathbb{S}^1$,
then for any times $\tau\ge \sigma\ge  0$, one has
\begin{equation}\label{eq:estproject2}
\big\| \langle z\rangle^{-n} e^{\frac{\alpha}4 z^2}  e^{-L_\alpha \sigma}
P^\alpha_2 g \big\|_{\infty}
\ls  e^{(1-n)\alpha\sigma}
\big\| \langle  z\rangle^{-n} e^{\frac{\alpha}4 z^2} g\big\|_\infty
\end{equation}
with $2\ge   n\ge  1$. Moreover,
\begin{align}\label{eq:firstProj}
\big\| \langle z\rangle^{-1} e^{\frac{\alpha}4 z^2} e^{-L_\alpha\sigma}
P^\alpha_1 g \big\|_{\infty}
\ls \big\|\langle z\rangle^{-1} e^{\frac{\alpha}4  z^2}g\big\|_\infty;
\end{align}
and there exist constants $\gamma,\delta>0$ such that if $\beta(0)\le  \delta$, then
\begin{equation}\label{eq:OperatorWithV}
\big\| \langle z\rangle^{-n}  e^{\frac{\alpha}4 z^2} P^\alpha_n
U_n (\tau,\sigma) e^{(\tau-\sigma)\frac12  \dz^2} P^\alpha_n g \big\|_\infty
\ls
e^{-(\gamma + (n - 3) \alpha)(\tau-\sigma)}
\big\| \langle z\rangle^{-n} e^{\frac{\alpha}4 z^2} g\big\|_\infty,
\end{equation}
where $U_n (\tau,\sigma)$ denotes the propagator generated by the operator
$- P^\alpha_n [\mathcal{L}_\alpha +\frac12 \dz^2] ^\alpha_n$, $(n=1,2,3)$.
\end{proposition}

\begin{proof}
What makes estimates~\eqref{eq:estproject2}--\eqref{eq:OperatorWithV} different
from the corresponding results in the earlier paper \cite{GS09} is the presence
of the operator $\dz^2$, which has discrete spectrum $\{0, 4, 8,\cdots \}$ in the
space $\{h:\mathbb{S}^1\rightarrow\mathbb{R}:h(\theta) = h(\theta+\pi)\}.$
(Compare Remark~\ref{FixAxis}.)

The general strategy in the present situation is to decompose the function
$g:\ \mathbb{R}\times \mathbb{S}^1\to \mathbb{R}$ according to the spectrum of $\dz^2$.
One then obtains the desired estimate by studying the terms of this decomposition.
If a real-valued function $g$ satisfies $g(y,\theta) = g(y,\theta + \pi)$, then it
admits the decomposition
$$g(y,\theta) =  \sum_{k=-\infty}^\infty  e^{2i k\theta} g_k(y),$$
with $g_k(y) := \frac1{2\pi} \int_0^{2\pi} e^{-2i k\theta} g(y,\theta)\,\mathrm{d}\theta$.
By the definition of $P_n$ in equation~\eqref{eq:finiProj}, one has
$(1-P_n)e^{2i k\theta}g_k(y) = 0$ if $k\ne 0$, and hence
$$P_n g = P_n g_0 + \sum_{k\ne 0} e^{2i k\theta}g_k.$$
Moreover,
$$e^{\frac12 \sigma \dz^2} P_n g = P_n g_0
 + e^{\frac12 \sigma \dz^2} \sum_{k\ne 0} e^{2i k\theta} g_k.$$

Now apply these observations to the propagator acting on $g$ to obtain
\begin{align}\label{eq:twoparts}
\langle z\rangle^{-n} e^{\frac{\alpha}4  z^2} P^\alpha_n
&U_n (\tau,\sigma) e^{(\tau-\sigma)\frac12 \dz^2} P^\alpha_n g
\nonumber\\
\noalign{\vskip6pt}
&\qquad
 =\langle z\rangle^{-n}  e^{\frac{\alpha}4  z^2} P^\alpha_n
U_n (\tau,\sigma) P^\alpha_n g_0\\
\noalign{\vskip6pt}
&\qquad\qquad
 +\langle z\rangle^{-n} e^{\frac{\alpha}4 z^2}
\tilde U (\tau,\sigma) e^{(\tau-\sigma)\frac12 \dz^2}
\sum_{k\ne 0}  e^{2ik\theta} g_k,\nonumber
\end{align}
where $\tilde U$ is generated by the operator
$$- \left[ \mathcal{L}_\alpha + \frac12 \dz^2\right]
= - \left[ -\partial_z^2 + \frac{\alpha^2}4  z^2 - \frac{3\alpha}2
- \frac{2 \alpha}{2+\beta(\tau(\sigma)) z^2} \right].$$
For the first term on the \textsc{rhs}, we apply results proved in \cite{GS09} to obtain
\begin{align}\label{eq:estG0}
\left\| \langle z\rangle^{-n}  e^{\frac{\alpha}4  z^2} P^\alpha_n
U_n (\tau,\sigma) P^\alpha_n g_0\right\|_{\infty}
\ls e^{-(\gamma + (n-3)\alpha)(\tau-\sigma)}
\left\| \langle  z\rangle^{-n} e^{\frac{\alpha}4 z^2} g_0\right\|_\infty.
\end{align}
For the second term on the \textsc{rhs}, we apply the Trotter formula to see that
$$ \tilde U (\tau,\sigma)
\le   e^{-(\tau-\sigma)(-\partial_z^2 + \frac14 \alpha^2 z^2 - \frac52 \alpha)}.$$
The latter harmonic oscillator be bounded by the non-decaying estimate
$e^{2\alpha(\tau-\sigma)}$ in suitable normed spaces
(for details see \cite{DGSW08, GS09}). Thus one has
\begin{align}\label{eq:nondecay}
& \bigg\| \langle z\rangle^{-n}  e^{\frac{\alpha}4  z^2}
\tilde U (\tau,\sigma) e^{(\tau-\sigma) \frac12  \dz^2}
\sum_{k\ne 0} e^{2ik\theta} g_k\bigg\|_\infty\\
\noalign{\vskip6pt}
&\qquad
\le  e^{2\alpha(\tau-\sigma)}
\bigg\|  \langle z\rangle^{-n} e^{\frac{\alpha}4 z^2} e^{(\tau-\sigma)\frac12
\dz^2}
\sum_{k\ne 0} e^{2ik\theta} g_k \bigg\|_\infty.\nonumber
\end{align}
For the last term above, one may  apply the maximum principle to see that
\begin{align}\label{eq:shortTime}
\bigg| e^{\frac12 \sigma \dz^2} \sum_{k\ne 0} e^{2i k\theta}g_k \bigg|
\le  \max_{\theta \in [0,2\pi)}
\bigg| \sum_{k\ne 0} e^{2i k\theta} g_k \bigg|  \le 2\max_{\theta \in [0,2\pi)} |g|
\end{align}
for any $\sigma\ge 0$; and for $\sigma\ge  1$
\begin{align}\label{eq:higherDecay}
\bigg| e^{\frac12 \sigma \dz^2} \sum_{k\ne 0} e^{2i k\theta} g_k \bigg|
= & \bigg| \sum_{k\ne 0} e^{-2 \sigma k^2} e^{2i k\theta} g_k \bigg|\\
\ls & e^{-2\sigma} \max_k |g_k| \nonumber\\
= & e^{-2\sigma} \frac1{2\pi}
\max_k \left |\langle g,\ e^{2k i\theta}\rangle_\theta \right| \nonumber\\
\le & e^{-2\sigma} \max_\theta |g| \nonumber.
\end{align}
Collecting the estimates above, we conclude that
\begin{align}
\bigg\| \langle z\rangle^{-n} e^{\frac{\alpha}4 z^2}
\tilde U (\tau,\sigma) e^{(\tau-\sigma)\frac12 \dz^2}
\sum_{k\ne 0} e^{2ik\theta} g_k \bigg\|_\infty
\ls e^{-(\tau-\sigma)} \left\| \langle z\rangle^{-n} e^{\frac{\alpha}4 z^2} g\right\|_\infty.
\end{align}
This together with equation~\eqref{eq:twoparts} and estimate~\eqref{eq:estG0} implies
inequality~\eqref{eq:OperatorWithV}.

The proofs of estimates~\eqref{eq:estproject2}--\eqref{eq:firstProj} are similar
modifications of those detailed in \cite{DGSW08, GS09}, hence are omitted here.
\end{proof}

\subsection{Proof of estimate~\eqref{eq:M30}}\label{SEC:EstM1}
In this section, we prove estimate~\eqref{eq:M30} for the function $M_{3,0}$.
Fix a ($\tau$-scale) time $T$. Then the estimates of Proposition~\ref{NewTrajectory} hold
for $\tau\le T$. We start by estimating $\eta$ defined in equation~\eqref{NewFun}.
We observe that $\eta$ is not orthogonal to the first three eigenvectors of the operator $L_\alpha$.
Therefore, we derive an equation for $P^\alpha_3 \eta$, namely
\begin{equation}\label{EQ:eta2}
\partial_\sigma (P^\alpha_3 \eta) = - P^\alpha_3 \mathcal{L}_\alpha P^\alpha_3 \eta
+ \sum_{k=1}^6 D^{(k)}_{3,0}(\sigma)
\end{equation}
where $D^{(k)}_{m,n} \equiv D^{(k)}_{m,n}(\sigma)$, with
$k=1,2,3,4, 5, 6$, $(m,n)=(3,0),\ (2,0),\ (2,1),\ (1,1)$, are defined
as\,\footnote{Recall definitions~\eqref{eq:defineTW}--\eqref{eq:DefNonlin} and \eqref{eq:finiProj}.}
\begin{align*}
D^{(1)}_{m,n} & := - P^\alpha_m V e^{-\frac{\alpha}4 z^2} \partial_z^n
[e^{\frac{\alpha}4 z^2} \eta]
+ P^\alpha_m V P^\alpha_m e^{-\frac{\alpha}4 z^2} \partial_z^n [e^{\frac{\alpha}4 z^2} \eta]\\
& =  -P^\alpha_m V e^{-\frac{\alpha}4 z^2}
[1 - P^\alpha_m] \partial_z^2  [e^{\frac{\alpha}4 z^2} \eta]\\
&=P^\alpha_m \frac{\alpha \beta(\tau(\sigma)) z^2}{2+\beta(\tau(\sigma))z^2}
[1-P^\alpha_m ] e^{-\frac{\alpha}4 z^2} \partial_z^n [e^{\frac{\alpha}4 z^2} \eta],
\end{align*}
and
\begin{align*}
D^{(2)}_{m,n} &:=  P^\alpha_m e^{-\frac{\alpha}4 z^2} \partial_z^n
[e^{\frac{\alpha}4 z^2} \mathcal{W} \eta],\\
D^{(3)}_{m,n} &:=  P^\alpha_m e^{-\frac{\alpha}4 z^2} \partial_z^n
[e^{\frac{\alpha}4 z^2} \mathcal{F}(a,b)],\\
D^{(4)}_{m,n} &:=  P^\alpha_m e^{-\frac{\alpha}4 z^2} \partial_z^n
[e^{\frac{\alpha}4 z^2} \mathcal{N}_1 (a,b,\alpha,\eta)],\\
D^{(5)}_{m,n} &:=  P^\alpha_m e^{-\frac{\alpha}4 z^2} \partial_z^n
[e^{\frac{\alpha}4 z^2} \mathcal{N}_2],\\
D^{(6)}_{m,n} &:=  P^\alpha_m e^{-\frac{\alpha}4 z^2} \partial_z^n
[e^{\frac{\alpha}4 z^2} \mathcal{N}_3].
\end{align*}

Define
\begin{equation}\label{T2}
S := \int_0^{t(T)}  \lambda_1^{-2} (s)\,\mathrm{d}s.
\end{equation}
Then by Duhamel's principle, we may rewrite equation~\eqref{EQ:eta2} as
\begin{equation}\label{eq:eta3}
P^\alpha_3 \eta(S) =  P^\alpha_3 U_3 (S,0) P^\alpha_3 \eta(0)
+ \sum_{n=1}^6 \int_0^S P^\alpha_3 U_3 (S,\sigma) P^\alpha_3 D_{3,0}^{(n)}
(\sigma)\,\mathrm{d}\sigma,
\end{equation}
where $U_3(\tau,\sigma)$ is defined in Proposition~\ref{PRO:propagator} and
satisfies estimate~\eqref{eq:OperatorWithV}. It follows that
\begin{equation}\label{eq:M1Ge}
\begin{split}
& \beta^{-\frac85} (T)
\|e^{\frac{\alpha}4 z^2} P^\alpha_3 \eta(S)\|_{3,0}\\
\noalign{\vskip6pt}
&\qquad
\ls   e^{-\gamma S} \beta^{-\frac85}(T) \| e^{\frac{\alpha}4 z^2} \eta(0)\|_{3,0}\\
\noalign{\vskip6pt}
&\qquad\quad   + \beta^{-\frac85} (T)
\sum_{k=1}^6  \int_0^S e^{-\gamma(S-\sigma)}
\| e^{\frac{\alpha}4 z^2} D^{(k)}_{3,0}(\sigma)\|_{3,0}\,\mathrm{d}\sigma.
\end{split}
\end{equation}
For the terms $D^{(k)}_{3,0}$ with $k=1,2,3,4,5$, on the \textsc{rhs}
of equation~\eqref{EQ:eta2}, one has the following estimates.

\begin{lemma}\label{LM:EstDs}
If $ A(\tau),\ B(\tau)\ls \beta^{- \frac1{20} }(\tau)$ and if $\sigma\le S$,
equivalently $\tau\le  T$, then
\begin{equation}\label{eq:MplusN3}
\sum_{k=1}^3 \| e^{\frac{\alpha}4 z^2} P_3^\alpha D_{3,0}^{(k)}(\sigma)\|_{3,0}
\ls\beta^{\frac{33}{20}} (\tau(\sigma)) P(M(T),A(T)),
\end{equation}
\begin{equation}\label{eq:estN1}
\|e^{\frac{\alpha}4 z^2} P_3^\alpha D_{3,0}^{(4)} (\sigma)\|_{3,0}
\ls \| e^{\frac{a y^2}4} N_1 (a,b,\xi)\|_{3,0}
\ls \beta^{\frac{33}{20}} (\tau) P(M(\tau)),
\end{equation}
\begin{equation}\label{eq:estN2}
\| e^{\frac{\alpha}4 z^2} P_3^\alpha D_{3,0}^{(5)} (\sigma)\|_{3,0}
\ls \| e^{\frac{a y^2}4} N_2 (a,b,\xi)\|_{3,0}
\ls \beta^{\frac{33}{20}} (\tau) P(M(\tau)),
\end{equation}
and
\begin{equation}\label{eq:estN31}
\| e^{\frac{\alpha}4 z^2} P_3^\alpha D_{3,0}^{(6)} (\sigma)\|_{3,0}
\ls \| e^{\frac{a y^2}4} N_3 (a,b,\xi)\|_{3,0}
\ls \beta^{\frac{33}{20}} (\tau).
\end{equation}
\end{lemma}

\begin{proof}
We first prove the most involved estimate, namely \eqref{eq:estN31}.
Its first inequality is a consequence of the fact that
$\frac{\lambda_1}{\lambda} - 1 = \mathcal{O}(\beta)$ proved in estimate~\eqref{eq:compare}.
Because estimates of this type will appear many times below, we provide a detailed
proof (only) here. Unwrapping the definitions of the functions involved, one finds that
\begin{align}\label{eq:sample}
\| e^{\frac{\alpha}4 z^2} P_3^\alpha D_{3,0}^{(6)} (\sigma)\|_{3,0}
&\le \| e^{\frac{\alpha}4 z^2}D_{3,0}^{(6)} (\sigma)\|_{3,0}\\
&\le \| \langle z\rangle^{-3} e^{\frac{a}4 y^2} N_3 (a,b,\xi) (\tau(\sigma))\|_\infty\nonumber\\
& \ls \| \langle y\rangle^{-3} e^{\frac{a}4 y^2} N_3 (a,b,\xi) (\tau(\sigma))\|_\infty\nonumber\\
& = \| e^{\frac{a}4 y^2} N_3 (a,b,\xi)\|_{3,0}\nonumber.
\end{align}
To prove the second inequality in  estimate~\eqref{eq:estN31}, we write the new term $N_3$ as
\begin{equation}\label{eq:N312}
N_3 := N_{3,1} + N_{3,2},
\end{equation}
with
$$N_{3,1} := \Big[v^{-2}-\frac12\Big] (\dz^2 v) e^{-\frac{ay^2}4}$$
and
\begin{multline*}
N_{3,2} := - v^{-2} \frac{q^2}{1+q^2+p^2} (\dz^2 v) e^{-\frac{ay^2}4}\\
    + e^{-\frac{ay^2}4} v^{-1} \frac{2pq}{1+p^2+q^2}(\dz \dy v)
    + e^{-\frac{ay^2}4} v^{-2} \frac{q}{1+p^2+q^2} (\dz v).
\end{multline*}
We shall only prove estimate~\eqref{eq:estN31} for $N_{3,1}$;
the result for $N_{3,2}$ is easier, when one recalls from estimate~\eqref{eq:thetaDe} that
$v^{-2} \dz^2 v,\ v^{-1} \dz \dy v = \mathcal{O} (\beta^{\frac{33}{20}})$.
For the term $v^{-2} \dz v$ above, we apply the interpolation result in
Lemma~\ref{InterpEst} to inequality~\eqref{eq:thetaDe} to get
\begin{align}\label{eq:embedding}
|v^{-2} \dz v|
\ls \max_{\theta\in [0,2\pi]} |v^{-2} \dz^2 v|
\ls \beta^{\frac{33}{20}}.
\end{align}
Then by estimate~\eqref{v-above} that $\langle y\rangle^{-1}v\ls 1$
and estimate~\eqref{eq:thetaDe}, we obtain
$$\langle y\rangle^{-3} \left[ v^{-2} - \frac12\right] |\dz^2v|
\ls v^{-2}| \dz^2 v|  \ls \beta^{\frac{33}{20}}, $$
which gives the desired result.
\smallskip

Now consider inequality~\eqref{eq:MplusN3}.
For $k=1$, we use the identity $P_3^\alpha (1-P_3^\alpha) = 0$
to rewrite $P^\alpha_3 D^{(1)}_{3,0}$ as
$$P^\alpha_3 D_{3,0}^{(1)} (\sigma)
=  P^\alpha_3 \frac{\alpha+\frac12}{2+b(\tau(\sigma))z^2} b (\tau (\sigma))
z^2(1- P^\alpha_3 )\eta(\sigma).$$
Then direct computation gives the desired estimate,
\begin{align*}
\| e^{\frac{\alpha}4 z^2} P^\alpha_3
D_{3,0}^{(1)} (\sigma) \|_{3,0}
&\ls \max_{z\in \mathbb{R}}
\left| \langle z\rangle^{-1} \frac{b(\tau(\sigma)) z^2}{1+bz^2}\right| \
\Big| \langle z\rangle^{-2}  e^{\frac{\alpha}4  z^2} (1- P^\alpha_3 ) \eta (\sigma)\Big| \\
&\ls  b^{ \frac12 } (\tau(\sigma)) \| e^{\frac{\alpha}4 z^2} \eta (\sigma) \|_{3,0}\\
&\ls  \beta^2 (\tau(\sigma)) M_{3,0}(T).
\end{align*}
Here we used the fact that $|b(\tau)| \le 2\beta (\tau)$, which is implied by
$B(\tau)\le  \beta^{- \frac1{10} } (\tau)$, and the fact for any $m\ge  1$, one has
\begin{equation}\label{eq:estPro}
\| e^{\frac{\alpha}4 z^2} (1- P^\alpha_m )g \|_{m-1,0}
\ls  \| e^{\frac{\alpha}4 z^2} g \|_{m,0}
\end{equation}
by definition~\eqref{eq:projection} of the orthogonal projections
$\overline{P}^\alpha_n \equiv1 - P_m^\alpha$.
\smallskip

Now we estimate $D^{(k)}_{3,0}$ for $k=2,3$.
By estimate~\eqref{eq:estPro} and direct computation, we have
$$\sum_{k=2}^5 \| e^{\frac{\alpha}4 z^2} P_3^\alpha  D_{3,0}^{(k)} (\sigma) \|_{3,0}
\le  \| e^{\frac{\alpha}4 z^{2}} \mathcal{F} (a,b) (\sigma)\|_{3,0}
+ \| e^{\frac{\alpha}4 z^2} \mathcal{W} \eta (\sigma) \|_{3,0}.$$
By estimate~\eqref{eq:compare} that $\frac{\lambda_1}{\lambda} - 1 = \mathcal{O}(\beta)$, one
then obtains
$$\| e^{\frac{\alpha}4 z^2} \mathcal{F} (\sigma)\|_{3,0}
\ls \| F(a,b)(\tau(\sigma)) \|_{3,0}
\le  \beta^2 P (M(T),A(T)),$$
where $F(a,b)$ is defined in  equation~\eqref{eq:source} and estimated in \eqref{eq:estSource}.
Recalling definition~\eqref{eq:defineTW} for $\mathcal{W} = \mathcal{O}(\beta)$, one has
$$\| e^{\frac{\alpha}4 z^2} \mathcal{W} \eta\|_{3,0}
\ls \|\mathcal{W} \|_\infty  \| e^{\frac{\alpha}4 z^2} \eta\|_{3,0}
\ls \beta^{\frac{13}5} M_{3,0}.
$$
The first inequalities in estimates~\eqref{eq:estN1}--\eqref{eq:estN2} are derived in almost
exactly the same way as estimate~\eqref{eq:sample}; so we omit further details here.

The proof of the second inequality in estimate~\eqref{eq:estN1} makes it necessary to
prove a bound for $\|\langle y\rangle^{\frac{11}{10}} e^{\frac{a}4 y^2}\xi\|_{\infty}$.
Specifically,
\begin{align}
\|\langle z\rangle^{-3} D_{3,0}^{(4)} (\sigma)\|_\infty
&\ls   \| \langle y\rangle^{-3} e^{\frac{a}4 y^2} N_1 (\tau(\sigma))\|_\infty \nonumber\\
\noalign{\vskip6pt}
&\ls   \Big\| \langle y\rangle^{-3} e^{\frac{a}2 y^2} \frac1{2+by^2} \xi^2\Big\|_\infty \nonumber\\
\noalign{\vskip6pt}
&\ls   \| \langle y\rangle^{-3} e^{\frac{a}4 y^2} \xi \|_\infty
\Big\| \frac1{2+by^2}  e^{\frac{a}4 y^2} \xi \Big\|_\infty \nonumber.
\end{align}
By the fact that $\frac1{1+by^2}\ls  \beta^{-\frac{11}{20}} \langle y\rangle^{-\frac{11}{10}}$
and the definitions of $M_{3,0}$ and $M_{11/10,0}$, we obtain the desired estimate ~\eqref{eq:estN1}.

Next we turn to the second inequality in \eqref{eq:estN2}.
Compute directly to obtain
\begin{align}\label{eq:N230}
|\langle y\rangle^{-3} e^{\frac{a}4 y^2} N_2 (a,b,\xi)|
\ls | \dy^2 v|\,|\langle y\rangle^{-1} \dy v|^2.
\end{align}
By estimate~\eqref{eq:yDe}, the first term on the \textsc{rhs}
is bounded by $| \dy^2 v| \ls \beta^{\frac{13}{20}}$.
For the second term, we recall that
$v = \sqrt{ 2+by^2 / a+\frac12} + e^{\frac{ay^2}4}\xi$.
So by direct computation, one has
\begin{align*}
|\langle y\rangle^{-1} \dy v |
&\ls  b+\langle y\rangle^{-1} | \dy e^{\frac{ay^2}4} \xi |\\
&\le    b+b^{\frac{11}{10}} M_{1,1}\\
&\ls   \beta(1+M_{1,1}).
\end{align*}
Feed these estimates into inequality~\eqref{eq:N230} to obtain \eqref{eq:estN2}.

Collecting the estimates above completes the proof of the lemma.
\end{proof}

Now we resume our study of equation~\eqref{eq:eta3}. For its first term,
we use the slow decay of $\beta(\tau)$ and estimate~\eqref{eq:compareXiEta} to get
\begin{equation}\label{eq:estIni}
e^{-\gamma S} \beta^{-\frac85} (T) \| e^{\frac{\alpha}4 z^2} \eta(0)\|_{3,0}
\ls \beta^{-\frac85} (0) \| e^{\frac{\alpha}4 z^2} \eta(0)\|_{3,0}
\ls  M_{3,0}(0).
\end{equation}
For its second term, we use the bounds for $D_{3,0}^{(k)}$ from Lemma~\ref{LM:EstDs} and
the integral estimate~\eqref{INT} proved in Lemma~\ref{SimpleIntegralEstimate} below to obtain
\begin{equation}\label{eq:estD30}
\begin{split}
&\sum_{k=1}^6 \int_0^S e^{-\gamma(S-\sigma)} \| e^{\frac{\alpha}4 z^2}
D^{(k)}_{3,0}(\sigma) \|_{3,0}\,\mathrm{d}\sigma\\
\noalign{\vskip6pt}
&\qquad \ls
\left[\int_0^S e^{-\gamma (S-\sigma)} \beta^{\frac{17}{10}}(\sigma)(\sigma)\,\mathrm{d}\sigma\right]
     P(M(T),A(T))\\
\noalign{\vskip6pt}
&\qquad \ls
\beta^{\frac{17}{10}} (T) P(M(T),A(T)).
\end{split}
\end{equation}
By estimate~\eqref{eq:AgreeEnd}, proved in Lemma~\ref{Comparison} below, we have
$\| e^{\frac{\alpha}4 z^2} P^\alpha_3 \eta(S) \|_{3,0} = \| \phi(\cdot,T)\|_{3,0},$
which together with  \eqref{eq:estIni} and \eqref{eq:estD30} implies that
$$\beta^{-\frac85} (T)\| \phi(\cdot,T)\|_{3,0}
\ls M_{3,0}(0) + \beta^{\frac1{10} } (T) P (M(T),A(T)),$$
where $P$ is a nondecreasing polynomial.
By definition~\eqref{eq:majorMmn} of $M_{3,0}$, we obtain
$$M_{3,0} (T) \ls  M_{3,0}(0) + \beta^{ \frac1{10}} (0) P (M(T),A(T)).$$
Since $T$ was arbitrary, estimate~\eqref{eq:M30} will follow, once we provide
a proof of inequalities~\eqref{INT} and \eqref{eq:AgreeEnd} used above.
\medskip

Here is the first of the two promised lemmas.
\begin{lemma}\label{SimpleIntegralEstimate}
For any $c_1,c_2 > 0$, there exists a constant $c(c_1,c_2)$ such that
\begin{equation}\label{INT}
\int_0^S e^{-c_1 (S-s)} \beta^{c_2} (\tau(s))\,\mathrm{d}s \le  c(c_1,c_2) \beta^{c_2} (T).
\end{equation}
\end{lemma}

\begin{proof}
The estimate is based on a simpler observation:
there exists $C(c_1,c_2)$ such that for any $S\ge  0$, one has
\begin{equation}
\int_0^S \frac{e^{-c_1  (S-s)}}{(1+s)^{c_2}}\,\mathrm{d}s
\le   \frac{C(c_1,c_2)}{(1+S)^{c_2}}.
\end{equation}
This is proved by separately considering the regions $s\in [0,\frac{S}2]$ and
$s\in [\frac{S}2,S]$, and comparing the sizes of $e^{-c_1 (S-s)}$ and $\frac1{(1+s)^{c_2}}$
there.

The proof of inequality~\eqref{INT} is complicated by the presence of the two time scales
$\sigma$ and $\tau$, which are related by $\sigma(t) := \int_0^t  \lambda_1^{-2} (s)\,\mathrm{d}s$
and $\tau = \int_0^t \lambda^{-2} (s)\,\mathrm{d}s$. In what follows, we prove the desired estimate
by comparing these two time scales.

By estimate~\eqref{eq:compare}, one has $\frac{\lambda(\tau)}{\lambda_1(\tau(\sigma))}\approx 1$
and thus
\begin{equation}\label{TauS1}
4\sigma\ge  \tau(\sigma)\ge  \frac14 \sigma.
\end{equation}
This implies that $\frac1{\frac1{b_0} + \tau(\sigma)}\le 4\frac1{\frac1{b_0} +\sigma}$,
which in turn gives
\begin{equation}\label{InI2}
\int_0^S e^{-c_1 (S-s)}\beta^{c_2} (\tau(s))\,\mathrm{d}s
\le   \frac{c(c_1,c_2)}{(\frac1{b_0} +S)^{c_2}}.
\end{equation}
Recall that $\tau(S) = T$ and use estimate~\eqref{TauS1} again to obtain
\begin{align*}
\frac1{(\frac1{b_0} +S)^{c_2}}
\ls \frac1{(\frac1{b_0} +T)^{c_2}}.
\end{align*}
This together with inequality~\eqref{InI2} implies the desired estimate.
\end{proof}

The next lemma relates $\phi := e^{\frac{ay^2}4} \xi$ and
$\eta$ at times $\sigma = S$ and $\tau = T$.

\begin{lemma}\label{Comparison}
If $m+n\le  3$ and $\ell\ge  0$, then
\begin{equation}\label{eq:AgreeEnd}
\langle z\rangle^{-\ell}
e^{\frac{\alpha z^2}4} P^\alpha_m \left(\partial_z + \frac{\alpha}2  z\right)^n \eta(z,S)
= \langle y\rangle^{-\ell} \partial^n_y \phi (y,T).
\end{equation}
\end{lemma}

\begin{proof}
By definition of the various quantities in equation~\eqref{NewFun}, one has
$\lambda_1(t(T)) = \lambda (t(T))$, $a(T) = \alpha$, and hence $z=y$
and $e^{\frac{\alpha}4 z^2} \eta (z,S) = \phi (y,T)$.
Then because
$e^{\frac{\alpha}4 z^2} (\partial_z +\frac{\alpha}2 z) e^{-\frac{\alpha}4 z^2} = \partial_z$,
we have
\begin{equation}\label{eq:derivPro}
\langle z\rangle^{-\ell}
e^{\frac{\alpha z^2}4} P^\alpha_m \left(\partial_z + \frac{\alpha}2 z\right)^n \eta (z,\theta,S)
= \langle y\rangle^{-\ell} e^{\frac{a(T) y^2}4}  P^{a(T)}_m
e^{-\frac{a(T) y^2}4}  \dy^n \phi (y,\theta,T).
\end{equation}
Integrating by parts and using the orthogonality conditions
$\xi(\cdot,\tau) \perp \phi_{k,a(\tau)}$ for $k=0,1,2$, with
$\phi_{k,a}$ defined in equation~\eqref{BadEvals}, one gets
$$e^{-\frac{a(\tau) y^2}4} \partial^n_y \phi (\cdot,\tau)
\perp  \phi_{k,a(\tau)}, \qquad (0\le  k\le  2-n),$$
which shows that
$$P^{a(T)}_m e^{-\frac{a(T) y^2}4} \dy^n \phi (y,\theta,T)
= e^{-\frac{a(T) y^2}4} \dy^n \phi(y,\theta,T)$$ if $m+n\le  3$.
This together with equation~\eqref{eq:derivPro} implies  identity~\eqref{eq:AgreeEnd}.
\end{proof}
\smallskip
Our proof of estimate~\eqref{eq:M30} is complete.

\subsection{Proof of estimate~\eqref{eq:M20}}\label{SEC:EstM2}
In this section, we prove estimate~\eqref{eq:M20} for the function $M_{11/10,0}$.
To begin, we use equation~\eqref{eq:eta} to compute that $P^\alpha_2 \eta (\sigma)$
evolves by
\begin{equation}\label{eq:eta4}
\partial_\sigma(P^\alpha_2 \eta) = - L_\alpha P^\alpha_2 \eta - P^\alpha_2
V \eta + P_2^\alpha  \sum_{k=2}^6  D_{2,0}^{(k)}
\end{equation}
where the functions $D^{(k)}_{2,0}$ and the operator $L_\alpha$ (which differs
from $\mathcal{L}_\alpha$ by the absence of the potential $V$) are defined
immediately after equations~\eqref{EQ:eta2} and  \eqref{eq:eta}, respectively.
Our first step is to obtain the following bounds.

\begin{lemma}
If $ A(\tau),\ B(\tau) \ls  \beta^{- \frac1{10} }(\tau)$, then
\begin{equation}\label{eq:estF3}
\|e^{\frac{\alpha}4 z^2} V \eta (\sigma)\|_{\frac{11}{10},0}
\ls \beta^{\frac{13}{20}} (\tau(\sigma)) M_{3,0}(T),
\end{equation}
\begin{equation}\label{eq:1.1}
\sum_{k=2}^{5} \| e^{\frac{\alpha}4 z^2} P_2^\alpha D_{2,0}^{(k)} (\sigma)\|_{\frac{11}{10},0}
\ls  \beta^{\frac7{10}} (\tau) P (M(T),A(T)),
\end{equation}
\begin{equation}\label{eq:estN32}
\| e^{\frac{\alpha}4  z^2} P_2^\alpha  D_{2,0}^{(6)} (\sigma)\|_{\frac{11}{10},0}
\ls \| e^{\frac{a y^2}4}  N_3 (a,b,\xi)\|_{\frac{11}{10},0}
\ls \beta^{\frac{13}{20}}  (\tau) (1 + \epsilon_0 M_{\frac{11}{10},0}),
\end{equation}
where $\epsilon_0$ is the small constant that appears in estimate~\eqref{eq:smallness}.
\end{lemma}

\begin{proof}
We start by proving inequality~\eqref{eq:estF3}.
By Condition~[Cb], we have $\frac1b \ls \frac1{\beta}$ and hence
$$ \frac{\langle z\rangle^{-\frac{11}{10}}}{1+b(\tau(\sigma)) z^2}
\le  \langle z\rangle^{-\frac{11}{10}}
\left(1 + b(\tau(\sigma))z^2\right)^{-\frac{19}{20}}
\ls \beta^{-\frac{19}{20}} (\tau(\sigma))  \langle z\rangle^{-3},$$
which, together with estimate~\eqref{eq:compareXiEta} and
the definition of $V$ following equation~\eqref{eq:eta}, yields
\begin{align*}
\|e^{\frac{\alpha}4 z^2} V \eta (\sigma)\|_{\frac{11}{10},0}
&\ls
\Big\| \frac1{1+b(\tau(\sigma))z^2}  e^{\frac{\alpha}4 z^2} \eta (\sigma)\Big\|_{\frac{11}{10},0}\\
\noalign{\vskip6pt}
&\ls   \beta^{-\frac{19}{20}} (\tau(\sigma)) \| e^{\frac{\alpha}4 z^2} \eta(\sigma)\|_{3,0}\\
\noalign{\vskip6pt}
&\ls  \beta^{\frac{13}{20}} (\tau(\sigma)) M_{3,0}(T).
\end{align*}
This gives  inequality~\eqref{eq:estF3}.

The proof of inequality~\eqref{eq:1.1} is almost identical to that of \eqref{eq:MplusN3},
thus is omitted here. (It may be compared with the corresponding part of \cite{GS09}.)

Now we turn to inequality~\eqref{eq:estN32}.
As in equation~\eqref{eq:N312}, we decompose the new term $N_3$ into two parts,
$N_3 = N_{3,1} + N_{3,2}$.
We only prove estimate~\eqref{eq:estN31} for $N_{3,1}$;
the estimate for $N_{3,2}$ are easier and use the sufficiently fast decay
estimates for $v^{-2} \dz^2 v,\ v^{-1} \dz  \dy v,\ v^{-2}\dz v = \mathcal{O} (\beta^{\frac{33}{20}})$
found in  \eqref{eq:thetaDe} and  \eqref{eq:embedding}.

By direct computation, we find that
$$
\langle y\rangle^{-\frac{11}{10}} e^{\frac{a y^2}4} N_{3,1}
= \frac12\ \frac{\sqrt2  + v}v
\langle y\rangle^{-\frac{11}{10}}
\Big[\sqrt2 - v\Big] v^{-1} \dz^2 v.$$
The lower bound $v\ge  1$ from \eqref{eq:lower} and condition~[Cb] imply that
\begin{equation}
\Big| \langle y\rangle^{-\frac{11}{10}} e^{\frac{a y^2}4} N_{3,1} \Big|
\ls  \langle y\rangle^{-\frac{11}{10}}
\Big|\sqrt2 - v\Big|\ v^{-1}| \dz^2 v|.  \label{eq:N31parts}
\end{equation}
For the term $\sqrt2 - v$, we use the decomposition $v=e^{\frac{a}4 y^2} w$
in  equation~\eqref{eqn:split3} to write
\begin{align*}
\sqrt2 - v
& =  \sqrt2 - \sqrt{\frac{2 + by^2}{a+\frac12}} - e^{\frac{a}4 y^2}\xi\\
\noalign{\vskip6pt}
&=  - \left(a + \frac12\right)^{-\frac12}  \frac{by^2}{\sqrt2 + \sqrt{2+b y^2}}
+ \frac{\sqrt2}{\sqrt{a + \frac12}} \frac{a - \frac12}{1 + \sqrt{a+\frac12}}
- e^{\frac{a}4 y^2} \xi.
\end{align*}
Substitute this into estimate~\eqref{eq:N31parts} to obtain
\begin{align*}
\Big| \langle y\rangle^{-\frac{11}{10}}  e^{\frac{a y^2}4}  N_{3,1} \Big|
&\ls  \langle y\rangle^{-\frac{11}{10}}
\left[ \frac{by^2}{\sqrt{2+by^2}} + |1-2a| + | e^{\frac{ay^2}4}\xi |\right]
v^{-1} |\dz^2 v|\\
\noalign{\vskip6pt}
&=:   C_1 + C_2 + C_3.
\end{align*}
To conclude, we shall estimate the three components on the \textsc{rhs}.

The estimate for $C_1$ is the most involved.
By direct computation, we obtain
$$C_1 \le  b^{\frac12} \langle y\rangle^{-\frac1{10}} v^{-1}  | \dz^2 v|.$$
Using the estimates $v^{-1} |\dz^2 v| \le  \epsilon_0$ from  \eqref{eq:smallness},
$v^{-2} |\dz^2 v| \ls \beta^{\frac{17}{10}}$ from \eqref{eq:thetaDe}, and
$\frac{v}{1+|y|} \ls 1$ from \eqref{v-above}, the consequence $\beta\approx b$
of Condition~[Cb] for $B$, and H\"older's inequality, we get
\begin{equation}
C_1  \ls b^{\frac12} \left[v^{-2} | \dz^2 v|\right]^{\frac1{10}}
\left[v^{-1} | \dz^2 v|\right]^{\frac9{10}}
\ls \epsilon_0^{\frac9{10}}  b^{\frac{133}{200}}.\label{eq:sharpEst}
\end{equation}
For $C_2$, we use the fact $a-\frac12 = \mathcal{O}(b) + \mathcal{O}(\beta^2 A)$
from estimate~\eqref{eq:majorA} and Condition~[Cb] to obtain
\begin{equation}
C_2 \ls \epsilon_0 \beta.\label{eq:estC2}
\end{equation}
For $C_3$, we use definition~\eqref{eq:majorMmn} of $M_{1,1}$ to get
\begin{equation}
C_3\le  \epsilon_0 M_{1,1} \beta^{\frac{13}{20}}.\label{eq:estC3}
\end{equation}

Collecting inequalities~\eqref{eq:sharpEst}--\eqref{eq:estC3} completes
the estimate for
$$\Big\| \langle y\rangle^{-\frac{11}{10}}  e^{\frac{a y^2}4} N_{3,1}\Big\|_\infty.$$
\end{proof}

Now fix a ($\tau$-scale) time $T$.
To conclude our proof of ~estimate~\eqref{eq:M20}, we continue to study
equation~\eqref{eq:eta4}, which may by Duhamel's principle be rewritten in the form
$$P^\alpha_2 \eta(S) = e^{- L_\alpha S} P^\alpha_2  \eta(0)
+ \int_0^S  e^{- L_\alpha (S-\sigma)} P^\alpha_2
\bigg[- V\eta +  \sum_{k=2}^6  D_{2,0}^{(k)}\bigg]\,\mathrm{d}\sigma, $$
where $S=S(t(T))$ is defined in equation~\eqref{T2}.
By estimate~\eqref{eq:estproject2} for $e^{- L_\alpha \sigma} P^\alpha_2 $, one has
\begin{equation}\label{K123s}
\Big\| e^{\frac{\alpha}4 z^2} P^\alpha_2 \eta(S)\Big\|_{\frac{11}{10},0}
\ls K_0 +K_1 + K_2,
\end{equation}
where
\begin{align*}
K_0 & := e^{-\alpha  S} \Big\| e^{\frac{\alpha}4 z^2} \eta(0)\Big\|_{\frac{11}{10},0},\\
\noalign{\vskip6pt}
K_1 & := \int_0^S  e^{-\alpha (S-\sigma)}
\Big\| e^{\frac{\alpha}4 z^2} V \eta(\sigma)\Big\|_{\frac{11}{10},0}\,\mathrm{d}\sigma,\\
\noalign{\vskip6pt}
K_2 & :=   \sum_{k=2}^6 \int_0^S  e^{-\alpha (S-\sigma)}
\Big\| e^{\frac{\alpha}4 z^2} D_{2,0}^{(k)}\Big\|_{\frac{11}{10},0}\,\mathrm{d}\sigma.
\end{align*}
Next, we estimate the $K_n$ for $n=0,1,2$.
\begin{itemize}
\item[(K0)]
Estimate~\eqref{eq:compareXiEta} and the slow decay of $\beta$ yield
\begin{equation}\label{eq:estK0}
K_0
\ls \beta^{\frac{13}{20}} (T) \beta^{-\frac{13}{20}} (0)
\Big\| e^{\frac{\alpha}4 z^2} \eta(0)\Big\|_{\frac{11}{10},0}
\ls  \beta^{\frac{13}{20}} (T) M_{\frac{11}{10},0}(0).
\end{equation}
\item[(K1)]  Estimate~\eqref{eq:estF3} and the integral inequality~\eqref{INT} imply that
\begin{equation}\label{EstK1}
K_1
\ls \left[\int_0^S  e^{-\alpha (S-\sigma)} \beta^{\frac{13}{20}}(\tau(\sigma))\,
    \mathrm{d}\sigma\right] M_{3,0} (T)
\ls \beta^{\frac{13}{20}}(T) M_{3,0}(T).
\end{equation}
\item[(K2)]
The estimates of $D_{2,0}^{(k)}$, $(k=2\ldots6)$, implied by \eqref{eq:1.1} yield the bounds
\begin{equation}\label{eq:estK2}
\begin{split}
K_2 &\ls   \int_0^S  e^{-\alpha (S-\sigma)}
 \Big[ \beta^{\frac34} (\tau(\sigma)) P(M(T),A(T))
 + \beta^{\frac{13}{20}} \epsilon_0 M_{\frac{11}{10},0} (T) ]\,\mathrm{d}\sigma\\
 \noalign{\vskip6pt}
 &\ls  \beta^{\frac34} (T) P(M(T),A(T))
 + \beta^{\frac{13}{20}} \epsilon_0 M_{\frac{11}{10},0}(T).
\end{split}
\end{equation}
\end{itemize}
\smallskip
Collecting estimates~\eqref{K123s}--\eqref{eq:estK2}, one obtains
\begin{equation}\label{FinalStep}
\begin{split}
&\beta^{-\frac{13}{20}} (T)
\Big\| e^{\frac{\alpha}4 z^2} P^\alpha_2 \eta(\tau(S)) \Big\|_{\frac{11}{10},0} \\
\noalign{\vskip6pt}
&\qquad
\ls  M_{\frac{11}{10},0}(0) + M_{3,0} (T) + \epsilon_0 M_{\frac{11}{10},0}
+ 1 + \beta^{ \frac1{10}} (0) P (M(T),A(T)).
\end{split}
\end{equation}
Equation~\eqref{eq:AgreeEnd} shows that
$$\beta^{-\frac{13}{20}} (T)\| \phi(T)\|_{\frac{11}{10},0}
= \beta^{-\frac{13}{20}} (T)
\Big\| e^{\frac{\alpha z^2}4} P^\alpha_2 \eta (S)\Big\|_{\frac{11}{10},0},$$
which together with estimate~\eqref{FinalStep} and the definition of
$M_{11/10,0}$ implies that
$$M_{\frac{11}{10},0}(T)
\ls   M_{\frac{11}{10},0} (0) + M_{3,0} (T)
+ \epsilon_0 M_{\frac{11}{10},0} (\tau) + \beta^{ \frac1{10}} (0) P (M(T),A(T)).$$
Since the time $T$ was arbitrary, the proof of estimate~\eqref{eq:M20} is complete.

\subsection{Proof of estimate~\eqref{eq:M12}}\label{SEC:estM12}
In this section, we prove estimate~\eqref{eq:M12} for the function $M_{1,1}$.
Observe that for any smooth function $g$, one has
\begin{equation}\label{eq:obser}
e^{-\frac{\alpha z^2}4} \partial_z   \Big( e^{\frac{\alpha}4 z^2 }g\Big)
= \Big( \partial_z + \frac{\alpha}2  z \Big) g.
\end{equation}
Thus it follows from equation~\eqref{eq:eta} that the function
$P^\alpha_1 (\partial_z + \frac{\alpha}2 z)\eta$ evolves by
\begin{equation}\label{eq:oneD}
\partial_\sigma\left[P^\alpha_1 \Big(\partial_z + \frac{\alpha}2 z\Big) \eta\right]
= - P^\alpha_1 (L_\alpha + \alpha) P^\alpha_1 \Big(\partial_z + \frac{\alpha}2 z\Big)\eta
+ \sum_{k=2}^6  D^{(k)}_{1,1}+D,
\end{equation}
where $D^{(k)}_{1,1}$ is defined immediately after equation~\eqref{EQ:eta2}, and
\begin{align*}
D & :=   - P^\alpha_1 \Big(\partial_z + \frac{\alpha}2 z\Big) V\eta\\
\noalign{\vskip6pt}
& \ =  - P^\alpha_1 [ \partial_z V] \eta - P_1^\alpha V \Big[\partial_z + \frac{\alpha}2 z\Big]\eta.
\end{align*}
The key observation here is that applying the operator $\partial_z + \frac{\alpha}2 z$
gives an equation with an improved linear part.

Fix a ($\tau$-scale) time $T$. Then Duhamel's principle lets us rewrite
equation~\eqref{eq:oneD} as
\begin{equation}\label{eq:durha2}
\begin{split}
P^\alpha_1 \Big(\partial_z + \frac{\alpha}2 z\Big)\eta (S)
&= P^\alpha_1  e^{-(L_\alpha +\alpha)S} P^\alpha_1
\Big(\partial_z +\frac{\alpha}2 z\Big)\eta(0)\\
\noalign{\vskip6pt}
&\qquad
+ \int_0^S P^\alpha_1 e^{-(L_\alpha +\alpha)(S-\sigma)} P^\alpha_1
\bigg[ \sum_{n=2}^6 D_{1,1}^{(n)}+D \bigg] d\sigma,
\end{split}
\end{equation}
where $P^\alpha_1 e^{-(L_\alpha +\alpha)S} P^\alpha_1$ is defined
in equation~\eqref{eq:projection} and estimated in \eqref{eq:firstProj}.
For the remaining terms on the \textsc{rhs}, we derive the following estimates.

\begin{lemma}\label{LM:remaind11}
If $A(\tau),\ B(\tau)\ls \beta^{- \frac1{20} }(\tau)$, then:
\begin{equation}\label{eq:estDD}
\big\| e^{\frac{\alpha}4 z^2} D(\sigma)\big\|_{1,0}
\ls \beta^{\frac{11}{10}} (\tau (\sigma)) \big[M_{3,0}(T) + M_{2,1}(T)\big],
\end{equation}
\begin{equation}\label{eq:estD11}
\bigg\| e^{\frac{\alpha z^2}4}  \sum_{k=2,3} D_{1,1}^{(k)} (\sigma)\bigg\|_{1,0}
\ls  \beta^{\frac65}  (\tau (\sigma)) P\big(M(T),A(T)\big),
\end{equation}
\begin{align}\label{eq:mn11}
\big\| e^{\frac{\alpha}4 z^2}  D_{1,1}^{(4)} \big\|_{1,0}
& \ls \big\| e^{\frac{a y^2}4} N_1  (a,b,\xi) \big\|_{1,1}\\
\noalign{\vskip6pt}
& \ls \beta^{\frac{11}{10}}  [M_{2,1} + M_{3,0}]
+ \beta^{\frac65} \Big(M_{\frac{11}{10},0}^2 + M_{1,1} M_{\frac{11}{10},0}\Big),\notag
\end{align}
\begin{equation}\label{eq:N211}
\big\| e^{\frac{\alpha}4 z^2} D_{1,1}^{(5)}\big \|_{1,0}
\ls \big\| e^{\frac{a y^2}4} N_{2}(a,b,\xi)\big\|_{1,1}
\ls \beta^{\frac65}(\tau),
\end{equation}
\begin{equation}\label{eq:N31II}
\big\| e^{\frac{\alpha}4 z^2} D_{1,1}^{(6)} \big\|_{1,0}
\ls \big\| e^{\frac{a y^2}4} N_3  (a,b,\xi)\big\|_{1,1}
\ls \beta^{\frac{23}{10}} (\tau).
\end{equation}
\end{lemma}

\begin{proof}
We first prove estimate~\eqref{eq:estDD}.
Direct computation leads to
\begin{align*}
|\langle z\rangle e^{\frac{\alpha}4 z^2} D(\sigma)|
& \ls  \frac{b}{(1 + b (\tau(\sigma))z^2)^2}
| e^{\frac{\alpha}4 z^2} \eta | + \frac1{\langle z\rangle (1+b(\tau(\sigma))z^2)}
| \partial_z e^{\frac{\alpha}4 z^2} \eta |\\
\noalign{\vskip6pt}
&\le   b^{-\frac12} \langle z\rangle^{-3}
| e^{\frac{\alpha}4 z^2} \eta | + b^{-\frac12} \langle z\rangle^{-2}
| \partial_z e^{\frac{\alpha}4 z^2} \eta |.
\end{align*}
To control the \textsc{rhs}, we apply inequality~\eqref{eq:compareXiEta} to obtain the desired estimate.

The proof of  inequality~\eqref{eq:estD11} is very similar to the corresponding part of
Lemma~\ref{LM:EstDs}, hence is omitted.

The first inequalities in  estimates~\eqref{eq:mn11}--\eqref{eq:N31II} can be proved by the same
method as estimate~\eqref{eq:sample}. So in what follows, we only prove the second ones.

Consider the second inequality in estimate~\eqref{eq:mn11} --- an estimate which was not necessary
in \cite{GS09}. By direct computation, we obtain
\begin{align*}
\dy N_1(a,b,\xi)
& = \frac{\dy v}{v^2}  \frac{a+\frac12}{2+by^2}
\big(e^{\frac{a}4 y^2} \xi\big)^2 + \frac1v  \frac{2(a + \frac12 )by}{(2 + by^2)^2}
\big(e^{\frac{a}4 y^2} \xi\big)^2 \\
&\qquad -  \frac2v  \frac{a+\frac12}{2+by^2}
e^{\frac{a}4 y^2} \xi  \dy \big(e^{\frac{a}4 y^2} \xi\big)
\\
\noalign{\vskip6pt}
& =: A_1 + A_2 + A_3.
\end{align*}
To bound $A_1$ and $A_2$, we note that $|v^{-1} \dy v| \ls \beta^{\frac12}$ by \eqref{eq:yDe},
and use the consequences $|\frac{by}{\sqrt{2+by^2}}| \ls  \beta^{\frac12}$ and $|a|\le  1$
of \eqref{eq:assuMAB} in Condition~[Cb] to obtain
\begin{equation*}
\langle y\rangle^{-1} (|A_1| + |A_2|)
\ls  b^{\frac12} \langle y\rangle^{-1} \frac1{1+by^2} \big(e^{\frac{a}4 y^2}\xi\big)^2
\ls   b^{\frac12} b^{-\frac35} \big(\langle y\rangle^{-\frac{11}{10}} e^{\frac{a}4 y^2}\xi \big)^2.
\end{equation*}
It follows that
\begin{equation}\label{eq:estA1A2}
\big\| \langle y\rangle^{-1} (A_1 + A_2) \big\|_{L^\infty}
\ls b^{-\frac1{10}}
\big\| \langle y\rangle^{-\frac{11}{10}} e^{\frac{a}4 y^2} \xi \big\|_\infty^2
\le  \beta^{\frac65} M_{\frac{11}{10},0}^2.
\end{equation}
It is not hard to bound $A_3$ via the observation
\begin{align}
\Big\| \langle y\rangle^{-1}\frac2v \frac{a + \frac12}{2+by^2}
e^{\frac{a}4 y^2} \xi \dy
(e^{\frac{a}4 y^2} \xi ) \Big\|_\infty
&\ls   b^{-\frac{11}{20}}
\big\| \langle y\rangle^{-\frac{11}{10}} e^{\frac{a}4 y^2} \xi \big\|_\infty
\big\|\langle y\rangle^{-1} \dy (e^{\frac{a}4 y^2}\xi )\big\|_\infty   \nonumber\\
&\le   \beta^{\frac65} M_{1,1} M_{\frac{11}{10},0}.\nonumber
\end{align}
When combined with estimate~\eqref{eq:estA1A2}, this implies \eqref{eq:mn11}.

To derive the second inequality in estimate~\eqref{eq:N211}, we recall that
$p = \dy v$ and $q=v^{-1} \dz v$ and then compute that
$$
\dy \left[ e^{\frac{a}4 y^2} N_2(a,b,\xi) \right]
= \frac{2p^2(p \dy^2 v + q v^{-1} \dy \dz v - q^2 v^{-1} \dy v)}{(1+p^2+q^2)}
\dy^2 v - \frac{p^2 \dy^3 v + 2p (\dy^2 v)^2}{1+p^2+q^2}.
$$
Recall that we need to bound
$\langle y\rangle^{-1} \dy \left[ e^{\frac{a}4 y^2} N_2(a,b,\xi) \right]$.
To control $\langle y\rangle^{-1} |\dy v |$, we use Condition~[Cb] for $M_{1,1}$,
which implies that
\begin{equation}\label{eq:DeV}
(1+|y|)^{-1} | \dy v| \ls  2\beta.
\end{equation}
To see this, we write $v = \sqrt{2 + by^2/a+\frac12} + \phi$
and use the fact that
$$\langle y\rangle^{-1} | \dy \phi | \le  \beta^{\frac{21}{20}} M_{1,1} \ll \beta.$$
This observation, together with estimates~\eqref{eq:yDe}--\eqref{eq:thetaDe} for
$|\dy^n v|$, $(n=2,3)$, and $v^{-1} |\dy \dz v|$ yields the second inequality in
estimate~\eqref{eq:N211}.

Now we turn to \eqref{eq:N31II}. As in equation~\eqref{eq:N312}, we decompose
the new term $N_3$ into two parts, writing $N_3 = N_{3,1} + N_{3,2}$.
We will prove estimate~\eqref{eq:estN31} for $N_{3,1}$; the proof for
$N_{3,2}$ follows readily from the decay estimates for
$v^{-2} \dz^2 v$, $v^{-1} \dz \dy v$, $v^{-2} \dz v$,
and their $y$-derivatives provided by estimates~\eqref{eq:thetaDe}
and  \eqref{eq:embedding}. By direct computation, we have
\begin{equation}\label{eq:N31Two}
\dy \big[ e^{\frac{ay^2}4} N_{3,1} \big]
= - 2v^{-3} \dy v \dz^2 v + \frac12 \Big[ \frac2{v^2}-1 \Big]
\dy \dz^2 v.
\end{equation}
For the first term on the \textsc{rhs}, we use estimates~\eqref{eq:yDe}--\eqref{eq:thetaDe}
to obtain
\begin{equation}\label{eq:N31Two1}
v^{-3} | \dy v \dz^2 v |
= v^{-1} | \dy v|\ v^{-2}  | \dy v \dz^2 v |
\ls \beta^{\frac{21}{20}}.
\end{equation}
Estimating the second term on the \textsc{rhs} requires more work.
We start by writing in the the more convenient form,
\begin{align*}
\langle y\rangle^{-1} \left| \Big[ \frac2{v^2}-1\Big] \dy \dz^2 v\right|
& =  \langle y\rangle^{-1} \Big| \frac1v + 1\Big|\
\Big| \frac1v - 1\Big|\ | \dy \dz^2 v|\\
\noalign{\vskip6pt}
& \ls  \langle y\rangle^{-1} \Big| \frac1v - 1\Big|^{\frac{10}{11}}
| \dy \dz^2 v|\\
\noalign{\vskip6pt}
=& \big(v^{-2} | \dy \dz^2 v|\big)^{\frac5{11}}
|\dy \dz^2 v|^{\frac6{11}}
\Big[ \langle y\rangle^{-\frac{11}{10}} |v-1| \Big]^{\frac{10}{11}}.
\end{align*}
It follows from estimates~\eqref{eq:thetaDe} and \eqref{eq:smallness} that
the first two factors on the \textsc{rhs} here can be bounded by
$\beta^{\frac34} \epsilon_0^{\frac6{11}}$. To bound the final factor,
we recall that $v = \sqrt{ 2+by^2/a+\frac12} + e^{\frac{a}4 y^2} \xi.$
By definitions~\eqref{eq:majorMmn} and \eqref{eq:majorA} for
$M_{11/10,0}$ and $A$, respectively, we have
\begin{equation}
\langle y\rangle^{-\frac{11}{10}} |v-1|
\ls \beta^{\frac12} + \beta^{\frac1{10}} [A+M_{11/10,0}].
\end{equation}
Collecting the estimates above and using the consequences
$A, M_{11/10,0}\ls  \beta^{-\frac1{20}}$ of Condition~[Cb], we obtain
\begin{equation}
\langle y\rangle^{-1} \left| \Big[ \frac2{v^2} - 1\Big]  \dy \dz^2 v \right|
\ls \beta^{\frac{53}{44}} \epsilon^{\frac6{11}}
\le  \beta^{\frac65}.
\end{equation}
When combined with inequalities~\eqref{eq:N31Two}--\eqref{eq:N31Two1}, this
implies the desired estimate.

Finally, \eqref{eq:N31II} follows easily from the estimates obtained earlier
in the lemma; so we omit further details.
\end{proof}

Now we return to equation~\eqref{eq:durha2}.
Applying the norm $\|\cdot\|_{1,0} = \|\langle y\rangle^{-1} \cdot \|_\infty$
to both sides yields
\begin{equation}\label{eq:estEta}
\left\| e^{\frac{\alpha}4 z^2} P^\alpha_1 \Big(\partial_z + \frac{\alpha}2 z\Big) \eta(S)
\right\|_{1,0}
\ls  Y_1 +Y_2 + Y_3,
\end{equation}
where
\begin{align*}
Y_1 & := e^{-\gamma S} \|
e^{\frac{\alpha}4 z^2} \eta(0) \|_{1,1},\\
Y_2 & := \int_0^S e^{-\gamma (S - \sigma)}  \sum_{k=2}^6
\| e^{\frac{\alpha z^2}4}  D_{1,1}^{(k)} (\sigma) \|_{1,0}\,\mathrm{d}\sigma,\\
Y_3 & := \int_0^S e^{-\gamma (S-\sigma)}
\| e^{\frac{\alpha z^2}4}  D(\sigma) \|_{1,0}\,\mathrm{d}\sigma.
\end{align*}
By  inequalities~\eqref{eq:estDD}--\eqref{eq:N31II} and the integral estimate~\eqref{INT},
one has
\begin{equation}\label{eq:Y2}
Y_2 \ls \left[\int_0^S e^{-\gamma (S-\sigma)} \beta^{\frac{23}{20}}\,\mathrm{d}\sigma\right]
P(M(T),A(T))
\ls \beta^{\frac{23}{20}} (T) P (M(T),A(T)).
\end{equation}
Similar reasoning yields
\begin{equation}
Y_3 \ls \beta^{\frac{11}{10}} (T) \Big[M_{3,0}(T)+M_{2,1}(T)\Big].
\end{equation}
Estimate~\eqref{eq:compareXiEta} and the slow decay of $\beta$ imply that
\begin{equation}\label{eq:Y1}
Y_1 \ls e^{-\gamma S} \| \phi(\cdot,0)\|_{1,1}
\ls \beta^{\frac{11}{10}} (T) (T) M_{1,1}(0).
\end{equation}

Now collecting estimates~\eqref{eq:estEta}--\eqref{eq:Y1}, we obtain
\begin{equation*}
\begin{split}
& \beta^{-\frac{11}{10}} (T)
\left\| e^{\frac{\alpha z^2}4} P^\alpha_1 \Big( \partial_z + \frac{\alpha}2 z\Big) \eta(S)
\right\|_{1,0}\\
\noalign{\vskip6pt}
&\qquad
\ls  M_{1,1} (0) + M_{3,0} (T) + M_{2,1} (T) + \beta^{\frac1{20} }
(0)P(M(T),A(T)).
\end{split}
\end{equation*}
By equation ~\eqref{eq:AgreeEnd}, one has
$\| e^{\frac{\alpha z^2}4} P^\alpha_1 ( \partial_z + \frac{\alpha}2 z)\eta(S)]\|_{1,0}
= \| \phi(T)\|_{1,1}$.
Thus by definition of $M_{1,1}$, one obtains
$$M_{1,1}(T)
\ls M_{1,1} (0) + M_{3,0} (T) + M_{2,1} (T) + \beta^{ \frac1{20} } (0)P(M(T),A(T)).$$
Because $T$ was arbitrary, this proves estimate~\eqref{eq:M12}.

\subsection{Proof of estimate~\eqref{eq:M21}}\label{SEC:EstM21}
In this section, we prove estimate~\eqref{eq:M21} for $M_{2,1}$ and thereby
complete our proof of Proposition~\ref{Prop:Main3}.

The same method used to derive equation~\eqref{eq:oneD} shows that
$P^\alpha_2 (\partial_z + \frac{\alpha}2 z) \eta$ evolves by
\begin{equation}\label{eq:oneD2}
\partial_\sigma\left[P^\alpha_2 \Big(\partial_z + \frac{\alpha}2 z\Big) \eta\right]
= - P^\alpha_2 (\mathcal{L}_\alpha +\alpha) P^\alpha_2
\Big(\partial_z + \frac{\alpha}2 z\Big) \eta
+ \sum_{k=1}^6 D^{(k)}_{2,1} + D_7,
\end{equation}
where $D^{(k)}_{2,1}$, $(k=1,\ldots,6)$, are defined immediately after
equation~\eqref{EQ:eta2}, and
$$D_7 := - P^\alpha_2 \eta \partial_z V.$$

We fix a ($\tau$-scale) time $T$ and apply Duhamel's principle to
rewrite  equation~\eqref{eq:oneD2} in the form
\begin{align}\label{eta21}
P^\alpha_2 \Big( \partial_z + \frac{\alpha}2 z \Big) \eta (S)
& = P^\alpha_2 U_2 (S,0)
e^{-\alpha S} P^\alpha_2 \Big( \partial_z + \frac{\alpha}2 z\Big) \eta (0)\\
\noalign{\vskip6pt}
&\quad
+ \int_0^S P^\alpha_2 U_2 (S,\sigma)
e^{-\alpha (S-\sigma)} P^\alpha_2
\bigg[ \sum_{n=1}^6 D_{2,1}^{(n)}+D_7 \bigg]\,\mathrm{d}\sigma,\nonumber
\end{align}
where $U_2$ is defined in Proposition~\ref{PRO:propagator}
and satisfies estimate~\eqref{eq:OperatorWithV}.
Terms on the \textsc{rhs} may be estimated as follows.

\begin{lemma}
If $A(\tau),\ B(\tau)\ls  \beta^{- \frac1{20} } (\tau)$, then
\begin{equation}\label{eq:estD6}
\Big\| e^{\frac{\alpha}4 z^2} D_7 (\sigma) \Big\|_{2,0}
\ls \beta^{\frac85} (\tau (\sigma)) M_{3,0}(T),
\end{equation}
and
\begin{equation}\label{eq:estD21}
\bigg\| e^{\frac{\alpha z^2}4} \sum_{k=1}^6 D_{2,1}^{(k)} (\sigma)\bigg\|_{2,0}
\ls \beta^{\frac{33}{20}} (\tau(\sigma)) P(M(T),A(T)).
\end{equation}
\end{lemma}

The proof is similar to but easier than that of Lemma~\ref{LM:remaind11},
hence is omitted.
\smallskip

Continuing, we apply the $\|\cdot\|_{2,0} = \|\langle z\rangle^{-2} \cdot\|_\infty$
norm to equation~\eqref{eta21}, obtaining
\begin{equation}\label{eq:estEtaAgain}
\left\| e^{\frac{\alpha}4 z^2} P^\alpha_2 \Big(\partial_z + \frac{\alpha}2 z \Big) \eta (S)
\right\|_{2,0}
\ls Y_1 +Y_2 + Y_3,
\end{equation}
where
\begin{align*}
Y_1 & := e^{-\gamma S} \Big\| e^{\frac{\alpha}4 z^2} \eta(0) \Big\|_{2,1},\\
Y_2 & := \int_0^S e^{-\gamma (S-\sigma)}   \sum_{k=1}^6
\Big\| e^{\frac{\alpha z^2}4} D_{2,1}^{(k)} (\sigma) \Big\|_{2,0}\,\mathrm{d}\sigma,\\
Y_3 & := \int_0^S e^{-\gamma (S-\sigma)}
\Big\| e^{\frac{\alpha z^2}4} D_7(\sigma) \Big\|_{2,0}\,\mathrm{d}\sigma.
\end{align*}
By estimate~\eqref{eq:estD21} and the integral inequality~\eqref{INT}, we have
\begin{equation}\label{eq:Y2Again}
Y_2
\ls \left[\int_0^S e^{-\gamma (S-\sigma)} \beta^{\frac{33}{20}}\,\mathrm{d}\sigma\right]
    P(M(T),A(T))
\ls \beta^{\frac{33}{20}} (T) P (M(T),A(T)).
\end{equation}
By similar reasoning, we get
\begin{equation}
Y_3 \ls \beta^{\frac85} (T)M_{3,0} (T).
\end{equation}
Estimate~\eqref{eq:compareXiEta} and the slow decay of $\beta$ imply that
\begin{equation}\label{eq:Y1Again}
Y_1 \ls e^{-\gamma S} \|\phi(\cdot,0)\|_{2,1}
\ls \beta^{\frac85} (T) (T) M_{2,1}(0).
\end{equation}

Collecting estimates~\eqref{eq:estEtaAgain}--\eqref{eq:Y1Again}, we conclude that
$$\beta^{-\frac85} (T)
\left\| e^{\frac{\alpha z^2}4} P^\alpha_2 \Big(\partial_z + \frac{\alpha}2 z \Big) \eta (S)
\right\|_{2,0}
\ls M_{2,1} (0) + M_{3,0} (T) + \beta^{ \frac1{20} } (0) P (M(T),A(T)).$$
Then we use equation~\eqref{eq:AgreeEnd} to rewrite the \textsc{lhs}, yielding
$$\| e^{\frac{\alpha z^2}4} P^\alpha_2 (\partial_z + \frac{\alpha}2 z ) \eta (S)] \|_{2,0}
= \| \phi(T)\|_{2,1}.$$
Because $T$ was arbitrary, estimate~\eqref{eq:M21} follows by recalling the
definition of $M_{2,1}$.

This completes our proof of Proposition~\ref{Prop:Main3} and hence of Theorem~\ref{THM:aprior}.

\section{Proof of the Main Theorem}\label{RepeatMainProof}
The proof is almost identical to that in \cite{GS09}. We provide a sketch below.

\begin{proof}[Proof of Main Theorem]
For $b_0,c_0$ small enough, our Main Assumptions (see Section~\ref{Basic}) imply that there exists
$\tau_1>0$ such that the the necessary conditions --- including the hypotheses of Proposition~\ref{IFT} ---
for starting both bootstrap machines (see Sections~\ref{FBMI} and \ref{SBMI})  are satisfied for
$\tau(t)\in[0,\tau_1]$. The outputs of both machines then show that the Main Assumptions hold at $\tau_1$,
with smaller constants. So one can iterate this procedure to a larger interval $[0,\tau_2]$.

We now return to the original time scale $t$. The iteration above produces
$t^* \le T$, where $T\leq\infty$ denotes the maximal existence time of the solution,
and a function $\lambda(t)$, admissible on $[0, t^*)$, such that Theorems~\ref{FirstOrderEstimates},
\ref{InnerEstimates}, \ref{OuterEstimates}, \ref{SmallnessEstimates}, and \ref{THM:aprior} hold
on $[0, t^*)$. Observe that if the iteration cannot be continued past $t^*$, then either $t^* = T$
or $\lambda(t^*) = 0$. Hence either
\textsc{(i)} $t^* < T$ but $\lambda(t^*) = 0$, or
\textsc{(ii)} $t^*=T=\infty$, or
\textsc{(iii)} $t^*=T<\infty$.

We first claim that case~\textsc{(i)} cannot occur. Indeed, if $\lambda(t^*) = 0$,
then it follows from Theorem~\ref{THM:aprior} that
\[
0<u(0,\theta,t)\leq \lambda(t)
    \left[\sqrt{\frac{2}{1+C\beta(\tau(t))}}+C \beta^\frac{8}{5}(\tau(t))\right]\rightarrow 0
\]
as $t\nearrow t^*$, which implies that $T\leq t^*$.

We next show that case~\textsc{(ii)} cannot occur. Indeed, by Theorem~\ref{THM:aprior},
one has
\begin{equation}\label{EndTimes}
a(t)-\frac{1}{2}=- b(t)+\mathcal{O}(\beta^{2}(\tau(t)))\quad\text{and}\quad
b(t)=\beta(\tau(t))[1+\mathcal{O}( \beta^{\frac{1}{2}}(\tau(t)))],
\end{equation}
where $\tau(t)=\int_0^t \lambda^{-2}(s)\,\mathrm{d}s$.
So $|a(t)-\frac{1}{2}|=\mathcal{O}(\beta(\tau))$.
To obtain a contradiction, suppose $T=\infty$.
Then because $\partial_t\lambda=-a/\lambda<0$, we have
$\tau\rightarrow\infty$ as $t\rightarrow\infty$, and hence
$|a(t)-\frac{1}{2}|\rightarrow 0$. Because
$\lambda^{2}(t)=\lambda_{0}^{2}-2 \int_{0}^{t}a(s)\,\mathrm{d}s$,
it follows that there exists $\tilde{T}<\infty$ with $\lambda(T)=0$.
Our proof for case~\textsc{(i)} forces $T\leq\tilde{T}<\infty$, proving the claim.

Therefore, case~(iii) is the only possibility. This means that there exists
$\lambda(t)$, admissible on $[0, T)$ with $T<\infty$, such that
Theorems~\ref{FirstOrderEstimates}, \ref{InnerEstimates}, \ref{OuterEstimates},
\ref{SmallnessEstimates}, and \ref{THM:aprior} hold for $t\in[0,T)$, with
$\lambda(t)\searrow 0$ as $t\nearrow T$.

We establish the asymptotic behaviors of $\lambda(t)$, $b(t)$, and
$c(t):=\frac12+a(t)$ as follows. Equation~\eqref{EndTimes} shows that
that $b(t)\rightarrow 0$ and $a(t)\rightarrow\frac{1}{2}$ as $t\nearrow T$.
Hence as $t\nearrow T$, one has
\[\lambda(t)=(1+o(1))\sqrt{T-t}\quad\text{and}\quad \tau(t)=(1+o(1))(-\log(T-t)),\]
as well as
\[\beta(\tau(t))=(1+o(1))(-\log(T-t))^{-1}.\]
This concludes the proof.
\end{proof}

\appendix

\section{Mean curvature flow of normal graphs} \label{Graphs}

Here we derive the basic formulas that describe \textsc{mcf} of a normal graph over
a cylinder. Although these equations are familiar to experts, we include the short
derivation for the convenience of the reader, and to make the paper self-contained.

Write points in $\mathbb{R}^{n+2}$ as $(w,x)=(w^{1},\ldots,w^{n+1},x)$. The
round cylinder $\mathbb{S}^{n}\times\mathbb{R}$ is naturally embedded as
$\mathcal{C}^{n+1}=\{(w,x):\left\vert w\right\vert ^{2}=1\}\subset\mathbb{R}^{n+2}$.
Let $\mathcal{M}_{t}$ denote a smooth family of smooth surfaces
determined by $u:\mathcal{C}^{n+1}\times\lbrack0,T)\rightarrow\mathbb{R}_{+}$,
namely
\[
\mathcal{M}_t =\left\{(u(w,x,t)w,x):(w,x)\in\mathcal{C}^{n+1},\,t\in\lbrack0,T)\right\}.
\]
For most of the derivation, we suppress time dependence. In what follows, we
sum over repeated indices and restrict Roman indices to the range $1\ldots n$.

For $(w,x)\in\mathcal{C}^{n+1}$, let
$(e_{1},\ldots,e_{n},e_{n+1}=\frac{\partial}{\partial x})$ denote an orthonormal
basis of $T_{(w,x)}\mathcal{C}^{n+1}\subset T_{(w,x)}\mathbb{R}^{n+2}$, with
$(e_{1},\ldots,e_{n})$ tangent to $\mathbb{S}^{n}\times\{x\}$. Let $D$ denote covariant
differentiation on $\mathcal{C}^{n+1}$, and let $\nabla$ denote covariant
differentiation on a round $\mathbb{S}^{n}$ of radius one. If
\[
F:(w,x)\mapsto(u(w,x)w,x),
\]
then, recalling that $1\leq i\leq n$, one has
\begin{align*}
D_i F &  =(\nabla_i u)w+ue_i ,\\
D_{n+1}F  &  =(\partial_x u)w+e_{n+1}.
\end{align*}
It follows that the Riemannian metric
$g$ induced on $\mathcal{M}$ has components
\begin{align*}
g_{ij}  &  =\langle D_i F,D_j F\rangle = u^2 \delta_{ij}+\nabla_i u \nabla_j u,\\
g_{i\,n+1}  &  =\langle D_i F,D_{n+1}F\rangle=\partial_x u \nabla_{i}u,\\
g_{n+1\,n+1}  &  =\langle D_{n+1}F,D_{n+1}F\rangle=1+(\partial_x u)^2.
\end{align*}
Note that in this Appendix (but nowhere else in this paper) the symbol
$\langle\cdot,\cdot\rangle$ denotes the pointwise Euclidean inner product
of $\mathbb{R}^{n+2}$.

An outward normal to $\mathcal{M}$ is
\[
N=uw-\nabla_{i}ue_i -u(\partial_x u)e_{n+1}.
\]
Noting that
\begin{equation}
|N|^{2}=[1+(\partial_x u)^2]u^{2}+|\nabla u|^{2}, \label{Norm-N}
\end{equation}
we denote the outward unit normal by $\nu=|N|^{-1}N.$
Straightforward calculations show that
\begin{align*}
|N|D_{j}\nu & =2(\nabla_j u)w+ue_j -\nabla_j \nabla u
    -(\partial_x u \nabla_j u+u\nabla_{j}\partial_x u)e_{n+1}+\{\cdots\}\nu,\\
|N|D_{n+1}\nu &  =(\partial_x u)w-\nabla \partial_x u
    -[(\partial_x u)^2+u\partial_x^2 u]e_{n+1}+\{\cdots\}\nu,
\end{align*}
where the terms in braces are easy to compute but irrelevant for what follows.
One thus finds that the components of the second fundamental form $h$ are
\begin{align*}
h_{ij}  &=\langle D_{i}F,D_{j}\nu\rangle=|N|^{-1}(g_{ij}+\nabla_{i}u\nabla_{j}u-u\nabla_{i}\nabla_{j}u),\\
h_{i\,n+1}  &  =\langle D_{i}F,D_{n+1}\nu\rangle=|N|^{-1}(g_{i\,n+1}-u\nabla_{i}\partial_x u),\\
h_{n+1\,n+1}  &  =\langle D_{n+1}F,D_{n+1}\nu\rangle=|N|^{-1}\left(g_{n+1\,n+1}-[(\partial_x u)^2+u\partial_x^2 u]\right).
\end{align*}

The standard trick to compute the mean curvature $H=\operatorname{tr}_{g}(h)$
of $\mathcal{M}$ is as follows. Fix any point of the graph and rotate the
local orthonormal frame $(e_{1},\ldots,e_{n})$ so that all $\nabla_{i}u$
vanish except perhaps $\nabla_{n}u$. In these coordinates, one may identify
$g$ and $h$ with block matrices
\[
g=
\begin{pmatrix}
u^2 I_{n-1} & 0\\
0 & P_2
\end{pmatrix}
\quad\text{and}\quad
h=\frac{1}{|N|}
\begin{pmatrix}
u^2 I_{n-1}-u\nabla\nabla u & 0\\
0 & Q_2
\end{pmatrix}
\]
respectively, where
\[
P_2=
\begin{pmatrix}
u^{2}+|\nabla u|^{2} & \partial_x u \nabla_{n}u\\
\partial_x u \nabla_{n}u & 1+(\partial_x u)^2
\end{pmatrix}
\]
and
\[
Q_2=
\begin{pmatrix}
u^{2}+2|\nabla u|^{2}-u\nabla_{n}\nabla_{n}u & \partial_x u \nabla_{n}u-u\nabla_{n}\partial_x u\\
\partial_x u \nabla_{n}u-u\nabla_{n}\partial_x u & -u\partial_x^2 u
\end{pmatrix}
.
\]
Computing $\mathrm{tr}(g^{-1}h)$ and recasting the resulting terms in
invariant notation, one obtains the coordinate-independent formula
\begin{equation} \label{NH}
H=\frac{n|N|^2-u
    \left\{Q_{ij}\nabla_{i}\nabla_{j}u+(u^{2}
    +|\nabla u|^{2})\partial_x^2 u
    -2\partial_x u\langle\nabla u,\nabla \partial_x u\rangle
    -u^{-1}|\nabla u|^{2}\right\}}{|N|^3},
\end{equation}
where $Q$ is the quasilinear elliptic operator on $\mathbb{S}^{n}$ defined below
in equation~\eqref{Q-general}. \medskip

Now we reintroduce time. It is a standard fact that \textsc{mcf} of
$\mathcal{M}_t$ is equivalent (modulo time-dependent reparameterization by
tangential diffeomorphisms) to the evolution equation
\[
\frac{u^{2}}{|N|^{2}}\partial_t u=-H\langle\nu,w\rangle=-\frac{H\langle N,w\rangle}{|N|}=-\frac{Hu}{|N|},
\]
namely to $\partial_t u=-|N|Hu^{-1}$. By equations~\eqref{Norm-N} and
\eqref{NH}, this is the quasilinear parabolic \textsc{pde}
\begin{equation} \label{MCF-general}
\partial_t u
    = \frac{Q_{ij}\nabla_{i}\nabla_{j}u
            +(u^{2}+|\nabla u|^2)\partial_x^2 u
            -2\partial_x u\langle\nabla u,\nabla \partial_x u\rangle
            - u^{-1}|\nabla u|^{2}}
    {|\nabla u|^{2}+[1+(\partial_x u)^2]u^2} - \frac{n}{u},
\end{equation}
where $Q$ is the elliptic operator on $\mathbb{S}^{n}$
with coefficients
\begin{equation} \label{Q-general}
Q_{ij}=[1+u^{-2}|\nabla u|^{2}+(\partial_x u)^2]\delta_{ij}-u^{-2}\nabla_i u \nabla_j u.
\end{equation}

\section{Interpolation estimates}\label{InterpolationInequalities}
Here, we state and prove various interpolation inequalities used in this paper.
We begin with an elementary embedding result.

\begin{lemma}\label{LM:embedding}
Let $g: [0,2\pi]\rightarrow \mathbb{R}$ be a smooth periodic function satisfying the condition
$\int_0^{2\pi}g(\theta)\,\mathrm{d}\theta=0$. Then for any $n\in \mathbb{Z}^{+}$, the following
inequalities hold:
\[
\displaystyle\max_{\theta\in [0,2\pi]} |g(\theta)|
    \leq \frac{5}{2}
    \left[\frac{1}{2\pi} \int_{0}^{2\pi}
        |\dz^{n}g|^2\,\mathrm{d}\theta\right]^{\frac{1}{2}}
    \leq \frac{5}{2}\max_{\theta\in [0,2\pi]} |\dz^{n}g|.
\]
\end{lemma}
\begin{proof} We provide a detailed proof for $n=1$; the cases $n\geq 2$ are similar.

Recall that the Fourier decomposition
$g(\theta)=\sum_{m=-\infty}^{\infty}a_m e^{in\theta}$
and the Plancherel equality
$\frac{1}{2\pi}\int_{0}^{2\pi} |\partial_{\theta}g|^2\,\mathrm{d}\theta=\sum_{m=-\infty}^{\infty}m^2 |a_m|^2$
hold for determined complex coefficients $a_m$. The key fact that $g\perp 1$ forces $a_0=0$.
Using H\"older and weighted Cauchy--Schwarz, this fact allows us to obtain
\[
\displaystyle\max_{\theta\in [0,2\pi]} |g(\theta)|
    \leq \sum_{m=-\infty}^{\infty}|a_m|
    \leq\frac{1}{2}\left(\ve\sum_{m\not=0} |m|^{-2}+\ve^{-1}\sum_{m=-\infty}^{\infty}m^2|a_m|^2\right)
\]
for any $\ve>0$. If $0<\ve^2:=\sum_{m=-\infty}^{\infty}m^2|a_m|^2$, this becomes
\[\displaystyle\max_{\theta\in [0,2\pi]} |g(\theta)|
    \leq\frac{\ve}{2}\left(\sum_{m\not=0} |m|^{-2}+1\right)
    =\frac{\ve}{2}\left(\frac{\pi^2}{3}+1\right)
    <\frac{5}{2}\ve.\]
The result follows.
\end{proof}

The main result of this appendix is the following easy corollary.

\begin{lemma}\label{InterpEst}
Let $v$ satisfy Assumption~[Cr], so that $|v_2|\leq\delta v_1$
with respect to the decomposition $v=v_1+v_2$ given in
definition~\eqref{Decomposition}. Then for any $k\geq 0$, $m\geq 0$,
and $n\geq 1$, one has
\[
v^{-k}|\dy^{m}\dz^{n}v|\ls
    \left[\frac{1}{2\pi}\int_{0}^{2\pi}v^{-2k}|\dy^{m}\dz^{n+1}v|^2\,
        \mathrm{d}\theta\right]^{\frac{1}{2}}
    \ls \max_{\theta\in[0,2\pi]}v^{-k}|\dy^{m}\dz^{n+1}v|.
\]
\end{lemma}

\begin{proof}
By Lemma~\ref{LM:embedding}, one has
\[
v_1^{-k}|\dy^{m}\dz^{n}v|\ls
    v_1^{-k}\left[\frac{1}{2\pi}\int_{0}^{2\pi}|\dy^{m}\dz^{n+1}v|^2\,\mathrm{d}\theta\right]
        ^{\frac{1}{2}}
    \leq v_1^{-k}\max_{\theta\in[0,2\pi]}|\dy^{m}\dz^{n+1}v|.
\]
The result follows from this and the fact that
$(1-\delta)v_1\leq v \leq(1+\delta)v_1$.
\end{proof}

\section{A parabolic maximum principle for noncompact domains}\label{ProvePMP}

Parabolic maximum principles do not in general hold in noncompact domains without a growth restriction,
even for the linear heat equation in Euclidean space. Here we extend Lemma~7.1 in \cite{GS09}, adding
$\theta$-dependence and allowing exponential growth at infinity. This maximum principle is designed
to apply to quantities $(v^{-k}\dy^m\dz^n v)^\ell$ in a time-dependent region $\{\beta(\tau) y^2\geq c\}$.
(See Remark~\ref{WLOG} below.)

\begin{proposition}\label{PMP}
Given $T>0$ and a continuous function $b(\tau)$, define
\[
\Omega:=\{(y,\theta,\tau)\in\mathbb{R}\times\mathbb{S}^1\times(0,T] : |y|\geq b(\tau)\}.
\]
Let
\[D:=\dt-(A_1\dy^2+A_2\dz^2+2A_3\dy\dz+A_4\dz+A_5 y\dy+B)\]
be a differential operator with continuous coefficients in $\bar\Omega$.

\textsc{(o)} Suppose that $A_1$, $A_5$, and $B$ are uniformly bounded from above
in $\Omega$, and the operator $D$ is parabolic in the sense that $A_1 A_2>0$ and
$A_1 A_2\geq A_3^2$.

\textsc{(i)} Suppose that $u\in C^{2,1}(\Omega)\cap C^0(\bar\Omega)$ is a $D$-subsolution,
namely that
\begin{equation}\label{PMP-1}
(\dt-A_1\dy^2-A_2\dz^2-2A_3\dy\dz-A_4\dz-A_5 y\dy-B)u\leq 0\quad\text{in}\quad\Omega.
\end{equation}

\textsc{(ii)} Suppose that $u$ satisfies the parabolic boundary conditions
\begin{equation}\label{PMP-2}
u(y,\theta,0)\leq 0\,\text{ if }\,|y|\geq b(0);\qquad
u(y,\theta,\tau)\leq 0\,\text{ if }\,|y|=b(\tau)\,\text{ and }\,\tau\leq T.
\end{equation}

\textsc{(iii)} Suppose that in $\Omega$,
\begin{equation}\label{PMP-3}
|u| \ls e^{\alpha y^2}\quad\text{ for some }\,\alpha>0.
\end{equation}

Then $u\leq 0$ in $\Omega$.
\end{proposition}
\begin{proof}
Because $B$ is bounded above, we may w.l.o.g.~assume that $B\leq 0$. Otherwise, for $C>0$
sufficiently large, it suffices to prove the result for $\tilde u:= e^{-C\tau}u$, which
satisfies $(\dt-A_1\dy^2-A_2\dz^2-2A_3\dy\dz-A_4\dz-A_5 y\dy-(B-C))\tilde u\leq 0$.

For $0<\gamma<1/4$ to be chosen, define
\[h(y,\tau):=(T_0-\tau)^{-1/2}e^{\gamma y^2(T_0-\tau)^{-1}},\]
where $T_0\leq T$ is also to be chosen below, such that $0<T_0<\gamma/\alpha$.
For $\ve>0$, define \[w(y,\theta,\tau):=u(y,\theta,\tau)-\ve[h(y,\tau)+\tau].\]

Let $T_1:=\frac{9}{10}T_0$. We claim that $w\leq0$ for $\tau\in[0,T_1]$.
To prove the claim, note that $w(\cdot,\cdot,0)\leq-\ve T_0^{-1/2}<0$
and $w(y,\theta,\tau)\rightarrow-\infty$ as $|y|\rightarrow\infty$,
uniformly for $\tau\in[0,T_1]$. So if the claim is false, there is a first
time $\tau\in(0,T_1]$ and a point $(y,\theta)$ with $|y|>b(\tau)$,
where $0=w(y,\theta,\tau)=\max_{0\leq\tau'\leq\tau}w(\cdot,\cdot,\tau')$.
At $(y,\theta,\tau)$, one has
\begin{equation}\label{MaxPoint}
0   \leq \dt u -\ve(\dt h+1)
    \leq-\ve\left\{(\dt-A_1\dy^2-A_5 y\dy)h-B(h+\tau)+1\right\}.
\end{equation}
But
\[
(\dt-A_1\dy^2-A_5 y\dy)h =
\frac{h}{2(T_0-\tau)}
\left\{1-4\gamma A_1+ 2\gamma y^2\left[\frac{1-4\gamma A_1}{T_0-\tau}-A_5\right]\right\}.
\]
By condition~(\textsc{o}), one can choose $0<\gamma<(4A_1)^{-1}$. Then choose
$T_0$ small enough so that $A_5\leq(1-4\gamma A_1)T_0^{-1}$. These choices ensure
that $(\dt-A_1\dy^2-A_5 y\dy)h\geq 0$, whence inequality~\eqref{MaxPoint} implies that
$0\leq-\ve$, because we are assuming $B\leq0$. This contradiction proves the claim.

Now letting $\ve\searrow 0$ shows that $u\leq0$ for $\tau\in[0,T_1]$.
The theorem follows by repeating this argument on successive time intervals,
$[T_1,\min\{2T_1,T\}]$, etc.
\end{proof}

\begin{remark}\label{WLOG}
In our applications, we only need to verify conditions (\textsc{i})--(\textsc{iii}),
i.e.~\eqref{PMP-1}--\eqref{PMP-3}, in Proposition~\ref{PMP}, along with an upper
bound for $B$. Indeed, parabolicity and an upper bound for $A_1$ follow easily from
definition~\eqref{Define-Fi}. Condition~[Ca] bounds $a$ from above. Condition~[C0] implies
upper bounds for $v^{-2}$. Hence condition~(\textsc{o}) will always be satisfied by
the operators in \eqref{MCF-v} and \eqref{Evolve-vmnk2}.
\end{remark}

\section{Estimates of higher-order derivatives}\label{GetHigh}

Recall that we defined $v_{m,n,k}$ in equation~\eqref{Define-vmnk} and
$E_{m,n,n,\ell}$ in definition~\eqref{Define-Xmnkj}.
In this section, we estimate the nonlinear commutators (``error terms'')
\[E_{m,n}:=
\int_{\mathbb{R}}\int_{\mathbb{S}^1}
        \left(\sum_{\ell=0}^5 E_{m,n,n,\ell}\right)v_{m,n,n}\,\Dz\,\sigma\,\Dy\]
that appear in the evolution equations satisfied by the Lyapunov functionals
$\Omega_{m,n}$ considered in Section~\ref{InnerBootstrap} and defined
in equation~\eqref{Define-Omega-mn}. We treat the cases that $3\leq m+n\leq5$ here,
because the corresponding quantities were estimated for $m+n=2$ in the proofs of
Lemmas~\ref{v200-estimate}--\ref{v022-estimate} in Section~\ref{PointwiseSecond}.
In the proofs in this Appendix, we use the fact that Assumptions~[C0i]--[C1i]
allow us to assume that estimates~\eqref{v_y-est}--\eqref{v_z-est} of
Theorem~\ref{FirstOrderEstimates} hold globally.
\medskip

\textsc{Notation}:
For brevity, we suppress irrelevant coefficients in this section, writing $\es$
to denote an equality that holds up to computable and uniformly bounded
constant coefficients. For example, \[(a+b)^2\es a^2+ab+b^2\ls a^2+b^2.\]

\begin{lemma}\label{E21estimate}
If the first-order inequalities~\eqref{v_y-est}--\eqref{v_z-est}
and Conditions~[C2]--[C3] hold, then there
exist constants $0<C<\infty$ and $\rho=\rho(\Sigma)>0$,
with $\rho(\Sigma)\searrow0$ as $\Sigma\rightarrow\infty$, such that for
all $m+n=3$ with $n\geq1$, one has
\[E_{m,n}\leq
\rho\beta^2\left[\Omega_{m,n}^{\frac{1}{2}}
    +\beta\Omega_{4,0}^{\frac{1}{2}}
    +\Omega_{3,1}^{\frac{1}{2}}
    +\Omega_{2,2}^{\frac{1}{2}}
    +\Omega_{1,3}^{\frac{1}{2}}
    +\Omega_{0,4}^{\frac{1}{2}}\right]
    +C\beta^{\frac{3}{5}}
    \left(\Omega_{2,1}+\Omega_{1,2}+\Omega_{0,3}\right).\]
\end{lemma}

\begin{proof}
We provide a detailed proof for $E_{2,1}$.

It is easy to see that $|\lp E_{2,1,1,0},v_{2,1,1}\rp|\ls\beta^{\frac{3}{5}}\Omega_{2,1}$.
In the rest of the proof, we will frequently use
estimates~\eqref{dyF_ell-estimate}--\eqref{dz2F_ell-estimate}, namely
\[|\dy F_\ell|\ls\beta^{\frac{3}{5}},\quad\quad
v^{-1}|\dz F_\ell|\ls\beta^{\frac{3}{2}}\]
and
\begin{align*}
|\dy^2 F_\ell|&\ls\beta^{\frac{6}{5}}+|\V 30| + |\V 21|\ls\beta,\\
v^{-1}|\dy\dz F_\ell|&\ls\beta^2+|\V 21|+|\V 12|\ls\beta^{\frac{3}{2}},\\
v^{-2}|\dz^2F_\ell|&\ls\beta^3+|\V 12|+|\V 03|\ls\beta^{\frac{3}{2}},
\end{align*}
which hold under the hypotheses in place here.
Direct computation establishes that
\begin{align*}
\lp E_{2,1,1,1},v_{2,1,1}\rp
    \es &\,\lp v^{-1}\dy^2\dz F_1,\, \V 21\V 20 \rp\\
    &+\lp \dy^2 F_1,\, {\V 21}^2\rp
    +\lp v^{-1}\dy\dz F_1,\, \V 30 \V 21 \rp\\
    &+\lp \dy F_1,\, \V 31 \V 21 \rp
    +\lp v^{-1}\dz F_1,\, \V 40 \V 21 \rp\\
    &+\lp F_1,\,\V 31 \V 21 \V 10 +{\V 21}^2[\V 20 +{\V 10}^2]\rp.
\end{align*}
Because we have stronger estimates for $\dy\dz F_1$ than we do for $\dy^2 F_1$,
we transform the first inner product above via integration by parts in $y$, yielding
\begin{align*}
\lp v^{-1}\dy^2\dz F_1, \V 21 \V 20 \rp
    \es&\,
    \lp v^{-1}\dy\dz F_1,\,\V 31 \V 20 +\V 30 \V 21 \rp\\
    &+\lp v^{-1}\dy\dz F_1,\,\V 21 \V 20 [\V 10 +\Sigma^{-\frac{1}{2}}]\rp,
\end{align*}
where the final term comes from \eqref{gamma-decay}.
Because the measure of $(\mathbb{S}^1\times\mathbb{R};\Dz\,\sigma\,\Dy)$ is finite,
it then follows easily from the hypotheses of the lemma that
\[ |\lp E_{2,1,1,1},v_{2,1,1}\rp|\leq\rho\left[
    \beta^3\Omega_{4,0}^{\frac{1}{2}}
    +\beta^2\Omega_{3,1}^{\frac{1}{2}}
    +\beta^{\frac{5}{2}}\Omega_{2,1}^{\frac{1}{2}}\right]
    +C\beta^{\frac{3}{5}}(\Omega_{2,1}+\Omega_{1,2}).\]
Integrating $\lp v^{-1}\dy^2\dz F_2,\,\V 21 \V 02 \rp$ by parts in $\theta$, one
obtains
\begin{align*}
\lp E_{2,1,1,2},v_{2,1,1}\rp\es&\,
    \lp \dy^2F_2,\,\V 22 \V 02 \rp\\
    &+\lp \dy^2F_2,\,\V 21 [\V 03 + \V 02 \V 01 ]\rp\\
    &+\lp v^{-1}\dy\dz F_2,\, \V 21 [\V 12 + \V 02 \V 10]\rp\\
    &+\lp \dy F_2,\,\V 21
        [\V 13 +\V 12 \V 01 +\V 03 \V 10 ]\rp\\
    &+\lp \dy F_2,\,\V 21 \V 02
        [\V 11 +\V 10 \V 01 ] \rp\\
    &+\lp v^{-1}\dz F_2,\,\V 21
        [\V 22 +\V 12 \V 10 ] \rp\\
    &+\lp v^{-1}\dz F_2,\,\V 21 \V 02
        [\V 20 +{\V 10 }^2 ] \rp\\
    &+\lp F_2,\,
        \V 21 [\V 13 \V 10
        +\V 21 {\V 01}^2]\rp\\
    &+\lp F_2,\,
        \V 21 \V 12 [\V 11 +\V 10 \V 01 ]\rp\\
    &+\lp F_2,\,
        \V 21 \V 03 [\V 20 +{\V 10}^2]\rp\\
    &+\lp F_2,\,
        \V 21 \V 20 \V 02 \V 01 \rp\\
    &+\lp F_2,\,
        \V 21 \V 02 [\V 11 \V 10 +{\V 10}^2\V 01 ]\rp,
\end{align*}
which shows that
\[ |\lp E_{2,1,1,2},v_{2,1,1}\rp|\leq\rho\left[
    \beta^{\frac{5}{2}}\Omega_{2,2}^{\frac{1}{2}}+\beta^2\Omega_{1,3}^{\frac{1}{2}}
    +\beta^{\frac{7}{2}}\Omega_{2,1}^{\frac{1}{2}}\right]
    +C\beta(\Omega_{2,1}+\Omega_{1,2}+\Omega_{0,3}).\]
An similar computation, again integrating by parts in $\theta$, proves that
\[ |\lp E_{2,1,1,3},v_{2,1,1}\rp|\leq\rho\left[
    \beta^3\Omega_{3,1}^{\frac{1}{2}}+\beta^{\frac{21}{10}}\Omega_{2,2}^{\frac{1}{2}}
    +\beta^{\frac{7}{2}}\Omega_{2,1}^{\frac{1}{2}}\right]
    +C\beta(\Omega_{2,1}+\Omega_{1,2}).\]
Integrating by parts in $\theta$ yet again, one calculates that
\begin{align*}
\lp E_{2,1,1,4},v_{2,1,1}\rp\es&\,
    \lp\dy^2F_4,\,\V 22 \V 01 +\V21 [\V 02 +{\V 01}^2 ]\rp\\
    +&\lp v^{-1}\dy\dz F_4,\,\V21 [\V 11 +\V 10 \V 01 ]\rp\\
    +&\lp \dy F_4,\,\V 21 [\V 12 +\V 11 \V 10 +\V 02 \V 10 ]\rp\\
    +&\lp v^{-1}\dz F_4,\,\V 21 [\V 21 +\V 20 \V 01 ]\rp\\
    +&\lp v^{-1}\dz F_4,\,\V 21 [\V 11 \V 10 +{\V 10}^2\V 01 ]\rp\\
    +&\lp F_4,\,  \V 21 \V 20 [\V 02 +{\V 01}^2]\rp\\
    +&\lp F_4 ,\, \V 21 \V 11 [\V 11 +\V 10 \V 01 ]\rp,
\end{align*}
yielding
\[ |\lp E_{2,1,1,4},v_{2,1,1}\rp|\leq\rho\beta^{\frac{5}{2}}
    \left[\Omega_{2,2}^{\frac{1}{2}}
    +\Omega_{2,1}^{\frac{1}{2}}\right]
    +C\beta(\Omega_{2,1}+\Omega_{1,2}).\]
Finally, one computes
\[E_{2,1,1,5}\es \V 20 \V 01 +\V11 \V 10 +{\V 10 }^2 \V 01 ,\]
yielding
\[|\lp E_{2,1,1,5},v_{2,1,1}\rp|\leq\rho\beta^2\Omega_{2,1}^{\frac{1}{2}}.\]
The estimate for $E_{2,1}$ follows when one combines the estimates above.
(Note that in this case, the term $\Omega_{0,4}^{\frac{1}{2}}$ is not needed.)
\medskip

The proof for $E_{1,2}$ is entirely analogous,
except that it suffices here to integrate the leading term
$\lp v^{-2}\dy\dz^2 F_1,\V 12 \V 20 \rp$
by parts in $\theta$. A typical subsequent term, also obtained
by integrating by parts in $\theta$, is
\begin{align*}
\lp E_{1,2,2,2},v_{1,2,2}\rp\es&\,
    \lp v^{-1}\dy\dz F_2,\,\V 13 \V 02 \rp\\
    &+\lp v^{-1}\dy\dz F_2,\,\V 12 [\V 03 + \V 02 \V 01 ]\rp\\
    &+\lp v^{-2}\dz^2 F_2,\, \V 12 [\V 12 + \V 02 \V 10]\rp\\
    &+\lp \dy F_2,\,\V 12
        [\V 04 +\V 03 \V 01  ]\rp\\
    &+\lp \dy F_2,\,\V 12
        [{\V 02}^2 +\V 02 {\V 01}^2 ]\rp\\
    &+\lp v^{-1}\dz F_2,\,\V 12
        [\V 13 +\V 12 \V 01 +\V 03 \V 10] \rp\\
    &+\lp v^{-1}\dz F_2,\,\V 12 \V 02
        [\V 11 +\V 10 \V 01 ] \rp\\
    &+\lp F_2,\,
        \V 04 \V 12 \V 10
        +{\V 12}^2 [\V 02 +{\V 01}^2 ]\rp\\
    &+\lp F_2,\,
        \V 12 \V 03 [\V 11 +\V 10 \V 01 ]\rp\\
    &+\lp F_2,\,
        \V 12 \V 02 [\V 11 \V 01 +\V 02 \V 10 ]\rp.
\end{align*}
The weakest estimate for $E_{1,2}$ again comes from its final term,
\[|\lp E_{1,2,2,5},v_{1,2,2}\rp|\leq\rho\beta^2\Omega_{1,2}^{\frac{1}{2}}.\]

One estimates $E_{0,3}$ in exactly the same fashion.
We omit further details.
\end{proof}

Our final third-order quantity satisfies a slightly different estimate.

\begin{lemma}\label{E30estimate}
If the first-order inequalities~\eqref{v_y-est}--\eqref{v_z-est}
and Conditions~[C2]--[C3] hold, then there
exist constants $0<C<\infty$ and $\rho=\rho(\Sigma)>0$,
with $\rho(\Sigma)\searrow0$ as $\Sigma\rightarrow\infty$, such that
\[E_{3,0}\leq
\rho\beta^{\frac{3}{2}}\left[
    \Omega_{3,0}^{\frac{1}{2}}
    +\beta^{\frac{1}{10}}\Omega_{4,0}^{\frac{1}{2}}
    +\Omega_{3,1}^{\frac{1}{2}}
    +\Omega_{2,2}^{\frac{1}{2}}\right]
    +C\beta\left(\sum_{i+j=3}\Omega_{i,j}\right)
    +C\beta^{\frac{1}{2}}\Omega_{3,0}^{\frac{1}{2}}\Omega_{2,0}^{\frac{1}{2}}.\]
\end{lemma}

\begin{proof}
Note that $E_{3,0,0,0}=0$. Integrating by parts in $y$, one calculates that
\begin{align*}
\lp E_{3,0,0,1},\V 30 \rp\es&\,
    \lp\dy^2F_1,\, \V 40 \V 20
        +\V 30 [\V 30+\Sigma^{-\frac{1}{2}}\V 20\rp\\
    &+\lp \dy F_1,\,\V 40 \V 30 \rp,
\end{align*}
where the $\Sigma^{-\frac{1}{2}}$ factor comes from \eqref{gamma-decay}.
Thus,
\[|\lp E_{3,0,0,1},\V 30 \rp|\leq
    \rho\left[\beta^{\frac{8}{5}}\Omega_{4,0}^{\frac{1}{2}}
    +\beta^{\frac{9}{5}}\Omega_{3,0}^{\frac{1}{2}}\right]
    +C\beta\Omega_{3,0}.\]
Again integrating by parts in $y$, one estimates the quantity
\begin{align*}
\lp E_{3,0,0,3},\V 30 \rp\es&\,
    \lp\dy^2F_3,\,\V 40\V 11\rp\\
    &+\lp\dy^2F_3,\,\V 30
        [\V 21+\V 11\V 10+\Sigma^{-\frac{1}{2}}\V 11]\rp\\
    &+\lp\dy F_3,\,\V 30[\V 31+\V 21\V 10]\rp\\
    &+\lp\dy F_3,\,\V 30[\V 20+{\V 10}^2]\V 11\rp\\
    &+\lp F_3,\,\V 31 \V 10
        +\V 21[\V 20+{\V 10}^2]\rp\\
    &+\lp F_3,\,[\V 30+\V 20\V 10 +{\V 10}^3]\V 11\rp
\end{align*}
by
\[|\lp E_{3,0,0,3},\V 30 \rp|\leq
    \rho\left[\beta^{\frac{5}{2}}\Omega_{4,0}^{\frac{1}{2}}
    +\beta^{\frac{3}{2}}\Omega_{3,1}^{\frac{1}{2}}
    +\beta^{\frac{13}{5}}\Omega_{3,0}^{\frac{1}{2}}\right]
    +C\beta\left(\Omega_{3,0}+\Omega_{2,1}\right).\]
The terms $\lp E_{3,0,0,2},\V 30 \rp$ and $\lp E_{3,0,0,4},\V 30 \rp$
have similar expansions and obey similar bounds; we omit the details.
The critical final term,
\[\lp E_{3,0,0,5},\V 30\rp\es\lp\V 20\V 10+{\V 10}^3,\V 30\rp,\]
is estimated by
\[|\lp E_{3,0,0,5},\V 30\rp|\leq
    C\beta^{\frac{1}{2}}\Omega_{3,0}^{\frac{1}{2}}\Omega_{2,0}^{\frac{1}{2}}
    +\rho\beta^{\frac{3}{2}}\Omega_{3,0}^{\frac{1}{2}}.\]
(Note that we deal with the term
$C\beta^{\frac{1}{2}}\Omega_{3,0}^{\frac{1}{2}}\Omega_{2,0}^{\frac{1}{2}}$
in the proof of Proposition~\ref{SharpLyapunov}.)
\end{proof}
\bigskip

The next three lemmas prepare us to estimate $E_{m,n}$ for $4\leq m+n\leq 5$.

\begin{lemma}\label{pq}
If the first-order inequalities~\eqref{v_y-est}--\eqref{v_z-est} and
Conditions~[C2]--[C3] hold, then one may estimate derivatives of $p:=\dy v$ by
\[\begin{array}{c}
|\dy p|\ls\beta^{\frac{3}{5}},\quad v^{-1}|\dz p|\ls\beta^{\frac{3}{2}},\\ \\
|\dy^2 p|\ls\beta,\quad v^{-1}|\dy\dz p|\ls\beta^{\frac{3}{2}},\quad v^{-2}|\dz^2 p|\ls\beta^{\frac{3}{2}},\\ \\
v^{-n}|\dy^m\dz^n p|\ls|(v_{m+1,n,n})|\quad\text{for}\quad 3\leq m+n\leq 4,
\end{array}\]
and of $q:=v^{-1}\dz v$ by
\begin{align*}
v^{-n}|\dy^m\dz^n q|&\ls\beta^{\frac{3}{2}}\quad\text{for}\quad 1\leq m+n\leq 2,\\ \\
v^{-n}|\dy^m\dz^n q|&\ls |v_{m,n+1,n+1}|+\beta^2\quad\text{for}\quad m+n=3,\\ \\
|\dy^4q|&\ls|\V 41|+\beta^{\frac{3}{2}}|\V 40|+\beta^{\frac{1}{2}}|\V 31|+\beta^{\frac{21}{10}},\\ \\
v^{-1}|\dy^3\dz q|&\ls|\V 32|+\beta^{\frac{3}{2}}|\V 31|+\beta^{\frac{1}{2}}|\V 22|+\beta^{\frac{21}{10}},\\ \\
v^{-2}|\dy^2\dz^2q|&\ls|\V 23|+\beta^{\frac{3}{2}}|\V 22|+\beta^{\frac{1}{2}}|\V 13|+\beta^{\frac{21}{10}},\\ \\
v^{-3}|\dy\dz^3q|&\ls|\V 14|+\beta^{\frac{3}{2}}|\V 13|+\beta^{\frac{1}{2}}|\V 04|+\beta^3,\\ \\
v^{-4}|\dz^4q|&\ls|\V 05|+\beta^{\frac{3}{2}}|\V 04|+\beta^3.
\end{align*}
\end{lemma}
\begin{proof}
The estimates for $v^{-n}|\dy^m\dz^n p|$ are easy and are left to the reader.

The estimates for $v^{-n}|\dy^m\dz^n q|$ are obtained by computing and verifying
all fourteen cases directly. For example, one has
\begin{align*}
v^{-2}\dy^2\dz^2 q\es&\,
    \V 23+\V 22\V 01+\V 13\V 10\\
    &+\V 21[\V 02+{\V 01}^2]
    +\V 12[\V 11+\V 10\V 01]\\
    &+\V 03[\V 20+{\V 10}^2]
    +\V 20[\V 02\V 01+{\V 01}^3]\\
    &+\V 11[\V 11\V 01+\V 02\V10 +\V 10{\V 01}^2]\\
    &+{\V 10}^2{\V 01}^3.
\end{align*}
The remaining calculations are similar and unenlightening, hence omitted.
\end{proof}

Our next results extend estimates~\eqref{dyF_ell-estimate}--\eqref{dz2F_ell-estimate}
to higher derivatives $v^{-n}|\dy^m\dz^n F_\ell|$.

\begin{lemma}\label{Fell3}
If the first-order inequalities~\eqref{v_y-est}--\eqref{v_z-est} and
Conditions~[C2]--[C3] hold, then third derivatives of the coefficients
$F_\ell$ introduced in definition~\eqref{Define-Fi} may be estimated by
\[
|\dy^3F_\ell|\ls|\V 40|+|\V 31|+\beta^{\frac{8}{5}},\]
and for $m+n=3$ with $n\geq1$, by
\[v^{-n}|\dy^m\dz^n F_\ell|\ls|(v_{m+1,n,n})|+|(v_{m,n+1,n+1})|+\beta^2.\]
\end{lemma}

\begin{proof}
After a wee bit of calculus, one computes, for example, that
\[|\dy^3F_\ell|\ls (|\dy p|+|\dy q|)^3
    +(|\dy p|+|\dy q|)(|\dy^2 p|+|\dy^2 q|)+|\dy^3 p|+|\dy^3 q|\]
and
\begin{align*}
|\dy^2\dz F_\ell|\ls&\,(|\dy p|+|\dy q|)^2(|\dz p|+|\dz q|)
    +(|\dy p|+|\dy q|)(|\dy\dz p|+|\dy\dz q|)\\
    &+(|\dz p|+|\dz q|)(|\dy^2 p|+|\dy^2 q|)
    +|\dy^2\dz p|+|\dy^2\dz q|,
\end{align*}
and then applies Lemma~\ref{pq}.
The remaining inequalities are derived similarly.
\end{proof}

\begin{lemma}\label{Fell4}
If the first-order inequalities~\eqref{v_y-est}--\eqref{v_z-est} and
Conditions~[C2]--[C3] hold, then fourth derivatives of the coefficients
$F_\ell$ introduced in definition~\eqref{Define-Fi} may be estimated by
\begin{align*}
|\dy^4F_\ell|&\ls|\V 50|+|\V 41|+\beta^{\frac{3}{5}}|\V 40|+\beta^{\frac{1}{2}}|\V 31|+\beta^2,\\ \\
v^{-1}|\dy^3\dz F_\ell|&\ls|\V 41|+|\V 32|+\beta^{\frac{3}{2}}|\V 40|
    +\beta^{\frac{3}{5}}|\V 31|+\beta^{\frac{1}{2}}|\V 22|+\beta^{\frac{21}{10}},\\ \\
v^{-2}|\dy^2\dz^2 F_\ell|&\ls|\V 32|+|\V 23|+\beta^{\frac{3}{2}}|\V 31|
    +\beta^{\frac{3}{5}}|\V 22|+\beta^{\frac{1}{2}}|\V 13|+\beta^{\frac{21}{10}},\\ \\
v^{-3}|\dy\dz^3 F_\ell|&\ls|\V 23|+|\V 14|+\beta^{\frac{3}{2}}|\V 22|
    +\beta^{\frac{3}{5}}|\V 13|+\beta^{\frac{1}{2}}|\V 40|+\beta^{\frac{13}{5}},\\ \\
v^{-4}|\dz^4F_\ell|&\ls|\V 14|+|\V 05|+\beta^{\frac{3}{2}}(|\V 13|+|\V 04|)+\beta^3.
\end{align*}
\end{lemma}
\begin{proof}
One proceeds as in Lemma~\ref{Fell3}, calculating, for example, that
\begin{align*}
|\dy^3\dz F_\ell|\ls&\,|\dy^3\dz p|+|\dy^3\dz q|\\
    &+(|\dy^3p|+|\dy^3q|)(|\dz p|+|\dz q|)\\
    &+(|\dy^2\dz p|+|\dy^2\dz q|)(|\dy p|+|\dy q|)\\
    &+(|\dy\dz p|+|\dy\dz q|)(|\dy^2p|+|\dy^2q|+|\dy p|^2+|\dy q|^2)\\
    &+(|\dy^2p|+|\dy^2q|)(|\dy p||\dz p|+|\dy q||\dz q|)\\
    &+(|\dy p|+|\dy q|)^3(|\dz p|+|\dz q|),
\end{align*}
and then applying Lemma~\ref{pq}. The other estimates are obtained
in like fashion.
\end{proof}
\medskip

We are now prepared to bound the ``error terms'' $E_{m,n}$ of orders four and five.
As indicated in Section~\ref{Lyapunov45}, our work below is considerably simplified
because these estimates do not need to be sharp, and because they can use
Proposition~\ref{SharpLyapunov}.

\begin{lemma}\label{E4estimate}
If the first-order inequalities~\eqref{v_y-est}--\eqref{v_z-est},
Conditions~[C2]--[C3], and $L^2_\sigma$ inequalities~\eqref{SharpLyapunovEstimates}
hold, then there exists $0<C<\infty$ such that for all $m+n=4$, one has
\[E_{m,n}\leq
C\beta^r\left(\sum_{4\leq i+j\leq 5}\Omega_{i,j}^{\frac{1}{2}}\right)
    +C\beta^{\frac{1}{2}}\left(\sum_{4\leq i+j\leq 5}\Omega_{i,j}\right),\]
where $r=\frac{8}{5}$ if $n=0$ and $r=2$ otherwise.
\end{lemma}

\begin{proof}
It is easy to see that $|\lp E_{m,n,n,0},v_{m,n,n}\rp|\ls\beta^{\frac{3}{5}}\Omega_{m,n}$.

In the remainder of the proof, we write $D^kF_\ell$ to denote any sum of terms
$v^{-j}\dy^i\dz^j F_\ell$ with $i+j=k$; and we write  $D^k v$ to denote any
sum of terms $v_{i,j,j}$ of total weight $i+j=k$. For example, $D^2v\es \V 20 +\V 11 +\V 02$.

For $\ell=1,2,3$, carefully adapting the proof of Lemma~\ref{E21estimate}
to the case that $m+n=4$ shows that after
integration by parts,\footnote{In contrast to the proof of Lemma~\ref{E21estimate}, the choice
of variable with which to integrate does not matter here unless $n=0$, because our estimates
do not need to be sharp.} one obtains an estimate of the following form,
\begin{align*}
|\lp E_{m,n,n,\ell},v_{m,n,n}\rp|\ls&\,
    \lp|D^3F_\ell|,
        |D^5v|\,|D^2v|+|v_{m,n,n}|(|D^3v|+\Sigma^{-\frac{1}{2}})\rp\\
    &+\lp|D^3F_\ell|,|v_{m,n,n}|
        (|D^3v|+|D^2v|\,|D^1v|)\rp\\
    &+\lp|D^2F_\ell|,|v_{m,n,n}|
        (|D^4v|+|D^3v|\,|D^1v|+|D^2v|^2)\rp\\
    &+\lp|D^1F_\ell|,|v_{m,n,n}|
        (|D^5v|+|D^4v|\,|D^1v|+|D^3v|\,|D^2v|)\rp\\
    &+\lp|v_{m,n,n}|,
        |D^5v|\,|D^1v|+|D^4v|\,|D^2v|+|D^3v|^2\rp.
\end{align*}
Examination of the terms above using our hypotheses and Lemma~\ref{Fell3}
shows that any terms which are quadratic in derivatives of orders four or
five are multiplied by factors whose decay is at least of order
$\beta^{\frac{1}{2}}$. Any derivatives of orders four or five that do not
occur in quadratic combinations appear with factors whose decay is at
least of order $\beta^r$, where $r=\frac{8}{5}$ if $n=0$ and $r=2$ otherwise.
A similar but simpler estimate holds for $|\lp E_{m,n,n,4},v_{m,n,n}\rp|$.

For the critical final term, one uses estimate~\eqref{SharpLyapunovEstimates}
to see that
\begin{align*}
|\lp E_{m,n,n,5},v_{m,n,n}\rp|\es&\,
    \lp |v_{m,n,n}|,|D^3v|\,|D^1v|+|D^2v|(|D^2v|+|D^1v|^2)+|D^1v|^4\rp\\
    &\ls C\beta^r \Omega_{m,n}^{\frac{1}{2}}.
\end{align*}
The result follows.
\end{proof}

\begin{lemma}\label{E5estimate}
If the first-order inequalities~\eqref{v_y-est}--\eqref{v_z-est},
Conditions~[C2]--[C3], and $L^2_\sigma$ inequalities~\eqref{SharpLyapunovEstimates}
hold, then there exists $0<C<\infty$ such that for all $m+n=5$, one has
\[E_{m,n}\leq
C\beta^{\frac{21}{10}}\left(\sum_{5\leq i+j\leq 6}\Omega_{i,j}^{\frac{1}{2}}\right)
    +C\beta^{\frac{1}{2}}\left(\sum_{4\leq i+j\leq 6}\Omega_{i,j}\right).\]
\end{lemma}

\begin{proof}
It is clear that $|\lp E_{m,n,n,0},v_{m,n,n}\rp|\ls\beta^{\frac{3}{5}}\Omega_{m,n}$.

As in the proof of Lemma~\ref{E4estimate}, we write $D^kF_\ell$ to denote any sum
of terms $v^{-j}\dy^i\dz^j F_\ell$ with $i+j=k$; and we write  $D^k v$ to denote
any sum of terms $v_{i,j,j}$ of total weight $i+j=k$.

For $\ell=1,2,3$, adapting the proof of Lemma~\ref{E21estimate} to the case that
$m+n=5$ shows that before integration by parts, one has
\begin{align*}
|\lp E_{m,n,n,\ell},v_{m,n,n}\rp|\ls&\,
    \lp|D^5F_\ell|,|v_{m,n,n}||D^2v|\rp\\
    &+\lp|D^4F_\ell|,|v_{m,n,n}|
        (|D^3v|+|D^2v|\,|D^1v|)\rp\\
    &+\lp|D^3F_\ell|,|v_{m,n,n}|
        (|D^4v|+|D^3v|\,|D^1v|+|D^2v|^2)\rp\\
    &+\lp|D^2F_\ell|,|v_{m,n,n}|
        (|D^5v|+|D^4v|\,|D^1v|+|D^3v|\,|D^2v|)\rp\\
    &+\lp|D^1F_\ell|,|v_{m,n,n}|
        (|D^6v|+|D^5v|\,|D^1v|+|D^4v|\,|D^2v|+|D^3v|^2)\rp\\
    &+\lp|v_{m,n,n}|,
    (|D^6v|\,|D^1v|+|D^5v|\,|D^2v|+|D^4v|\,|D^3v|)\rp.
\end{align*}
Integrating by parts proves that the first inner product on the \textsc{rhs}
above may be bounded by
\[\lp|D^5F_\ell|,|v_{m,n,n}||D^2v|\rp\ls
    \lp|D^4F_\ell|,|D^6v|\,|D^2v|+|v_{m,n,n}|(|D^3v|+\Sigma^{-\frac{1}{2}})\rp.\]
Integrating by parts is also necessary to show that the third inner product on the \textsc{rhs}
above may be bounded by
\begin{align*}
\lp|D^3F_\ell|,\,&|v_{m,n,n}|(|D^4v|+|D^3v|\,|D^1v|+|D^2v|^2)\rp\\
    \ls&\,\lp|D^2F_\ell|,|D^6v|(|D^4v|+|D^3v|\,|D^1v|+|D^2v|^2)\rp\\
    &+\lp|D^2F_\ell|,|v_{m,n,n}|
        (|D^5v|+|D^4v|\,|D^1v|+|D^3v|\,|D^2v|)\rp\\
    &+\lp|D^2F_\ell|,|v_{m,n,n}|(|D^4v|+|D^3v|\,|D^1v|+|D^2v|^2)\Sigma^{-\frac{1}{2}}\rp.
\end{align*}
Examination of all the terms above using our hypotheses and Lemmas~\ref{Fell3}--\ref{Fell4}
proves that any terms that are quadratic in derivatives of orders four, five, or six
are multiplied by factors whose decay is at least of order $\beta^{\frac{1}{2}}$.
Any derivatives of orders five or six not occuring in quadratic combinations
appear with factors whose decay is at least of order $\beta^{\frac{11}{5}}$.
(Derivatives of order four only occur in quadratic combinations.)
A similar but simpler estimate holds for $|\lp E_{m,n,n,4},v_{m,n,n}\rp|$.

For the critical final term, one uses estimate~\eqref{SharpLyapunovEstimates}
to see that
\begin{align*}
|\lp E_{m,n,n,5},v_{m,n,n}\rp|\es&\,
    \lp |v_{m,n,n}|,|D^4v|\,|D^1v|+|D^3v|(|D^2v|+|D^1v|^2)\rp\\
    &+\lp |v_{m,n,n}|,|D^2v|(|D^2v|\,|D^1v|+|D^1v|^3)+|D^1v|^5\rp\\
    \ls&\,C\beta^{\frac{21}{10}} \Omega_{m,n}^{\frac{1}{2}}
    +C\beta^{\frac{1}{2}}\left(\sum_{4\leq i+j\leq 5}\Omega_{i,j}\right)
\end{align*}
The result follows.
\end{proof}

\end{document}